%% file: spectral_weight_filtrations.tex
\pdfoutput=1 
\documentclass{amsart}
\usepackage[margin=1.25in]{geometry} 

\include{preamble}


\begin{document}

\title{Spectral weight filtrations}

\author{Peter J. Haine}

\author{Piotr Pstr\k{a}gowski}

\begin{abstract}
	We provide a description of Voevodsky's \category of motivic spectra in terms of the subcategory of motives of smooth proper varieties. 
	As applications, we construct weight filtrations on the Betti and étale cohomologies of algebraic varieties with coefficients in any complex oriented ring spectrum.
	We show that these filtrations satisfy \ldhdescent, giving an effective way of calculating them in positive characteristic. 
	In the complex motivic case, we further refine the weight filtration to one defined at the level of stable homotopy types. 
\end{abstract}

\thanks{The first-named author gratefully acknowledges support from the NSF Mathematical Sciences Postdoctoral Research Fellowship under Grant \#DMS-2102957 and a grant from the Simons Foundation, 816048 LC. 
The second-named gratefully acknowledges support from NSF Grant \#DMS-1926686 and Deutsche Forschungsgemeinschaft \#EXC-2047/1 – 390685813}

\maketitle
\setcounter{tocdepth}{2}
\tableofcontents


\section{Introduction}\label{sec:introduction}


\subsection{Motivation and overview} 

Let $ X $ be a complex variety.
In his fundamental series of papers \cites{deligne1970theorie}{deligne1971theorie}{deligne1974theorie}, Deligne explains how to use the algebraic structure of $ X $ to endow the rational singular cohomology $\Hrm^{*}(X(\CC); \QQ)$ with a canonical \emph{weight filtration} 
\begin{equation*}
	\Wup_{0}\Hrm^{*}(X(\CC); \QQ) \subseteq \Wup_{1}\Hrm^{*}(X(\CC); \QQ) \subseteq \cdots \period
\end{equation*}
Moreover, the complexification 
\begin{equation*}
	\CC \otimes_{\QQ} \Wup_{\bullet}\Hrm^{*}(X(\CC); \QQ)
\end{equation*}
has a canonical \textit{mixed Hodge module} structure on its associated graded pieces. 
In fact, the filtration exists \textit{before} passing to cohomology: Deligne shows that the singular cochain complex $ \Cup^{*}(X(\CC); \QQ) $ can be canonically refined to an object of the filtered derived \category. 
The weight filtration contains crucial algebraic information: it is not an invariant of the topological space $ X(\CC) $. 
Informally, the weight filtration is obtained by resolving $ X $ by smooth proper varieties. 

The weight filtration on rational cohomology has been extended to a variety of contexts. 
In \cite{gillet1996descent}, Gillet and Soulé show that the weight filtration can be refined to a canonical filtration on the complex of compactly supported \emph{integral} cochains $ \Cc^{*}(X(\CC);\ZZ)$, which was later vastly generalized in the work of Bondarko on weight structures \cites{MR2461902}{MR2746283}{MR2834727}. 
In this paper, one of our main results is that the weight filtration is defined at a \emph{spectral level}, even before passing to algebra. 
That is, we show that the weight filtration can be refined to a canonical filtration on the stable homotopy type of $ X(\CC) $ which equips the latter with a structure of a \textit{synthetic spectrum}.

Our result is based on a new description of Voevodsky's stable \category of motivic spectra $ \SH(\CC) $ in terms of the subcategory generated by motives of smooth proper varieties.
More generally, given any base field $ k $ of exponential characteristic $ e $, we give a new description of the \category $ \SH(k)[\einv] $ obtained from $ \SH(k) $ by inverting the exponential characteristic. This allows for a clean construction of filtered refinements of both Betti and étale realization with coefficients in a complex orientable cohomology theory. 

Applying these filtered realizations to various motivic spectra one can attach to a variety, we obtain weight filtrations on the (co)homology of varieties. 
In particular, we are able to construct weight filtrations on étale cohomology with coefficients in a complex orientable étale sheaf of spectra, extending Deligne's weight filtration on $ \ell $-adic étale cohomology \cite{deligne1980conjecture}. We note that our construction is manifestly functorial, and so the weight filtrations we construct care naturally compatible with the action of cohomology operations. 

We also show that the induced filtration on Borel--Moore homology satisfies hyperdescent with respect to Kelly's \ldhtopology \cite{kelly2017voevodsky}. 
Combined with the theory of alterations \cites[Theorem 4.4]{MR3730515}[Theorem 1.1]{Illusie:On_Gabbers_refined_uniformization}[Exposé IX, Théorème 1.1]{MR3309086}[Theorem 1.2.5]{MR3665001}, this gives an effective way of calculating this filtration in positive characteristic. 
We end the paper with a conjectural picture of the existence of a synthetic realization in the étale context. 

In the rest of this introduction, we explain our results in more detail.


\subsection{The complex orientable case} 

We first describe our result in its most basic case, over the complex numbers and in the case of a \textit{complex orientable} cohomology theory (such as complex bordism, complex $\Kup$-theory, or ordinary cohomology). 
We make use of Voevodsky's \category of motivic spectra $ \SH(\CC) $, and we assume that the reader is familiar with basics of motivic homotopy theory; see \cref{sec:recollections_on_motivic_homotopy_theory} or a brief review.

Let $A \in \CAlg(\spectra)$ be a commutative algebra in spectra.
We have the \textit{$ A $-linear Betti realization} functor  
\begin{equation*}
	\Be(-;A) \colon \SH(\CC) \to \Mod_{A}
\end{equation*}
which is the unique symmetric monoidal left adjoint such that for any smooth $ \CC $-scheme $ X $, we have
\begin{equation*}
	\Be(\Sigma_{+}^{\infty} X;A) \simeq A \otimes \Sigma_{+}^{\infty} X(\CC) \period
\end{equation*}
That is, $ \Sigma_{+}^{\infty} X \in \SH(\CC) $ is sent to the the $ A $-linear stable homotopy type of $ X(\CC) $. 

The functor $\Be(-; A)$ encodes the theory of Betti (co)homology of varieties. 
In more detail, it is a left adjoint, so any $ A $-module $ M $ determines through the right adjoint to $\Be(-; A)$ a motivic spectrum over $\CC$.
Through the formalism of six-functors of the stable motivic category, in turn any motivic spectrum determines (co)homology theories on varieties, in both ordinary and compactly supported variants, which in this case recovers Betti (co)homology with coefficients in $ M $. 

Our first result is that if $ A $ is complex orientable, then the $ A $-linear Betti realization can be equipped with a canonical filtration.
Recall that a \textit{filtered spectrum} is a functor $ X_{*} \colon \ZZ^{\op} \to \spectra $, where we regard $ \ZZ $ as a poset with the usual ordering. 
We write
\begin{equation*}
	\FilSp \colonequals \Fun(\ZZ^{\op}, \spectra)
\end{equation*}
for the \category of filtered spectra.
Every spectrum $ X $ has a canonical \textit{Postnikov filtration}
\begin{equation*}
	\cdots \to \tau_{\geq 1} X \to \tau_{\geq 0}X \to \tau_{\geq -1} X \to \cdots 
\end{equation*}
which can be naturally refined to a lax symmetric monoidal functor $\tau_{\geq *} \colon \spectra \to \FilSp$. 

\begin{theorem}[{\ref{cor:homological_filtered_Betti_realization_exists_for_complex_orientable_rings}}]
	\label{theorem:introduction_complex_orientable_weight_filtration}
	Let $A \in \CAlg(\spectra)$ be complex orientable.
	Then, there exists a unique colimit-preserving lax symmetric monoidal functor 
	\begin{equation*}
		\WBe(-;A) \colon \SH(\CC) \to \Modpost{A}
	\end{equation*}
	such that on the subcategory of motivic spectra of the form $S \simeq (\PP^{1})^{\otimes n} \otimes \Sigma_{+}^{\infty} Y$ with $n \in \ZZ$ and $ Y $ a smooth proper complex variety, we have a natural equivalence
	\begin{equation*}
		\WBe(S;A) \simeq \tau_{\geq *}(\Be(S;A)) \period
	\end{equation*}
	We refer to $ \WBe(-;A) $ as the \emph{filtered $ A $-linear Betti realization} functor.
\end{theorem}

Note that if $ A $ is an ordinary commutative ring, we have an identification
\begin{equation*}
	\Modpost{A} \simeq \Dcat^{\fil}(A)
\end{equation*}
with the classical filtered derived \category of $ A $, obtained by localizing filtered chain complexes at filtered quasi-isomorphisms.
See \Cref{prop:filtered_modules_ordinary_ring}. 

Informally, \cref{theorem:introduction_complex_orientable_weight_filtration} says that once we decide to equip the $ A $-homology of each smooth proper variety $ X $ with the ``trivial filtration'' given by the Postnikov tower, there is a unique way to extend this to a colimit-preserving functor defined on all of $ \SH(\CC) $. 
By construction, for any motivic spectrum $ S $, the canonical map from the colimit 
\begin{equation*}
	\varinjlim \WBe(S;A) \to \Be(S;A)
\end{equation*}
is an equivalence.
This induces a filtration on homology groups of $\Be(S; A)$; hence for any complex variety $ X $, we obtain a filtration on the complex oriented (co)homology of $ X(\CC) $. 

As a sample application, we explain how to use \Cref{theorem:introduction_complex_orientable_weight_filtration} to define virtual Euler characteristics with coefficients in Morava $ \Kup $-theories.
This description does not rely on Bittner's presentation of the Grothendieck ring of varieties \cite{bittner2004universal}, and is adaptable to more general base fields.
See \cref{subsec:virtual_Euler_characteristics}.


\subsection{A new description of motivic spectra}
\label{subsection:introduction:new_description_of_motivic_spectra}

Our proof of \cref{theorem:introduction_complex_orientable_weight_filtration} is based on the following description of the stable motivic category away from the characteristic.
Our description is inspired by the work of Bachmann--Kong--Wang--Xu on the \textit{Chow--Novikov \tstructure} on motivic spectra \cite{bachmann2022chow}. 

Let $ k $ be a field of exponential characteristic $ e $. 
We say that a motivic spectrum $ S \in \SH(k)[\einv] $ over $ k $ is \emph{perfect pure} if $ S $ belongs to the smallest subcategory 
\begin{equation*}
	\Pure(k) \subseteq \SH(k)[\einv] 
\end{equation*}
generated under extensions and retracts by motivic Thom spectra $ \Th(\eta) $, where $ \eta \in \Kup_{0}(X)$ and $ X $ is a smooth proper $ k $-variety. 
An \emph{additive sheaf} $ \Fcal \colon \Pure(k)^{\op} \to \spectra$ is a functor that sends cofiber sequences of perfect pure motivic spectra to fiber sequences of spectra; we denote the \category of additive sheaves of spectra on $ \Pure(k) $ by $\ShSigma(\Pure(k); \spectra)$%
\footnote{In \cref{section:motivic_spectra_as_sheaves_on_pure_motives}, we show that a spectral presheaf $\Fcal \colon \Pure(k)^{\op} \to \spectra$ sends preserves cofiber sequences if and only if it is additive and a sheaf with respect to a certain natural Grothendieck topology on $\Pure(k)$.
This justifies our terminology.}.

\begin{theorem}[{\ref{theorem:shk_is_additive_sheaves_on_perfect_pure_motives}}]
	\label{theorem:introduction_description_of_the_motivic_category_of_sheaves}
	Let $ k $ be a field of exponential characteristic $ e $.
	The spectral Yoneda embedding $S \mapsto \map_{\SH(k)[\einv]}(-, S)$ defines an equivalence of \categories
	\begin{equation*}
		\SH(k)[\einv] \equivalence \ShSigma(\Pure(k); \spectra) \period
	\end{equation*}
\end{theorem}

\begin{remark}[inverting $ e $]
	As usual, the reason \Cref{theorem:introduction_description_of_the_motivic_category_of_sheaves} requires inverting the exponential characteristic $ e $ ultimately relies on the fact that strong resolution of singularities is not known over general base fields; instead, we use Gabber's $ \ell' $-alteration theorem.
	Our proofs are written in such a way that if one assumes strong resolution of singularities over $ k $, then the refinement of \Cref{theorem:introduction_description_of_the_motivic_category_of_sheaves} without $ e $ inverted holds.
\end{remark}

By construction, the equivalence of \cref{theorem:introduction_description_of_the_motivic_category_of_sheaves} is compatible with the \tstructure recently introduced by Bachmann--Kong--Wang--Xu \cite{bachmann2022chow}.
More precisely, the Chow--Novikov \tstructure on $ \SH(k)[\einv] $ is identified with the canonical \tstructure on additive sheaves induced by the standard \tstructure on spectra. 

Let $ \MGL \in \SH(k) $ denote the motivic spectrum representing algebraic cobordism. 
If we replace $ \SH(k) $ with the \category of $\MGL[\einv]$-modules, \cref{theorem:introduction_description_of_the_motivic_category_of_sheaves} implies that there is an equivalence of \categories 
\begin{equation}
	\label{equation:mgl_modules}
	\Mod_{\MGL[\einv]}(\SH(k)) \simeq \PSigma(\Pure_\MGL(k); \spectra)[\einv]
\end{equation}
with additive spectral \emph{presheaves}. 
Here, 
\begin{equation*}
	\Pure_{\MGL}(k) \subseteq \Mod_{\MGL}(\SH(k))
\end{equation*}
is the subcategory of modules of the form $\MGL \otimes \Sigma_{+}^{\infty} X$, where $ X $ is smooth and proper. 
As explained in the work of Elmanto--Sosnilo \cite[\S2.2.11]{elmanto2022nilpotent}, this equivalence is also a consequence of the existence of Bondarko's weight structure on $ \MGL $-modules \cite{bondarko2018constructing}. 
Note that if we replace $ \MGL $ with the motivic cohomology spectrum $ \MZZ $, the equivalence \eqref{equation:mgl_modules} can be thought of as a homotopy-coherent refinement of the weight homology construction of Kelly--Saito \cite[Theorem 2.3]{kelly2017weight}. 

As an immediate consequence of \cref{theorem:introduction_description_of_the_motivic_category_of_sheaves}, we deduce the following new universal property of $ \SH(k)[\einv] $.

\begin{corollary}\label{theorem:introduction_universal_property_of_sh_in_terms_of_perfect_pures}
	Let $ k $ be a field of exponential characteristic $ e $ and let $ \Ccat $ be a cocomplete stable \category.
	Then restriction along the inclusion defines an equivalence of \categories
	\begin{equation*}
		\Fun^{\colim}(\SH(k)[\einv],\Ccal) \to \Fun^{\cofib}(\Pure(k),\Ccal)
	\end{equation*}
	between colimit-preserving functors $ \SH(k)[\einv] \to \Ccal $ and functors $ \Pure(k) \to \Ccal $ that preserve cofiber sequences.
\end{corollary}

\noindent The utility of \cref{theorem:introduction_universal_property_of_sh_in_terms_of_perfect_pures} comes down to the fact that cofiber sequences
\begin{equation*}
	\begin{tikzcd}[sep=1.5em]
		A \arrow[r] & B \arrow[r] & C \arrow[r, "\partial"] & \Sigma A
	\end{tikzcd}
\end{equation*}
in $\Pure(k)$ are easier to control than cofiber sequences of arbitrary motivic spectra. 
Indeed, since the $ \MGL[\einv] $-homology of a smooth proper $ k $-scheme vanishes in negative Chow degree \cite[Proposition 3.6(2)]{bachmann2022chow}, the boundary map 
\begin{equation*}
    \partial \colon (\MGL \otimes C)[\einv] \to (\MGL \otimes \Sigma A)[\einv]
\end{equation*}
is necessarily zero; see \Cref{proposition:pure_epimorphisms_detected_by_mgl}.
This fact is essentially equivalent to the existence of Bondarko's weight structure on $ \MGL[\einv] $-modules. 
It follows that any additive functor which preserves $\MGL[\einv]$-split cofiber sequences also preserves cofiber sequences of perfect pure motives. 
This implies \cref{theorem:introduction_complex_orientable_weight_filtration}: since any complex orientable $A \in \CAlg(\spectra)$ is module over $ \Be(\MGL) \simeq \MU$ in the homotopy category of spectra and Betti realization is symmetric monoidal, the functor
\begin{equation*}
    S \mapsto \tau_{\geq *} \Be(S;A)
\end{equation*}
preserves $ \MGL $-split cofiber sequences. 


\subsection{Filtered étale realization}

Since our construction of the filtered Betti realization is based on properties of the \category of motivic spectra itself, rather than the target of a given realization, it also allows us to prove the existence of weight filtrations in other contexts.
For example, let $ k $ be a field and let $ \ell \neq \characteristic(k) $ be a prime.
Write 
\begin{equation*}
	\Re_{\ell} \colon \SH(k) \to \Shethyp(\Et_k;\Sp)\ellcomp \comma
\end{equation*}
for the \textit{$ \ell $-adic étale realization} functor valued in  hypercomplete sheaves of $ \ell $-complete spectra on the small étale site of $ k $; see \cref{subsec:etale_realization}.
The target can be thought of as the \category of $ \ell $-complete spectra equipped with a continuous action of the absolute Galois group $ \Gal(\kbar/k) $. 
We are able to equip the étale realization of any motivic spectrum over $ k $ with a weight filtration:

\begin{theorem}[\ref{proposition:existence_of_the_a_linear_etale_realization}]
	\label{theorem:introduction_cohomological_etale_realization}
	Let $ k $ be a field of exponential characteristic $ e $ and let $ \ell \neq e $ be a prime.
	Let $A \in \CAlg(\Shethyp(\Et_k;\Sp)\ellcomp)$ be complex orientable in the sense that there exists a morphism $ \Re_{\ell}(\MGL) \to A $ of algebras in the homotopy category.
	There exists a unique colimit-preserving lax symmetric monoidal functor
	\begin{equation*}
		\WRe_{\ell}(- ; A) \colon \SH(k)[\einv] \to \Fil(\Shethyp(\Et_k;\Sp)\ellcomp)
	\end{equation*}
	valued in filtered hypersheaves such that for any $ S \in \Pure(k) $, we have
	\begin{equation*}
		\WRe_{\ell}(S; A) \simeq \tau_{\geq *} (\Re_{\ell}(S;A)) \period
	\end{equation*}
\end{theorem}


\subsection{Descent and the Gillet--Soulé filtration}

Let $ p \colon X \to \Spec(\CC) $ be a complex variety.
Then $ X $ determines a motivic spectrum
\begin{equation*}
	\motive(X) \colonequals p_{!}(\Unit_{X}) \in \SH(\CC)
\end{equation*}
that encodes the compactly supported cohomology of $ X $; see \cref{subsection:motives_of_varieties}.
In \cref{corollary:motives_of_varieties_are_dualizable_away_from_the_characteristic}, we show that this motivic spectrum is dualizable.
Thus, if $ A $ is complex orientable, then by applying filtered Betti realization to $ \Mc(X) $ and its dual $ \Mc(X)^{\vee} $, we obtain filtered spectra
\begin{equation*}
	\WBe(\motive(X); A) \andeq \WBe(\motive(X)^{\vee}; A) \period
\end{equation*}
These filtered spectra provide filtrations on the compactly supported $ A $-cohomology and Borel--Moore $ A $-homology of $ X $, respectively. 
Analogously, applying the filtered étale realization of \cref{theorem:introduction_cohomological_etale_realization} we obtain filtrations on $ \ell $-adic étale (co)homology over an arbitrary field. 

Since the weight filtrations considered in this paper are defined using a somewhat abstract characterization of the stable motivic category, it is natural to ask for an explicit way to calculate these filtrations only using varieties. 
In both the works of Deligne \cites{deligne1970theorie}{deligne1971theorie}{deligne1974theorie} and Gillet--Soulé \cite{MR1409056}, the weight filtration is obtained by repeatedly invoking resolution of singularities to resolve the starting variety by smooth projective varieties.
We show that the same method can be used in our context. 

Since we are also interested in the case of étale cohomology over fields of positive characteristic (where resolution of singularities is not known) we work with Kelly's \textit{\ldhtopology} \cite{kelly2017voevodsky}. 
Recall that the \ldhtopology is generated by the \cdhtopology and finite flat and surjective maps of degree prime to $ \ell $; see \cref{subsec:background_on_the_cdh-topology_and_ldh_topology} for a brief review.
By Gabber's $ \ell' $-alteration theorem \cite[Exposé IX, Théorème 1.1]{MR3309086}, for any field $ k $ and prime $ \ell \neq \characteristic(k) $, every $ k $-variety admits an \ldhhypercover by regular $ k $-varieties. 
Also note that since any \cdhcover is an \ldhcover, so the latter notion is strictly more general than classical resolution of singularities. 

\begin{theorem}[{\ref{theorem:hypercovers_of_varieties_are_colimit_diagrams_in_shk}}]
	\label{theorem:introduction_ldh_descent_for_dual_motive}
	Let $ k $ be a field and $ \ell \neq \characteristic(k) $ a prime.
	If $X_{\bullet} \to X$ is an \ldhhypercover of $ k $-schemes, then the canonical map 
	\begin{equation*}
		\varinjlim_{\Deltaop} \motive(X_{\bullet} )^{\vee}_{(\ell)} \to \motive(X)^{\vee}_{(\ell)} \period
	\end{equation*}
	is an $ \MGL $-local equivalence; that is, it becomes an equivalence after tensoring with $ \MGL $. 
	In particular, it is $ \infty $-connective with respect to the Chow--Novikov \tstructure. 
\end{theorem}

As our filtered realization functors have coefficients in a complex oriented homology theory, they invert $ \MGL $-local maps.
Let us now explain how \cref{theorem:introduction_ldh_descent_for_dual_motive} gives an effective way of calculating the filtration on Borel--Moore homology.
To treat both the Betti and étale cases uniformly, for a variety $ X $ and $ A \in \CAlg(\Sp) $ complex orientable, we write 
\begin{equation*}
		\Crm_{*}^{\BM}(X; A) \colonequals  
		\begin{cases}
			\Be(\motive(X)^{\vee}; A) & \text{(Betti)} \\
		  	\Re_{\ell}(\motive(X)_{(\ell)}^{\vee}; A) & \text{(étale)} \period
		\end{cases}
\end{equation*}
Informally, these are the $ A $-linear Borel--Moore ``cochains'', although note that in the étale case it is a hypersheaf of spectra on the étale site of $ k $ rather than a spectrum itself. 
Using \cref{theorem:introduction_complex_orientable_weight_filtration} and \cref{theorem:introduction_cohomological_etale_realization} these objects inherit canonical filtrations. 

\begin{theorem}[{\ref{theorem:weights_on_bm_homology_of_proper_variety_calculated_through_an_ldh_hypercover}}]
	\label{theorem:introduction_weights_on_bm_homology_of_proper_variety_calculated_through_an_ldh_hypercover}
	Let $ k $ be a field and let $ \ell \neq \characteristic(k) $ be a prime.
	Let $ X $ be a proper $ k $-scheme and let $X_{\bullet} \to X$ be an \ldhhypercover such that for each $i \geq 0$, the scheme $X_{i}$ is smooth and projective. 
	Then for any $ \ell $-local $ A $ we have 
	\begin{equation}
		\label{equation:calculating_weights_of_borel_moore_cochains}
		\Wup_{*} \Crm_{*}^{\BM}(X; A) \simeq \colim_{[i] \in \Deltaop} \tau_{\geq *} \Crm_{*}^{\BM}(X_{i}; A)
	\end{equation}
	where the colimit is calculated in filtered $\tau_{\geq *}A$-modules. 
	If $X_{\bullet} \to X$ is a \cdhcover, then \eqref{equation:calculating_weights_of_borel_moore_cochains} holds for any $ A $ in which the exponential characteristic of $ k $ is invertible. 
\end{theorem}

Note that the case of cohomology is more involved: although $ \MGL $-locally the motivic spectrum $\motive(X)$ can be written as a totalization of its hypercover, the filtered realization functors need not preserve infinite limits. 
We analyze this situation in more detail in the case of classical integral cohomology of complex varieties, where we prove that the necessary limit can be replaced by a finite one. 
As a consequence, we deduce the comparison result with the Gillet--Soulé filtration introduced in \cite{gillet1996descent}.
Given a complex variety, we write $ \Wup^{\GS}_{*} \Cc^{*}(X(\CC); \ZZ) $ for the Gillet--Soulé weight filtration on the compactly supported integral cochains on $ X(\CC) $.

\begin{theorem}[{\ref{theorem:gillet_soule_filtration_comparison}}]
	\label{theorem:introduction_gillet_soule_filtration_comparison}
	Let $ X $ be a complex variety. 
	Then there exists a natural equivalence
	\begin{equation}
		\label{equation:equivalence_of_filtrations_in_comparison_with_gs_filtration}
		\Wup_{*} \Cc^{*}(X(\CC); \ZZ) \simeq \Wup^{\GS}_{*} \Cc^{*}(X(\CC); \ZZ)
	\end{equation}
	of objects of the filtered derived \category of $\ZZ$. 
	In other words, the filtration on compactly supported integral cochains inherited from the filtered Betti realization coincides with the Gillet--Soulé filtration. 
\end{theorem}

For a comparison with more geometric flavor, we note that on the case of smooth open varieties, the weight filtration on compactly supported deRham cohomology introduced in the current work agrees with the décalage of Deligne's pole order filtration on logarithmic deRham cohomology. This is shown in \cite{annala2025note}, where the methods are of the current work are also extended to case of $p$-adic cohomology theories. 

\begin{remark}
	In the case of a field of characteristic zero, an alternative way to construct filtrations on complex oriented, compactly supported cohomology appears in the recent work of Kuijper \cite{kuijper2022general}.
	The filtrations constructed in this way also agree with the ones introduced in this paper, see \cref{remark:work_of_kuijper}.
\end{remark}


\subsection{Synthetic Betti realization}\label{intro_subsec:synthetic_Betti_realization}

In the case of the complex Betti realization we now describe how the weight filtration can be lifted to a filtration on the stable homotopy type itself. 
We believe that an analogous construction should yield a similar filtration in the real Betti and étale cases, and we sketch the conjectural picture in \cref{subsection:ideas_on_synthetic_real_and_etale_realizations}.

The monoidal unit $\sphere \in \spectra$ of spectra is not complex orientable.
However, the unit map $\sphere \to \MU$ is faithfully flat and induces a cosimplicial resolution 
\begin{equation*}
	\begin{tikzcd}
	    \sphere \arrow[r] & \MU \arrow[r, shift left=0.75ex] \arrow[r, shift right=0.75ex]  & \MU \tensor \MU  \arrow[r] \arrow[r, shift left=1.5ex] \arrow[r, shift right=1.5ex] &  \cdots  
	\end{tikzcd}
\end{equation*}
through complex orientable ring spectra.
Moreover, by the work of Hahn--Raksit--Wilson on the \textit{even filtration}%
\footnote{To be more precise, the $ \MU $-resolution of the sphere is universal as a resolution of the sphere through commutative ring spectra with even homotopy groups, i.e., an \textit{even} ring spectrum. 
However, any even ring spectrum is complex orientable, and any complex orientable spectrum can be made in an $ \MU $-algebra in the homotopy category, so we blur the distinction here.} 
\cite{hahn2022motivic}, this resolution is essentially universal with respect to this property. 
The limit of the associated diagram 
\begin{equation*}
	\begin{tikzcd}
	    \Modpost{\MU} \arrow[r, shift left=0.75ex] \arrow[r, shift right=0.75ex]  & \Modpost{\MU \tensor \MU}  \arrow[r] \arrow[r, shift left=1.5ex] \arrow[r, shift right=1.5ex] &  \cdots  
	\end{tikzcd}
\end{equation*}
of \categories of filtered modules can be thought of as a natural target of a weight filtration functor.
Even better, up to completion it can be identified with the \category $ \SynMU $ of \textit{$ \MU $-based synthetic spectra} introduced by the second-named author in \cite{pstrkagowski2022synthetic}. 

The \category $ \SynMU $ is best understood as \acategorical deformation encoding chromatic homotopy theory. 
It is a symmetric monoidal stable \category and monoidal unit has a canonical (degree-shifting) endomorphism $ \tau $.
This endomorphism $ \tau $ should be thought of as a formal parameter, and we have equivalences 
\begin{equation*}
	\SynMU^{\tau = 1} \simeq \spectra
\end{equation*}
between the generic fiber and spectra, and
\begin{equation}
	\label{equation:introduction_special_fiber_of_syn_is_indcoh_mfg}
	\SynMU^{\tau = 0} \simeq \IndCoh(\Mfg)
\end{equation}
between the special fiber and $\Ind$-coherent sheaves on the moduli stack of formal groups%
\footnote{In this paper, we mostly work with all (that is, not necessarily even) synthetic spectra, so that the right-hand side of \eqref{equation:introduction_special_fiber_of_syn_is_indcoh_mfg} is given by sheaves on the moduli of formal groups in Dirac geometry of Lars Hesselholt and the second-named author, see \cite[\S 5.2]{diracgeometry2}. 
It is a natural enlargement of $ \Ind $-coherent sheaves on the classical moduli stack where the Lie algebra line bundle $\omega$ has a canonical square root $\omega^{\otimes \nicefrac{1}{2}}$.}. 
There is a canonical fully faithful embedding $\nu \colon \spectra \hookrightarrow \SynMU$ which reduces to the identity of spectra on the generic fiber and to the association
\begin{equation*}
	X \mapsto \MU_{*}(X) \in \IndCoh(\Mfg)^{\heartsuit}
\end{equation*}
on the special fiber. 
For any spectrum $ X $, the $ \tau $-adic filtration on $\nu(X)$ encodes the Adams--Novikov spectral sequence calculating the stable homotopy groups $\pi_{*}(X)$. 

By the work of Gheorghe--Krause--Isaksen--Ricka \cite{MR4432907}, synthetic spectra are equivalent to filtered modules over the sphere spectrum equipped with the filtration
\begin{equation*}
	\fil^*(\Sup^{0}) \colonequals \lim_{[n] \in \DDelta} \tau_{\geq *}(\MU^{\otimes n+1})
\end{equation*}
given by descent along the faithfully flat map $ \Sup^{0} \to \MU $.
This is essentially the filtration on the sphere spectrum known as the \textit{Adams--Novikov filtration}%
\footnote{To be more precise, \cite{MR4432907} describes the subcategory of \emph{even} synthetic spectra as modules in $ \FilSp $ over the double-speed filtration $ \filev^*(\Sup^{0}) \colonequals \lim_{n} \tau_{\geq 2*}(\MU^{\otimes n+1})$.
The filtration $ \filev^*(\Sup^{0}) $ is what is typically refereed to as the Adams--Novikov filtration.
However, one can also describe the whole \category $ \SynMU $ as modules in $ \FilSp $ over $\fil^*(\Sup^{0}) $.
This is analogous to the difference between the even filtration and its half-integer version, see \cite[Remark 2.26]{pstrkagowski2023perfect}.}. 
Thus, the following realizes the promised weight filtration at the level of stable homotopy types:

\begin{theorem}[\ref{theorem:existence_of_complex_synthetic_betti_homology}]
	\label{intro_theorem:existence_of_complex_synthetic_betti_homology}
	There exists a unique lax symmetric monoidal left adjoint 
	\begin{equation*}
		\Besyn \colon \SH(\CC) \to \SynMU
	\end{equation*}
	such that for each $ S \in \Pure(\CC) $, we have
    \begin{equation*}
    	\Besyn(S) \simeq \nu(\Be(S)) \period
    \end{equation*}
\end{theorem}

The functor $ \Besyn $ is not strongly symmetic monoidal. 
To see this, note that the reduction to the special fiber $\SynMU \to \SynMU^{\tau = 0}$ is strongly symmetric monoidal. 
By construction, when restricted to synthetic spectra of the form $\Besyn(X)$ for $X \in \Pure(k)$, this reduction takes the form 
\begin{equation*}
    X \mapsto \MU_{*}(X(\CC)) \period
\end{equation*}
This functor is only lax symmetric monoidal: since $\MU_{*}$ is not a field, the Künneth map 
\begin{equation*}
    \MU_{*}(U) \tensorlimits_{\MU_{*}} \MU_{*}(V) \to \MU_{*}(U \otimes V)
\end{equation*}
is not generally an isomorphism. 
For the same reason, unless $A_{*}$ is a field, the $ A $-linear weight filtrations of \cref{theorem:introduction_complex_orientable_weight_filtration} are only lax symmetric monoidal.

The functor \cref{intro_theorem:existence_of_complex_synthetic_betti_homology} is weakly universal in the sense that if $ A $ is a complex orientable ring spectrum, there is a \textit{realization functor}
\begin{equation*}
	\nu(A) \tensor_{\nu(\sphere)} (-) \colon \SynMU \to \Modpost{A} \period
\end{equation*}
and a canonical natural transformation 
\begin{equation}
	\label{equation:introduction_comparison_of_synthetic_and_cpx_oriented_weight_filtrations}
	\nu(A) \tensor_{\nu(\sphere)} \Besyn(-) \to \WBe(-;A)
\end{equation}
of functors 
\begin{equation*}
	\SH(k) \to \Modpost{A} \period
\end{equation*}
We say only ``weakly universal'', because, due to the failure of the Künneth formula, the natural transformation \eqref{equation:introduction_comparison_of_synthetic_and_cpx_oriented_weight_filtrations} is \textit{not} generally an equivalence. 

In \Cref{thm:synthetic_realization_refines_filtered_realization_for_Landwebder_exact_rings}, we show that if the map $\Spec(A_{*}) \to \Mfg$ classifying the Quillen formal group is flat, then \eqref{equation:introduction_comparison_of_synthetic_and_cpx_oriented_weight_filtrations} \emph{is} an equivalence.
In particular, this is the case when $A = \QQ$; hence our synthetic weight filtration refines Deligne's rational weight filtration. 
Similarly, for any ring map $A \to B$ between complex orientable algebras in spectra, there is a comparison natural transformation 
\begin{equation*}
	\tau_{\geq *}(B) \tensorlimits_{\tau_{\geq *}(A)} \WBe(-;A) \to \WBe(-;B) \period
\end{equation*}
If $A_{*} \to B_{*}$ is flat, then this map is an equivalence; see \Cref{cor:Be_fil_of_flat_maps}. 

Since the synthetic refinement of the weight filtration provided by \cref{intro_theorem:existence_of_complex_synthetic_betti_homology} in particular encodes the stable homotopy type of the Betti realization, it is a much stronger invariant than the $\ZZ$-linear weight filtration.
As one piece of evidence towards its strength, observe that since the underlying homotopy type of any complex motivic sphere has only even cells, the synthetic weight filtration restricts to a functor 
\begin{equation*}
	\Besyn \colon \SH(\CC)^{\cell} \to \SynMUev
\end{equation*}
from the full subcategory spanned by \textit{cellular} motivic spectra into the full subcategory spanned by the even synthetic spectra.
This restriction was previously constructed by the second-named author in \cite[\S7.5]{pstrkagowski2022synthetic}.
There, it is shown that for any prime $ p $, this restriction becomes an equivalence
\begin{equation}
\label{equation:equivalence_between_cellular_motives_and_synthetic_spectra}
	(\SH(\CC)^{\cell})^{\wedge}_{p} \equivalence (\SynMUev)^{\wedge}_{p}
\end{equation}
after $ p $-completion \cite[Theorem 7.34]{MR4574661}. 
In other words, in the context of $ p $-complete cellular motivic spectra, the synthetic weight filtration is a complete invariant. 


\subsection{Relation to Bondarko's work on weight structures}
\label{intro_subsec:relation_to_Bondarko}

In this short subsection, we describe the relationship between our work and the theory of weight structures. 

In a celebrated series of papers \cites{MR2461902}{MR2746283}{MR2834727}, Bondarko introduced the notion of a \textit{weight structure} on a triangulated category. In a triangulated category $\mathcal{C}$ equipped with a weight structure, any object $X$ can be completed to a diagram 
\begin{equation*}
    \cdots \rightarrow X_{\w \geq 1}  \rightarrow  X_{\w \geq 0} \rightarrow X_{\w \geq -1} \rightarrow \cdots \rightarrow X 
\end{equation*}
satisfying certain axioms. 
This diagram is not generally unique.
However, Bondarko showed that if $H \colon \mathcal{C} \rightarrow \mathcal{A}$ is a homological functor with values in an abelian category, then the induced filtration 
\begin{equation}\label{equation:formula_for_bondarkos_weight_filtration}
    W_{i}(H)(X) \colonequals \mathrm{im}(H(X_{\w \geq i}) \rightarrow H(X)) 
\end{equation}
on $ H(X) $ is functorial and does not depend on any choices \cite[{Proposition 2.1.2}]{MR2746283}. 
Bondarko constructed a weight structure on the triangulated category of geometric Voevodsky motives away from the characteristic.
In later work with Sosnilo, they refined this to a weight structure on the triangulated category of compact $\MGL$-modules \cite{bondarko2018constructing}. 

It follows from the above body of work that any cohomology theory which can can be encoded by a homological functor from compact $\MGL$-modules has a canonical weight filtration, given by the formula \eqref{equation:formula_for_bondarkos_weight_filtration}. 
In particular, this applies to Betti cohomology with respect to a module over $\MU$ or étale cohomology with respect to a module over $\Re_{\ell}(\MGL)$. 
Since most familiar examples of complex orientable cohomology theories are represented by $\MU$-modules, in these cases this yields a filtration on cohomology groups, analogous to  \cref{theorem:introduction_complex_orientable_weight_filtration}. 

However, since an appropriate weight structure on $\SH(k)[\einv]$ itself does not exist,%
\footnote{More precisely, there does not exist a weight structure on the compact objects of $\SH(k)[\einv]$ whose weight heart is $\Pure(k)$. 
To see this, note that if $ \Ccal $ is a stable \category equipped with a weight structure, then the axioms imply that every fiber sequence $ X \to Y \to Z  $ in $ \Ccal $ with $ X $, $ Y $, and $ Z  $ in the weight heart $ \Ccal^{\wheart} $ is split.
However, not every fiber sequence with terms in $ \Pure(k) $ is split.} 
we cannot use weight structures directly to prove our main results. Instead, \cref{theorem:introduction_description_of_the_motivic_category_of_sheaves} and the spectral weight filtrations we produce can be thought of as refinements of Bondarko's work.%
\footnote{More precisely, as we observed in \cref{subsection:introduction:new_description_of_motivic_spectra}, by the work of Elmanto--Sosnilo to equip a compactly generated $\infty$-category with a bounded weight structure on its subcategory of compact objects is essentially equivalent to describing it as additive spectral presheaves on the heart \cite[\S2.2]{MR4504902}. 
Thus, our description of $\SH(k)[\einv]$ as spectral \emph{sheaves} can be thought of as a proof of existence of a structure slightly weaker than that of a weight structure, but which we show is still enough to obtain weight filtrations.} 
From this perspective, of particular novelty of the current work is that: 
\begin{enumerate}
    \item The weight filtration is defined as a homotopy-coherent functor taking values in an appropriate filtered \category, rather than at the level of homotopy categories.
    
    \item The filtration is functorially defined for any complex orientable cohomology theory, without choosing an $\MU$-module structure, and so has a natural action by cohomology operations.
\end{enumerate}
We note that the second point is very important in practice. For a concrete example, the homology of a spectrum with respect to $p$-complete complex $K$-theory is naturally equipped with the action of $\mathbb{Z}_{p}^{\times}$ through Adams operations, and by functoriality this descends to an action on weight filtered $\KU_{p}$-homology as constructed in the current work. However, since the Adams operations are not $\MU$-linear, these operations cannot be recovered using the weight structure on $\MGL$-modules. 

This additional functoriality is also crucial in our construction of synthetic Betti realization, which, as explained in \S \ref{intro_subsec:synthetic_Betti_realization}, is informally obtained by gluing weight filtrations on cohomology with respect to $\MU^{\otimes \bullet}$. 
Since the maps in the Adams--Novikov resolution are not $\MU$-linear, these gluings cannot be done in the setting of $\MU$-modules. 
The resulting invariant is much stronger: it restricts to a fully faithful functor on the subcategory of cellular objects after completion at the prime, see \eqref{equation:equivalence_between_cellular_motives_and_synthetic_spectra}.  

On a related note, if the representing spectrum is a commutative algebra in spectra, our functors are lax symmetric monoidal. 
To obtain lax symmetric monoidality in the context of modules over algebraic cobordism, one would need a to upgrade a given commutative ring spectrum to a commutative $\MU$-algebra, of which very few examples are known \cite{hopkins2018strictly}. 


\subsection*{Linear overview}

For the convenience of the reader, in \cref{sec:recollections_on_motivic_homotopy_theory}, we recall the basics of motivic homotopy theory, Betti realization, and étale realization.
We also prove a useful result that allows one to reduce statements about motives of arbitrary varieties to statements about motives of smooth proper varieties; see \Cref{lemma:subcategory_of_motives_containing_smooth_projectives_has_all_motives}.
In \cref{section:motivic_spectra_as_sheaves_on_pure_motives}, we prove \Cref{theorem:introduction_description_of_the_motivic_category_of_sheaves}.
In \cref{section:weight_filtration_on_complex_orientable_cohomology}, we apply our new description of $ \SH(k)[\einv] $ to construct filtered refinements of Betti and étale realization; this proves \Cref{theorem:introduction_complex_orientable_weight_filtration,theorem:introduction_cohomological_etale_realization}.
In \cref{sec:descent_and_the_Gillet-Soule_filtration}, given a complex variety $ X $, we show that our filtration on the compactly supported integral cohains $ \Cc^{*}(X(\CC);\ZZ) $ agrees with the filtration defined by Gillet and Soulé.
See \Cref{theorem:gillet_soule_filtration_comparison}.
In \cref{sec:synthetic_realizations}, we construct the synthetic Betti realization functor
\begin{equation*}
	\Besyn \colon \SH(\CC) \to \SynMU 
\end{equation*}
of \Cref{intro_theorem:existence_of_complex_synthetic_betti_homology} and compare synthetic Betti realization to filtered Betti realization.
See \Cref{theorem:existence_of_complex_synthetic_betti_homology,thm:synthetic_realization_refines_filtered_realization_for_Landwebder_exact_rings}.
We conclude the paper by giving a conjectural description of a synthetic lift of a general motivic realization functor; see \cref{subsection:ideas_on_synthetic_real_and_etale_realizations}.


\subsection*{Acknowledgements}

We would like to thank Tom Bachmann, Bhargav Bhatt, Elden Elmanto, Shane Kelly, Adeel A. Khan, Hana Jia Kong, Jacob Lurie, Vova Sosnilo, and Mura Yakerson for insightful conversations related to this work.


\section{Recollections on motivic homotopy theory}\label{sec:recollections_on_motivic_homotopy_theory}

In this section, we review some of the basic tools we need from stable motivic homotopy theory. 
Our account is quite brief; for more details we refer the reader to \cites{MR3971240}{MR3570135}[\S2]{MR4150251}.

In \cref{subsec:motivic_spectra_and_the_six_operations}, we recall the basic setup of stable motivic homotopy theory and the six operations.
In \cref{subsection:motives_of_varieties}, we collect some basic facts about compactly supported motives attached to schemes.
In \cref{subsection:betti_realization,subsec:etale_realization}, we recall the basics of Betti realization and étale realization, respectively.


\subsection{Motivic spectra and the six operations}\label{subsec:motivic_spectra_and_the_six_operations} 

Given a scheme $ S $, we write $ \Sm_S $ for the category of smooth $ S $-schemes. Informally, the \category of motivic spectra over $ S $ has the same relationship to $ \Sm_{S} $ as the topologists' \category of spectra has to the category of finite CW-complexes. 

\begin{recollection}
	To each scheme $ S $ we associate the symmetric monoidal \category $ \SH(S) $ of \emph{motivic spectra over $ S $}.
	This \category comes equipped with a symmetric monoidal functor 
	\begin{equation*}
		\Sigma_{+}^{\infty} \colon \Sm_{S} \to \SH(S) \comma 
	\end{equation*}
	where $ \Sm_{S} $ is has symmetric monoidal structive given by the cartesian product.
	This construction has the following properties: 
	\begin{enumerate}
	    \item The \category $\SH(S)$ is stable, presentable, and its tensor product preserves colimits separately in each variable.
	    
	    \item The functor $ \Sigma_{+}^{\infty} \colon \Sm_{S}^{\op} \to \SH(S)^{\op} $ is a sheaf with respect to the Nisnevich topology.
	    
	    \item For each $ X \in \Sm_{S} $, the projection $X \times \AA^{1} \to X$ induces an equivalence $ \Sigma_{+}^{\infty} (X \times \AA^{1}) \equivalence \Sigma_{+}^{\infty} X $.
	    
	    \item The \defn{Tate motive} given by the cofiber
	    \begin{equation*}
	        \label{equation:defin_of_s21}
	        \Sup^{2, 1} \colonequals \cofib(\infty \colon \Sigma_{+}^{\infty} S \to \Sigma_{+}^{\infty}( \PP_{S}^{1}))
	    \end{equation*}
	    of the point at infinity is $\tensor$-invertible in $ \SH(S) $. 
	\end{enumerate}
	Moreover, $\SH(S)$ is initial with respect to these properties; that is, given any symmetric monoidal functor $F \colon \Sm_{S} \to \Ccat$ satisfying properties (1)--(4), there exists a unique colimit-preserving symmetric monoidal functor $ \widetilde{F} $ fitting into a commutative triangle
	\begin{equation*}
		\begin{tikzcd}
			\Sm_S \arrow[r, "F"] \arrow[d, "\Sigma_{+}^{\infty}"'] & \Ccat \\ 
			\SH(S) \arrow[ur, dotted, "\widetilde{F}"'] & \phantom{\Ccat} \period
		\end{tikzcd}
	\end{equation*}
 	See \cite[Corollary 2.39]{MR3281141}.
\end{recollection} 

\begin{recollection}[bigraded homotopy groups]
	Given integers $ a, b \in \ZZ $, we have \defn{bigraded spheres} 
	\begin{equation*}
		\Sup^{a, b} \colonequals \Sigma^{a-2b} (\Sup^{2, 1})^{\otimes b} \in \SH(S) \comma 
	\end{equation*}
	where $\Sup^{2, 1}$ for the Tate motive. 
	Since the Tate motive is $ \tensor $-invertible, all bigraded spheres $ \Sup^{a,b} $ are also $\tensor$-invertible.
	Moreover, $\Sup^{0, 0}$ is the monoidal unit of $ \SH(S) $. 
	For any motivic spectrum $E \in \SH(S)$, the \emph{bigraded homotopy groups} of $ E $ are defined as 
	\begin{equation*}
		\pi_{p, q} E \colonequals \pi_{0} \Map_{\SH(S)}(\Sup^{p, q}, E)
	\end{equation*}
	the homotopy classes of maps from bigraded spheres. 
\end{recollection}

\begin{notation}[Thom spectra]
	Let $ S $ be a scheme.
	Given a $ \Kup $-theory class $ \eta \in \Kup_0(S) $, we write $ \Th_S(\eta) \in \SH(S) $ for the \defn{motivic Thom spectrum} associated to $ \eta $.
	If the base scheme is clear, we simply write $ \Th(\eta) $ instead of $ \Th_S(\eta) $.

	Importantly, the Thom spectrum $ \Th(\eta) $ is $ \tensor $-invertible in $ \SH(S) $ with inverse $ \Th(-\eta) $.
	We write
	\begin{equation*}
		\Sigma^{\eta} \colon \SH(S) \equivalence \SH(S)
	\end{equation*}
	for the functor $ \Th(\eta) \tensor (-) $.
\end{notation}

\begin{notation}[Eilenberg--MacLane spectra]
	Let $ S $ be a scheme and let $ R $ be an ordinary commutative ring.
	We write $ \MR_S \in \SH(S) $ for \defn{motivic Eilenberg--MacLane spectrum} representing motivic cohomology with coefficients in $ R $.
	Note that $ \MR_S $ is naturally a commutative algebra in $ \SH(S) $
	When it does not lead to confusion, we simply write $ \MR $ instead of $ \MR_S $.
\end{notation}

\begin{recollection}[relation to Voevodsky motives]
	Given a scheme $ S $, write $ \DM(S) $ for Voevodsky's \category of motives over $ S $.
	If $ S $ is regular over a field with resolution of singularities, then there is an equivalence of symmetric monoidal \categories
	\begin{equation*}
		\DM(S) \equivalent \Mod_{\MZZ}(\SH(S))
	\end{equation*}
	between $ \DM(S) $ and modules in $ \SH(S) $ over the motivic Eilenberg--MacLane spectrum $ \MZZ $.
	See \cites{MR4113773}{MR3404379}{MR2435654}.
\end{recollection}

We now review the basics of functoriality of the construction $S \mapsto \SH(S)$. 
Our account is brief, see \cites[\S1]{MR3874952}[\S2.1]{MR4321205} for a more thorough review.

\begin{recollection}
	For every morphism of schemes $ f \colon X \to Y $, we have an adjunction
	\begin{equation*}
		\adjto{\fupperstar}{\SH(Y)}{\SH(X)}{\flowerstar} \period
	\end{equation*}
	The functor $ \fupperstar $ is the unique symmetric monoidal left adjoint that extends the functor $ \Sm_Y \to \SH(X) $ given by
	\begin{equation*}
		S \mapsto \Sigma_{+}^{\infty} (X \cross_Y S) \period
	\end{equation*}
	If $ f \colon X \to Y $ is smooth, then the forgetful functor $ \Sm_X \to \Sm_Y $ induces a functor
	\begin{equation*}
		\flowersharp \colon \SH(X) \to \SH(Y) \period
	\end{equation*}
	that is left adjoint to $ \fupperstar $.
	Importantly, $ \flowersharp(\Unit_X) \equivalent \Sigma_{+}^{\infty} X $.
\end{recollection}

\begin{recollection}[exceptional adjoints]
	If $ f \colon X \to Y $ is a morphism locally of finite type, we have an `exceptional' adjunction
	\begin{equation*}
		\adjto{\flowershriek}{\SH(X)}{\SH(Y)}{\fuppershriek} 
	\end{equation*}
	along with a natural transformation $ \flowershriek \to \flowerstar $.
	These functors are more difficult to construct, but the following are their main features from the perspective of the present work:
\end{recollection}

\begin{recollection}[compatibilities between the six functors]
	Let $ f \colon X \to Y $ is a morphism locally of finite type.
	The following hold:
	\begin{enumerate}
		\item If $ f $ is proper, then $ \flowershriek \equivalent \flowerstar $.
	
		\item If $ f $ is étale, then
		\begin{equation*}
			\flowershriek \equivalent \flowersharp \andeq \fuppershriek \equivalent \fupperstar \period
		\end{equation*}
      	Combined with (1) we see that for any factorization $ f = p \of j $, where $ j $ is an open immersion and $ p $ is proper, we have
  		\begin{equation*}
			\flowershriek \equivalent \plowerstar \of \jlowersharp \period
		\end{equation*}

		\item \textit{Atiyah duality:} If $ f $ is smooth with relative tangent bundle $ \Tup_f $, then there are equivalences
		\begin{equation*}
			\Sigma^{-\Tup_f} \of \fuppershriek \equivalent \fupperstar \andeq \flowershriek \of \Sigma^{\Tup_f} \equivalent \flowersharp \period
		\end{equation*}

		\item \textit{Projection formula:} There is a natural equivalence
		\begin{equation*}
			\flowershriek(- \tensor \fupperstar(-)) \equivalent \flowershriek(-) \tensor (-)
		\end{equation*}
		of functors $ \SH(X) \cross \SH(Y) \to \SH(Y) $.

		\item \textit{Smooth projection formula:} If $ f $ is smooth, there is a natural equivalence
		\begin{equation*}
			\flowersharp(- \tensor \fupperstar(-)) \equivalent \flowersharp(-) \tensor (-)
		\end{equation*}
		of functors $ \SH(X) \cross \SH(Y) \to \SH(Y) $.

  		\item \textit{Basechange:} Given a cartesian square 
        \begin{equation*}
	        \begin{tikzcd}
				X' \arrow[d, "\fbar"'] \arrow[r, "\pbar"] \arrow[dr, phantom, very near start, "\lrcorner"] & X \arrow[d, "f"] \\
				Y' \arrow[r, "p"'] & Y
	        \end{tikzcd}
        \end{equation*}
        where $f $ is locally of finite type, we have natural equivalences 
		\begin{equation*}
			p^{*} f_{!} \equivalent \fbar_{!} \pbar^{*} \andeq \pbar_{*} \fbar^{!} \equivalent f^{!} p_{*} \period
		\end{equation*}

		\item \textit{Gluing:} Given a closed immersion $ i \colon Z \inclusion X $ with open complement $ j \colon U \inclusion X $, there are natural cofiber sequences
		\begin{equation*}
        	\begin{tikzcd}
            	\jlowershriek\juppershriek \arrow[r] & \id_{\SH(X)} \arrow[r] & \ilowerstar\iupperstar
	        \end{tikzcd}
	    \end{equation*}
	    and 
	    \begin{equation*}
        	\begin{tikzcd}
            	\ilowershriek\iuppershriek \arrow[r] & \id_{\SH(X)} \arrow[r] & \jlowerstar\jupperstar
	        \end{tikzcd}
	    \end{equation*}
		of exact functors $ \SH(X) \to \SH(X) $.
	\end{enumerate}
\end{recollection}

\begin{remark}\label{rem:!-pushforward_along_a_smooth_proper_morphism_preserves_compactness} 
	Let $ f \colon X \to Y $ be a smooth morphism of schemes.
	Then $ \flowershriek \colon \SH(X) \to \SH(Y) $ preserves compact objects.
	To see this, observe that by Atiyah duality, the right adjoint to $ \flowershriek $ is given by $ \fuppershriek \equivalent \Sigma^{\Tup_f} \circ f^{*} $, hence preserves all colimits.
\end{remark}

\begin{recollection}[{\cite[Lemma 2.5]{bachmann2022chow}}]\label{rec:dual_of_Thom_spectrum_on_smooth_proper}
	Let $ f \colon X \to S $ be a smooth proper morphism of schemes.
	Given a class $ \eta \in \Kup_0(X) $, we write
	\begin{equation*}
		\Th_S(\eta) \colonequals \flowersharp \Th_X(\eta) \period
	\end{equation*}
	Write $ \Tup_X $ for the tangent bundle of $ X $.
	Then the motivic spectrum $ \Th_S(\eta) $ is dualizable in $ \SH(S) $ with dual $ \Th_S(-\eta - \Tup_{X}) $.
\end{recollection}

Using the six functors, we can define (co)homology theories associated to morphisms of schemes:

\begin{recollection}[{cohomology}]\label{rec:cohomology_theories}
	Fix a base scheme $ S $ and motivic spectrum $ E \in \SH(S) $.
	Let $ p \colon X \to S $ be a morphism of schemes and $ a,b \in \ZZ $.
	Then:
	\begin{enumerate}
		\item We have the motivic spectrum $ \plowerstar \pupperstar(E) \in \SH(S) $ encoding the \defn{$ E $-cohomology} of $ X $.
		We write 
		\begin{equation*}
			E^{a,b}(X/S) \colonequals \pi_{-a,-b}(\plowerstar\pupperstar(E)) \period
		\end{equation*}

		\item If $ p \colon X \to S $ is locally of finite type, we have the motivic spectrum $ \plowerstar \puppershriek(E) \in \SH(S) $ encoding the \defn{Borel--Moore $ E $-homology} or \defn{bivariant $ E $-homology} of $ X $.
		We write
		\begin{equation*}
			E_{a,b}^{\BM}(X/S) \colonequals \pi_{a,b}(\plowerstar\puppershriek(E)) \period
		\end{equation*}

		\item If $ p \colon X \to S $ is locally of finite type, we have the motivic spectrum $ \plowershriek\pupperstar(E) \in \SH(S) $ encoding the \defn{compactly supported $ E $-cohomology} of $ X $.
		We write
		\begin{equation*}
			E_{\c}^{a,b}(X/S) \colonequals \pi_{-a,-b}(\plowershriek\pupperstar(E)) \period
		\end{equation*}

		\item If $ p \colon X \to S $ is locally of finite type, we have the motivic spectrum $ \plowershriek\puppershriek(E) \in \SH(S) $ encoding the \defn{$ E $-homology} of $ X $.
		We write
		\begin{equation*}
			E_{a,b}(X/S) \colonequals \pi_{a,b}(\plowershriek\puppershriek(E))\period
		\end{equation*}
	\end{enumerate}
\end{recollection}

\subsection{Compactly supported motives of schemes} 
\label{subsection:motives_of_varieties}

As surveyed in \cref{rec:cohomology_theories}, the six functor formalism provides a very general form of cohomology theory.
However, it is often convenient to work with an alternative description, obtained by attaching to any $ S $-scheme a suitable motivic spectrum. 
The relationships between schemes (such as an open-closed decomposition) can then be encoded via relationships between these motivic spectra.  

\begin{definition}
	\label{definition:motive_and_motivec_of_a_variety} 
	Let $ p \colon X \to S $ be a locally of finite type morphism of schemes.
	The \emph{(compactly supported) motive associated to $ X $} is the motivic spectrum over $ S $ given by 
	\begin{equation*}
		\motive(X/S) \colonequals p_{!}(\Unit_{X}) \period
	\end{equation*}
	If the base scheme $ S $ is clear from the context, then we simply write $ \Mc(X) $ for $ \Mc(X/S) $.
\end{definition}

\begin{example}
	\label{example:motive_of_smooth_variety}
	Let $ k $ be a field and let $ p \colon X \to \Spec(k) $ be a smooth morphism with relative tangent bundle $ \Tup_{X} $. 
	Then by Atiyah duality we have 
	\begin{equation*}
		p_{!} \simeq p_{\sharp} \circ \Sigma^{-\Tup_{X}} \period
	\end{equation*}
It follows that $\motive(X)$ can be identified with the Thom spectrum $\Th_{X}(-\Tup_{X})$ of the negative tangent bundle. This is, informally, a twisted form of the suspension spectrum of $ X $; in the particular case when $ X $ is a variety of dimension $d$ with trivial tangent bundle, then 
	\begin{equation*}
		\motive(X) \simeq \Sigma^{-2d, d} \Sigma_{+}^{\infty} X \period
	\end{equation*} 
\end{example}

\begin{observation}
\label{observation:motive_of_smooth_projective_variety_is_dual_to_the_suspension_spectrum} 
	Let $ k $ be a field and let $ X $ be a smooth projective $ k $-scheme.
	As a consequence of Atiyah duality, $\motive(X)$ is the monoidal dual of $\Sigma_{+}^{\infty} X$, see \cite[Theorem 2.2]{riou2005dualite}. 
	Moreover, \Cref{rem:!-pushforward_along_a_smooth_proper_morphism_preserves_compactness} shows that $ \Mc(X) $ is also compact.
\end{observation}

The compatibilies of the six functors show that the compactly supported motive $ \Mc(X/S) $ encodes both the Borel--Moore homology and compactly supported cohomology of $ X $ with coefficients in an arbitrary motivic spectrum over $ S $:

\begin{observation}\label{obs:compactly_supported_cohomology_as_a_tensor_product}
	Let $ p \colon X \to S $ be a locally of finite type morphism of schemes and $ E \in \SH(S) $.
	Using the projection formula, the motivic spectrum encoding compactly supported $ E $-cohomology can also be described as
	\begin{equation*}
		\plowershriek\pupperstar(E)
		\equivalent \plowershriek(\Unit_X) \tensor E 
		= \Mc(X/S) \tensor E \period
	\end{equation*}
\end{observation}

\begin{observation}\label{obs:Borel-Moore_homology_as_a_Hom}
	Let $ p \colon X \to S $ be a locally of finite type morphism of schemes and $ E \in \SH(S) $.
	Using the fact that $ \plowerstar\puppershriek $ is right adjoint to $ \plowershriek\pupperstar $, we have equivalences
	\begin{align*}
		\plowerstar\puppershriek(E) &\equivalent \Hom_{\SH(S)}(\Unit_S, \plowerstar\puppershriek(E)) \\ 
		&\equivalent \Hom_{\SH(S)}(\plowershriek\pupperstar(\Unit_S), E) \\ 
		&\equivalent \Hom_{\SH(S)}(\Mc(X/S), E) \period
	\end{align*}
\end{observation}

\begin{observation} 
	\label{observation:bm_homology_and_c_cohomology_using_the_motive}
	As a consequence of \Cref{obs:compactly_supported_cohomology_as_a_tensor_product,obs:Borel-Moore_homology_as_a_Hom}, the compactly supported cohomology and Borel--Moore homology of $ X $ over $ S $ are computed by
	\begin{equation*}
		E_{\c}^{a,b}(X/S) \simeq \pi_{-a,-b}(\motive(X/S) \tensor E)
	\end{equation*}
	and  
	\begin{equation*}
		E^{\BM}_{a,b}(X/S) \simeq \pi_{a,b} \Hom_{\SH(S)}(\motive(X/S), E) \period
	\end{equation*}
\end{observation} 

\begin{warning}
	The isomorphisms of \cref{observation:bm_homology_and_c_cohomology_using_the_motive} are opposite to the ones appearing in topology for usual homology and cohomology.
	That is, it is homology which is defined by mapping into a spectrum and cohomology which is defined using the tensor product. 
	This is because $ \motive(X) $ encodes the \textit{compactly supported} theories: $ \motive(X) $ should be thought of as a ``cohomological'' motive, as witnessed by its contravariant functoriality of \cref{construction:functoriality_of_the_motive_of_variety}. 
\end{warning}

The formation of the compactly supported motive of a scheme commutes with basechange:

\begin{lemma}\label{lemma:formation_of_motive_of_variety_commutes_with_field_extensions}
	Given a cartesian square of schemes
	\begin{equation*}
		\begin{tikzcd}
			X' \arrow[d, "p'"'] \arrow[r, "f'"] \arrow[dr, phantom, very near start, "\lrcorner"] & X \arrow[d, "p"] \\
			S' \arrow[r, "f"'] & S
		\end{tikzcd}
	\end{equation*}
	where $ p $ is locally of finite type, there is an equivalence $ \motive(X'/S') \equivalent \fupperstar \motive(X/S) $.
\end{lemma}

\begin{proof}
	Using the fact that pullback and exceptional pushforward satisfy basechange, we compute
	\begin{align*}
		\motive(X'/S') &= p_{!}' (\Unit_{X'}) \simeq p_{!}' (f')^{*} (\Unit_{X}) \\
		&\simeq f^{*} p_{!} (\Unit_{X}) \\ 
		&= f^{*} \motive(X/S) \period \qedhere
	\end{align*}
\end{proof}

\begin{construction}[{functoriality of $ \Mc $}]
	\label{construction:functoriality_of_the_motive_of_variety} 
	Consider a commutative triangle of locally of finite type morphisms of schemes
	\begin{equation*}
	\begin{tikzcd}
		X \arrow[rr, "f"] \arrow[dr, "p"'] & & Y \arrow[dl, "q"] \\
		& S & \phantom{Y} \period
	\end{tikzcd}
	\end{equation*}
	The construction $ X \mapsto \motive(X/S) $ has the following functorialities:
	\begin{enumerate}
		\item \textit{Contravariant functoriality in proper maps:} Assume that $ f $ is proper.
		Using the equivalence $ f_{!} \simeq f_{*} $, the unit of the adjunction $f^{*} \dashv f_{*}$
		provides a map 
		\begin{equation*}
			\motive(Y/S) = q_{!}(\Unit_{Y}) \longrightarrow q_{!} f_{*} f^{*} (\Unit_{Y}) \simeq q_{!} f_{!} f^{*} (\Unit_{Y}) \simeq q_{!} f_{!} (\Unit_{X}) \simeq p_{!} (\Unit_{X}) = \motive(X/S) \period
		\end{equation*}

		\item \textit{Covariant functoriality in étale maps:} Assume that $ f $ is étale.
		Using the equivalence $ f^{*} \simeq f^{!} $
		the counit map of $f_{!} \dashv f^{!}$ yields a map 
		\begin{equation*}
			\motive(X/S) = q_{!} f_{!} (\Unit_{X}) \simeq q_{!} f_{!} f^{*} (\Unit_{Y}) \simeq q_{!} f_{!} f^{!} (\Unit_{Y}) \longrightarrow q_{!} (\Unit_{Y}) = \motive(Y/S) \period
		\end{equation*}
	\end{enumerate} 
\end{construction}

\begin{lemma}
	\label{lemma:cofiber_sequence_of_motives_associated_to_open_closed_decomposition}
	Let $ p \colon X \to S $ be locally of finite type morphism of schemes.
	Let $ i \colon Z \hookrightarrow X $ be a closed immersion with open complement $ j \colon U \hookrightarrow X $. 
	Then the induced maps $\motive(U/S) \to \motive(X/S)$ and $\motive(X/S) \to \motive(Z/S)$ assemble into a natural cofiber sequence 
    \begin{equation*}
        \begin{tikzcd}
            \motive(U/S) \arrow[r] & \motive(X/S) \arrow[r] & \motive(Z/S)
        \end{tikzcd}
    \end{equation*}
    in $ \SH(S) $.  
\end{lemma}

\begin{proof}
    There is a gluing cofiber sequence
    \begin{equation*}
        \begin{tikzcd}
            \jlowershriek\jupperstar(\Unit_X) \arrow[r] & \Unit_X \arrow[r] & \ilowershriek\iupperstar(\Unit_X)
        \end{tikzcd}
    \end{equation*}
    in $ \SH(X) $.
    Applying $ \plowershriek \colon \SH(X) \to \SH(S) $ to this cofiber sequence and using the fact that $ \iupperstar $ and $ \jupperstar $ are symmetric monoidal, we obtain a cofiber sequence  
     \begin{equation*}
        \begin{tikzcd}
            \plowershriek\jlowershriek(\Unit_U) \arrow[r] & \plowershriek(\Unit_X) \arrow[r] & \plowershriek\ilowershriek(\Unit_Z)
        \end{tikzcd}
    \end{equation*}
    in $ \SH(S) $. 
    The claim now follows from the definition of the compactly supported motive of an $ S $-scheme.
\end{proof}

The following is often useful, as it allows one to reduce statements about arbitrary varieties to statements about smooth proper varieties. 

\begin{lemma} 
	\label{lemma:subcategory_of_motives_containing_smooth_projectives_has_all_motives}
	Let $ k $ be a field of exponential characteristic $ e $.
	Let $\Ccat \subseteq \SH(k)[\einv]$ be a full subcategory with the following two properties: 
	\begin{enumerate}
	    \item The subcategory $ \Ccat $ is closed under extensions, fibers, and retracts.

	    \item For each smooth projective $ k $-variety $ X $, we have $ \motive(X)[\einv] \in \Ccat $.
	\end{enumerate}
	Then for any $ k $-variety $ U $, we have $\motive(U)[\einv] \in \Ccat $. 
\end{lemma} 

\begin{proof}
	We argue by induction on the dimension of $ U $.

	The base case is when $\dim(U) = 0$, so that $U$ is projective. 
	In this case, if $ k $ is perfect, then $U$ is also smooth, and we are done. 
	If $ k $ is not perfect, we consider the perfection $ r \colon k \hookrightarrow k' $ given by the colimit over the Frobenius morphism. 
	By a result of Elmanto--Khan \cite[Corollary 2.1.7]{MR3999675}, the pullback functor 
	\begin{equation*}
			r^{*} \colon \SH(k)[\einv] \to \SH(k')[\einv]
	\end{equation*}
	is an equivalence. 
	Writing $ U' $ for the basechange of $ U $ to $ k' $, \cref{lemma:formation_of_motive_of_variety_commutes_with_field_extensions} shows that 
	\begin{equation*}
		r^{*} \motive(U/k) \simeq \motive(U'/k') \period
	\end{equation*}
	Write $ \Et_k $ and $ \Et_{k'} $ for the small étale sites of $ k $ and $ k' $, respectively.
	Since $r$ is a universal homeomorphism, the topological invariance of the étale site \cites[Exposé IX, Théorème 4.10]{SGA1}[Exposé VIII, Théorème 1.1]{SGA4} implies that the basechange functor
	\begin{equation*}
		\Et_k \to \Et_{k'}
	\end{equation*}
	is an equivalence of categories.
	It follows that there exists a zero-dimensional étale $ k $-scheme $ V $ such that $ V' \simeq U' $ as $k'$-schemes. 
	Again applying \Cref{lemma:formation_of_motive_of_variety_commutes_with_field_extensions}, we see that
	\begin{equation*}
		r^{*} \motive(V/k)[\einv] \simeq r^{*} \motive(U/k)[\einv] \period
	\end{equation*}
	Since $r^{*}$ is fully faithful, we deduce that $\motive(V)[\einv] \simeq \motive(U)[\einv]$. 
	By assumption, $ \motive(V)[\einv] \in \Ccal $, hence $ \motive(U)[\einv] \in \Ccal $ as well.

	For the induction step, assume that $ \dim(U) > 0 $ and that for each $ k $-variety $ Z $ such that $ \dim(Z) < \dim(U) $, we have $ \motive(Z)[\einv] \in \Ccat $. 
	By \cref{lemma:cofiber_sequence_of_motives_associated_to_open_closed_decomposition}, for any closed $Z \subseteq U$ we have a cofiber sequence 
	\begin{equation*}
		\motive(U \smallsetminus Z)[\einv] \to \motive(U)[\einv]  \to \motive(Z)[\einv] \period
	\end{equation*}
	Hence it is enough to show that, after possibly replacing $U$ by an open dense subset, we have $\motive(U)[\einv] \in \Ccat$. 
	Applying \cref{lemma:cofiber_sequence_of_motives_associated_to_open_closed_decomposition} to a decomposition into connected components, we can assume that $U$ is connected.
	By further shrinking $ U $, we can also assume that $U$ is smooth with trivial tangent bundle. 

	By the theory of alterations we can find a finite étale cover $V \to U$ of degree $d$ coprime to $ e $ such that $ V $ is an open dense subset of a smooth and projective $ k $-variety $ X $. 
	By the inductive hypothesis and an application of \cref{lemma:cofiber_sequence_of_motives_associated_to_open_closed_decomposition}, we deduce that $\motive(V)[\einv] \in \Ccat$.
	We want to deduce the same for $\motive(U)[\einv]$. 
	Since both $U$ and $V$ have trivial tangent bundles, \cref{example:motive_of_smooth_variety} shows that 
	\begin{equation*}
		\motive(U) \simeq \Sigma^{-2d, d} \Sigma_{+}^{\infty} U \andeq \motive(V) \simeq \Sigma^{-2d, d} \Sigma_{+}^{\infty} V \period
	\end{equation*}
	We deduce from 
	\cite[Lemma B.3]{levine2019algebraic} that after possibly shrinking $U$, the motivic spectrum $\motive(U)[\einv]$ is a retract of $\motive(V)[\einv]$, ending the argument. 
\end{proof} 

\begin{corollary}\label{corollary:motives_of_varieties_are_dualizable_away_from_the_characteristic} 
	Let $ k $ be a field of exponential characteristic $ e $.
	Then for any $ k $-variety $ X $, the motivic spectrum $\motive(X)[\einv]$ is a compact and dualizable object of $\SH(k)[\einv]$. 
\end{corollary}

\begin{proof}
	Since compact and dualizable objects form a stable subcategory, this follows from \cref{lemma:subcategory_of_motives_containing_smooth_projectives_has_all_motives} and the smooth projective case of \cref{observation:motive_of_smooth_projective_variety_is_dual_to_the_suspension_spectrum}. 
\end{proof}


\subsection{Betti realization} 
\label{subsection:betti_realization}

We now recall the basics of Betti realizations in characteristic zero.
The first is over the complex numbers.

\begin{construction}[complex Betti realization]\label{construction:complex_betti_realization}
	The functor $ \Sm_{\CC} \to \Spc $ sending a smooth $ \CC $-scheme to the underlying homotopy type of the topological space $ X(\CC) $ with the analytic topology is $ \AA^1 $-invariant, sends elementary Nisnevich squares to pullback squares, and preserves finite products. 

	Moreover, the functor $ \Sm_{\CC} \to \Sp $ given by $ X \mapsto \Sigma_+^{\infty} X(\CC) $ also inverts the Tate motive. 
	As a consequence of the universal property of motivic spectra, this functor uniquely extends to a symmetric monoidal left adjoint
	\begin{equation*}
	    \Be \colon \SH(\CC) \to \Sp
	\end{equation*}
	referred to as \defn{Betti realization}.
\end{construction}

\begin{example}[{\cite[Proposition 5.10]{MR3217623}}]\label{ex:Betti_realization_of_HR}
	There is a natural equivalence 
	\begin{equation*}
		\Be(\MR) \equivalent \HR
	\end{equation*}
	between the Betti realization of the motivic Eilenberg--MacLane spectrum $ \MR $ and the usual Eilenberg--MacLane spectrum of $ R $.
\end{example}

\begin{example}\label{ex:Betti_realization_of_MU}
	There is an equivalence 
	\begin{equation*}
		\Be(\MGL) \equivalent \MU
	\end{equation*}
	of commutative algebras in $ \Sp $.
\end{example}



\begin{construction}[{$ \Cup_2 $-Betti realization}]
    Similarly, if $ X $ is a smooth $ \RR $-scheme, then the complex points $ X(\CC) $ acquire an action of the Galois group $ \Cup_2 \colonequals \Gal(\CC/\RR) $.
    The underlying homotopy type of $ X(\CC) $ refines to a genuine $ \Cup_2 $-space.
    Again by the universal proeprty of motivic spectra, the functor
    \begin{align*}
        \Sm_{\RR} &\to \Sp_{\Cup_2} \\ 
        X &\mapsto \Sigma_{\Cup_2,+}^{\infty} X(\CC)
    \end{align*}
    uniquely extends to a symmetric monoidal left adjoint
    \begin{equation*}
        \Be_{\Cup_2} \colon \SH(\RR) \to \Sp_{\Cup_2}
    \end{equation*}
    valued in genuine $ \Cup_2 $-spectra.
    This functor is referred to as \defn{$ \Cup_2 $-Betti realization}.
\end{construction}



\subsection{Étale realization}\label{subsec:etale_realization}

Let $ k $ be a separably closed field and $ \ell $ a prime different from $ \characteristic(k) $.
We now explain a construction of an étale realization functor from $ \SH(k) $ to $ \ell $-complete spectra.
In fact, we give a more general construction that works over any base scheme.

\begin{notation}
	Let $ S $ be a scheme.
	Write $ \Et_S \subseteq \Sm_S $ for the full subcategory spanned by the étale $ S $-schemes.
	Giving both of these categories the \textit{étale} topology, this inclusion $ \Et_S \subseteq \Sm_S $ is a morphism of sites that satisfies the covering lifting property.
	In particular, this inclusion induces a fully faithful symmetric monoidal pullback functor
	\begin{equation*}
		i^{\ast} \colon \Shethyp(\Et_S;\Sp) \hookrightarrow \Shethyp(\Sm_S;\Sp)
	\end{equation*}
	on étale hypersheaves of spectra.
\end{notation}

\begin{notation}
	Let $ S $ be a scheme.
	Write $ \SH_{\et}(S) $ for the localization of $ \SH(S) $ at the \textit{desuspensions of} étale hypercoverings.
	Write $ \Let \colon \SH(S) \to \SH_{\et}(S) $ for the symmetric monoidal localization functor.
\end{notation}

\begin{nul}
	Equivalently, the \category $ \SH_{\et}(S) $ can be obtained by first taking $ \AA^1 $-local objects in the \topos $ \Shethyp(\Sm_S) $ of étale hypersheaves of spaces on smooth $ S $-schemes, then $ \PP^1 $-stabilizing.
	As a result, there is a natural symmetric monoidal left adjoint 
	\begin{equation*}
		\Shethyp(\Sm_S;\Sp) \to \SH_{\et}(S) \period 
	\end{equation*}
\end{nul}
	
\begin{notation}
	Let $ \Ccal $ be a presentable stable \category and $ \ell $ a prime number.
	A morphism $ f \colon X \to Y $ in $ \Ccal $ is an \defn{$ \ell $-equivalence} if $ \cofib(f)/\ell = 0 $.
	We write $ \Ccal\ellcomp \subseteq \Ccal $ for the localization of $ \Ccal $ at the $ \ell $-equivalences.
	We refer to $ \Ccal\ellcomp $ as the subcategory of \defn{$ \ell $-complete} objects.
	Then inclusuon $ \Ccal\ellcomp \subseteq \Ccal $ admits a left adjoint that we denote by $ (-)\ellcomp \colon \Ccal \to \Ccal\ellcomp $.
\end{notation}

The following rigidity result of Bachmann generalizes work of Ayoub \cite[\S5]{MR3205601} as well as earlier work by Bachmann \cite[Theorem 6.6]{MR4224739}.

\begin{theorem}[{rigidity \cite[Theorem 3.1]{arXiv:2104.06002}}]\label{thm:Bachmann_rigidity_of_etale_SH}
	Let $ S $ be a scheme and $ \ell $ a prime number invertible on $ S $.
	Then the natural symmetric monoidal left adjoint
	\begin{equation*}
		\begin{tikzcd}
			\Shethyp(\Et_S;\Sp)\ellcomp \arrow[r, "i_{\ell}^{\ast,\wedge}"] & \Shethyp(\Sm_S;\Sp)\ellcomp \arrow[r] & \SH_{\et}(S)\ellcomp 
		\end{tikzcd}
	\end{equation*}
	is an equivalence.
\end{theorem}

\begin{definition}[étale realization]
\label{definition:etale_realization}
	Let $ S $ be a scheme and $ \ell $ a prime number invertible on $ S $.
	The \defn{$ \ell $-adic étale realization functor} is the composite
	\begin{equation*}
		\begin{tikzcd}
			\Re_{\ell} \colon \SH(S) \arrow[r, "\Let"] & \SH_{\et}(S) \arrow[r, "(-)\ellcomp"] & \SH_{\et}(S)\ellcomp \arrow[r, "\sim"{yshift=-0.25ex}] & \Shethyp(\Et_S;\Sp)\ellcomp \period
		\end{tikzcd}
	\end{equation*}
	Here the last equivalence is the inverse of the rigidity equivalence of \Cref{thm:Bachmann_rigidity_of_etale_SH}.
	Note that $ \Re_{\ell} $ is a composite of symmetric monoidal left adjoints, hence is a symmetric monoidal left adjoint.
\end{definition}

\begin{example}
	Let $ k $ be a separably closed field and $ \ell \neq \characteristic(k) $.
	Then $ \ell $-adic étale realization provides a symmetric monoidal left adjoint 
	\begin{equation*}
		\Re_{\ell} \colon \SH(k) \to \Sp\ellcomp 
	\end{equation*}
	to $ \ell $-complete spectra.
\end{example}


\section{Motivic spectra as sheaves on pure motives} 
\label{section:motivic_spectra_as_sheaves_on_pure_motives}

Let $ k $ be a field of exponential characteristic $ e $.
Our goal in this section is to describe the \category $ \SH(k)[\einv] $ of motivic spectra away from the characteristic in terms of motives of smooth proper $ k $-schemes (see \Cref{theorem:shk_is_additive_sheaves_on_perfect_pure_motives}). 

In \cref{subsec:perfect_pure_motivic_spectra} we introduce a subcategory $ \Pure(k) \subseteq \SH(k)[\einv] $ of pure motives and explore its basic properties.
In \cref{subsec:characterization_of_cofiber_sequences_of_perfect_pure_motivic_spectra} characterizes the cofiber sequences in $ \Pure(k) $; see \Cref{proposition:pure_epimorphisms_detected_by_mgl}.
In \cref{subsec:pure_sheaves} we prove our alternative description of $ \SH(k)[\einv] $.

\begin{notation}\label{ntn:SH_means_SH_with_e_inverted}
	Let $ k $ be a field of exponential characteristic $ e $.
	For the remainder of this section, we simply write
	\begin{equation*}
		\SH(k) \colonequals \SH(k)[\einv]
	\end{equation*}
	for the localization of the stable motivic category away from the exponential characteristic. 
	All of the motivic spectra appearing below are implicitly localized as well. 
\end{notation}


\subsection{Perfect pure motivic spectra}\label{subsec:perfect_pure_motivic_spectra}

We start by introducing the subcategory of `pure motives' relevant for our work.
Our definition is inspired by Bachmann, Kong, Wang, and Xu's recent introduction of the \textit{Chow--Novikov \tstructure} on motivic spectra \cite{MR4387236}.

\begin{definition}
	\label{definition:perfect_pure_motivic_spectrum}
	We write 
	\begin{equation*}
		\perfectpures \subseteq \SH(k)
	\end{equation*}
	for the smallest subcategory closed under extensions and retracts which contains the Thom spectrum $ \Th(\eta) $ for any smooth proper $ k $-scheme $ X $ and any class $ \eta \in \Kup_{0}(X) $.
	We say a motivic spectrum $ A $ is \defn{perfect pure} if $ A \in \Pure(k) $.
\end{definition}

\begin{remark}
	The connective part $ \SH(k)_{c\geq 0} $ of the Chow--Novikov \tstructure is the closure of $ \Pure(k) \subseteq \SH(k) $ under colimits and extensions.
\end{remark}

We begin by enumerating the basic features of $ \Pure(k) $.

\begin{lemma}\label{lemma:perfect_pures_closed_under_monoidal_product_and_duals}
	The following statements hold:
	\begin{enumerate}
		\item Every object of $ \Pure(k) $ is dualizable in $ \SH(k) $.

		\item The subcategory $ \Pure(k) \subseteq \SH(k) $ is closed under monoidal duals.

		\item Every object of $ \Pure(k) $ is compact in $ \SH(k) $.

		\item The subcategory $ \Pure(k) \subseteq \SH(k) $ is closed under tensor products.
	\end{enumerate}
\end{lemma}

\begin{proof}
	Items (1) and (2) are immediate from the definition of $ \Pure(k) $, \Cref{rec:dual_of_Thom_spectrum_on_smooth_proper}, and the fact that dualizable objects are closed under extensions.
	Item (3) follows from item (1) and the fact that, since the unit of $ \SH(k) $ is compact, every dualizable object of $ \SH(k) $ is compact.

	For item (4), note that if $ X $ and $ X' $ are smooth $ k $-schemes and $ \eta \in \Kup_0(X) $ and $ \eta' \in \Kup_0(X') $, then 
	\begin{equation*}
		\Th(\eta) \tensor \Th(\eta') \equivalent \Th(\eta \cross \eta') \period
	\end{equation*}
	Hence the claim follows from the definition of $ \Pure(k) $ and the fact that smooth proper $ k $-schemes are closed under fiber products in $ \Sm_k $.
\end{proof}

\begin{warning}
	\Cref{definition:perfect_pure_motivic_spectrum} is related to, but distinct from, the notion of a \textit{pure motivic spectrum} introduced in \cite[Definition 2.10]{bachmann2022chow}.
	The subcategory of pure motivic spectra in the sense of Bachmann--Kong--Wang--Xu is the closure of $ \Pure(k) $ under filtered colimits and extensions. 
	Using the fact that perfect pure motivic spectra are compact, it is not difficult to show that a pure motivic spectrum $ A $ is perfect pure if and only $ A $ is compact. 
\end{warning}

\begin{remark}
	Since we work away from the characteristic, \cites[Remark 2.19]{bachmann2022chow}[Theorem 3.2.1]{MR3999675}[Proposition B.1]{MR4023393} show that $ \Pure(k) $ generates $ \SH(k) $ under colimits and desuspensions. 
\end{remark}

An important class of examples of motivic spectra are Thom spectra associated to vector bundles on Grassmannians:

\begin{notation}[Grassmannians]
	Let $ n \geq d \geq 0 $ be integers.
	Write
	\begin{equation*}
		\Gr_{d}(n) \colonequals \Gr_{d}(\AA_{k}^{n})
	\end{equation*}
	for the Grassmanian of $d$-dimensional linear subspaces of $\AA_k^{n}$. 
	Recall that $ \Gr_d(n) $ is a smooth projective variety of dimension $d(n-d)$.
\end{notation}

\begin{example}
    \label{example:thom_spectra_of_grassmanians_are_perfect_pure}
	Write $\gamma_{d, n}$ for the tautological bundle of rank $ d $ over $ \Gr_{d}(n) $  and 
	\begin{equation*}
		\epsilon_{d, n} \colonequals [\gamma_{d, n}] - [\Ocal_{\Gr_{d}(n)}^{\oplus d}] \in \Kup_{0}(\Gr_{d}(n)) \period
	\end{equation*}
    for the associated virtual vector bundle of rank zero.
    Write $ \Th_{d}(n) \colonequals \Th(\epsilon_{d, n}) $ for the associated Thom spectrum.
    Since $ \Gr_{d}(n) $ is smooth and proper, $ \Th_{d}(n) $ is perfect pure. 

    Since 
	\begin{equation*}
		\MGL \equivalent \varinjlim_{d , n \to \infty} \Th_{d}(n+d) \comma
	\end{equation*}
	we deduce that $ \MGL $ is a filtered colimit of perfect pure motivic spectra. 
\end{example}

We are particularly interested in cofiber sequences in $ \Pure(k) $; hence we make the following definitions.

\begin{definition}
	\label{definition:pure_epimorphism}
	\label{definition:pure_monomorphism}
    We say that a morphism $ f \colon B \to A $ in $ \Pure(k) $ is: 
    \begin{enumerate}
         \item A \emph{pure epimorphism} if its fiber $ \fib(f) $ in $ \SH(k) $ is again a perfect pure motivic spectrum. 
         
         \item A \emph{pure monomorphism} if its monoidal dual $ f^{\vee} \colon B^{\vee} \to A^{\vee}$ is a pure epimorphism; equivalently, if the cofiber $ \cofib(f) $ in $ \SH(k) $ is perfect pure. 
    \end{enumerate}
\end{definition}

The transition maps appearing in \Cref{example:thom_spectra_of_grassmanians_are_perfect_pure} are all pure monomorphisms:

\begin{lemma}\label{lem:pure_monomorphisms_between_Thom_spectra_of_vector_bundles_on_Grassmannians}
	Let $ d,m \geq 0 $ be integers.
	Then the following maps are pure monomorphisms:
	\begin{enumerate}
		\item The map $\Th_{d}(m) \to \Th_{d+1}(m+1)$ induced by the morphism $ \Gr_{d}(m) \to \Gr_{d+1}(m+1) $ classifying $ \gamma_{d, m} \oplus \Ocal_{\Gr_{d}(m)} $.

		\item The map $\Th_{d+1}(m) \to \Th_{d}(m+1)$ induced by the map $\Gr_{d+1}(m) \to \Gr_{d}(m+1)$ classifying 
		\begin{equation*}
			\gamma_{d+1, m} \subseteq \Ocal_{\Gr_{d+1}(m)}^{\oplus m} \subseteq \Ocal_{\Gr_{d+1}(m)}^{\oplus m+1} \period 
		\end{equation*}
	\end{enumerate}
\end{lemma}

\begin{proof}
	For (1), write $ U $ for the open complement of the closed immersion
	\begin{equation*}
		\Gr_{d+1}(m) \hookrightarrow \Gr_{d+1}(m+1)
	\end{equation*}
	induced by the inclusion $ \AA_{k}^{m} \subseteq \AA_{k}^{m+1} $.
	Note that the map $ \Gr_{d}(m) \to \Gr_{d+1}(m+1) $ factors as 
	\begin{equation}
	    \label{equation:factorization_of_map_of_grassmanians}
	    \Gr_{d}(m) \hookrightarrow U \hookrightarrow \Gr_{d+1}(m+1) \period
	\end{equation}
	Moreover, the left map in \eqref{equation:factorization_of_map_of_grassmanians} is an affine vector bundle and hence a motivic homotopy equivalence.
	Applying purity (see \cite[Lemma A.2]{bachmann2022chow}) to the open-closed decomposition
	\begin{equation}
	\label{equation:open_closed_decomposition_of_grassmanians}
	    \begin{tikzcd}
	        U \arrow[r, hook, "j"] & \Gr_{d+1}(m+1) & \Gr_{d+1}(m) \arrow[l, hook', "i"']
	    \end{tikzcd}
	\end{equation}
	and the virtual vector bundle $\epsilon_{d+1, m+1}$ gives a cofiber sequence in $ \SH(k) $ of the form 
	\begin{equation*}
		\begin{tikzcd}
	         \Th_{d}(m) \arrow[r] & \Th_{d+1}(m+1) \arrow[r] & \Th_{\Gr_{d+1}(m)}(\epsilon_{d+1, m+} \oplus \Ncal) \period
	    \end{tikzcd}
	\end{equation*}
	Here, $ \Ncal $ is the normal bundle of $\Gr_{d+1}(m) \hookrightarrow \Gr_{d+1}(m+1)$. 
	As the cofiber is perfect pure, we deduce that the first map is a pure monomorphism.

	For (2), note that these are the maps corresponding to the closed component in \eqref{equation:open_closed_decomposition_of_grassmanians}. 
	Write $ \Tup_{\Gr_{d+1}(m+1)} $ for the tangent bundle of $ \Gr_{d+1}(m+1) $, and define a virtual vector bundle $ V $ on $ \Gr_{d+1}(m+1) $ by
	\begin{equation*}
	    V \colonequals \Tup_{\Gr_{d+1}(m+1)} \oplus \epsilon_{d+1, m+1} \period
	\end{equation*}
	Applying purity to $ V $, we obtain a cofiber sequence of the form 
	\begin{equation*}
	    \begin{tikzcd}
	         \Th_{\Gr_{d}(m)}(j^{*}V) \arrow[r] & \Th_{\Gr_{d+1}(m+1)}(V) \arrow[r] & \Th_{\Gr_{d+1}(m)}(i^{*}V \oplus \Ncal) \period
	    \end{tikzcd}
	\end{equation*}
	Passing to monoidal duals and applying \Cref{rec:dual_of_Thom_spectrum_on_smooth_proper}, we obtain a cofiber sequence
	\begin{equation*}
	    \begin{tikzcd}
	        \Th_{d+1}(m) \arrow[r] & \Th_{d+1}(m+1) \arrow[r] & \Th_{\Gr_{d}(m)}(i^{*}V - \Tup_{\Gr_{d}(m)}) \period
	    \end{tikzcd}
	\end{equation*} 
	This shows that the right-hand map is a pure monomorphism, as needed. 
\end{proof}

\begin{example}\label{ex:MGL_as_filtered_colimit_along_pure_monomorphisms}
	In light of \Cref{example:thom_spectra_of_grassmanians_are_perfect_pure,lem:pure_monomorphisms_between_Thom_spectra_of_vector_bundles_on_Grassmannians}, we can write 
	\begin{equation*}
		\MGL \equivalent \varinjlim_{d,n \to \infty} \Th_{d}(n+d) 
	\end{equation*}
	as the colimit a filtered diagram of perfect pure motivic spectra where all of the transition maps are pure monomorphisms. 
\end{example}


\subsection{Characterization of cofiber sequences of perfect pure motivic spectra}\label{subsec:characterization_of_cofiber_sequences_of_perfect_pure_motivic_spectra}

We now give a useful characterization of pure epimorphisms.
In \cref{subsec:pure_sheaves}, we use this characterization to give a description of $ \SH(k) $ as \acategory of sheaves of spectra on $ \Pure(k) $.
Before we start, let us recall a number of equivalent characterizations of split cofiber sequences. 

\begin{recollection}[split cofiber sequences]\label{rec:split_cofiber_sequences}
	If $ \Ccal $ is an additive \category, a cofiber sequence
	\begin{equation}\label{eq:split_cofiber_sequence}
		\begin{tikzcd}
			A \arrow[r, "i"] & B \arrow[r, "p"] & C
		\end{tikzcd}
	\end{equation}
	is said to be \emph{split} if there exists a section $ s \colon C \to B $ of $ p $, which implies that $B \simeq A \oplus C$. 
    In this case, we say that $ i \colon A \to B $ is a \defn{split monomorphism}, and $ p \colon B \to C $ is a \defn{split epimorphism}.
\end{recollection}

\begin{recollection}
    Any additive functor $\Ccal \to \Dcal$ of additive \categories preserves split cofiber sequences.
\end{recollection}

\begin{recollection}
	Let $ \Ccal $ be a symmetric monoidal stable \category, and assume that the tensor product is exact separately in each variable.
	Let $ A $ be an $ \E_{1} $-algebra in $ \Ccal $.
	We say that a cofiber sequence $ X \to Y \to Z $ in $ \Ccal $ is \defn{$ A $-split} if the induced cofiber sequence
	\begin{equation*}
		\begin{tikzcd}
			A \tensor X \arrow[r] & A \tensor Y \arrow[r] & A \tensor Z
		\end{tikzcd}
	\end{equation*}
	is a split cofiber sequence in $ \Mod_{A}(\Ccal) $.
\end{recollection}

In order to characterize pure epimorphisms, we make use of the fact that $ \MGL $-homology of perfect pure motivic spectra vanishes in negative \textit{Chow degree}:

\begin{lemma}[{\cite[Proposition 3.6(2)]{bachmann2022chow}}]
	\label{lemma:vanishing_of_mgl_homology_of_perfect_pures}
	Let $ A \in \SH(k)_{c \geq 0} $ be a connective object of the Chow--Novikov \tstructure, and let $ d,w \in \ZZ $.
	If $ d - 2w < 0 $, then $ \MGL_{d, w}(A) = 0 $.
\end{lemma}

\begin{proposition}\label{proposition:pure_epimorphisms_detected_by_mgl}
	Let $ f \colon B \to A $ be a morphism in $ \Pure(k) $.
	The following are equivalent:
	\begin{enumerate}
	    \item The morphism $ f \colon B \to A$ is a pure epimorphism.

	    \item The morphism $ \MGL \tensor f \colon \MGL \otimes B \to \MGL \otimes A $ is a split epimorphism of $ \MGL $-modules. 
	\end{enumerate}
\end{proposition}

\begin{proof}
	(1)$ \Rightarrow $(2) Write $ C \colonequals \fib(f) $. 
	Since 
	\begin{equation*}
		\MGL \otimes C \to \MGL \otimes B \to \MGL \otimes A
	\end{equation*}
	is a cofiber sequence of $ \MGL $-modules, it is enough to show that the boundary map 
	\begin{equation*}
		\partial \colon \MGL \otimes A \to \Sigma(\MGL \otimes C)
	\end{equation*}
	is zero. 
	Since $ A $ is dualizable, we can identify the homotopy class of $ \partial $ with an element of 
	\begin{equation*}
		\MGL_{-1, 0}(A^{\vee} \otimes C) \period
	\end{equation*}
	By \cref{lemma:perfect_pures_closed_under_monoidal_product_and_duals}, $ A^{\vee} \otimes C $ is again perfect pure.
	Hence \cref{lemma:vanishing_of_mgl_homology_of_perfect_pures} shows that $ \MGL_{-1, 0}(A^{\vee} \otimes C) = 0 $. 

	(2)$ \Rightarrow $(1) By assumption, the boundary map $ A \to \Sigma C $ is zero after tensoring with $ \MGL $. 
	Writing $ \MGL $ as a filtered colimit of Thom spectra of Grassmanians along pure monomorphisms as in \cref{ex:MGL_as_filtered_colimit_along_pure_monomorphisms} and using that $ A $ is compact, we deduce that there exists integers $ d,n \geq 0$ such that the composite 
	\begin{equation*}
		A \to \Sigma C \simeq \Th_{0}(0) \otimes \Sigma C \to \Th_{d}(n+d) \otimes \Sigma C 
	\end{equation*}
	is zero. 
	Passing to the dual of the Thom spectrum, we deduce that the composite 
	\begin{equation*}
		\Th_{d}(n+d)^{\vee} \otimes A \to A \to \Sigma C 
	\end{equation*}
	is zero. 
	Write 
	\begin{equation*}
		B' \colonequals B \times_{A} (\Th_{d}(n+d)^{\vee} \otimes A) \period
	\end{equation*}
	Then we have a commutative diagram 
	\begin{equation*}
		\begin{tikzcd}
			C \arrow[r] \arrow[d, equals] & B' \arrow[r] \arrow[d] \arrow[dr, phantom, very near start, "\lrcorner", xshift=-0.5em, yshift=0.1em] & \Th_{d}(n+d)^{\vee} \otimes A \arrow[d] \\
			C \arrow[r] & B \arrow[r] & A 
		\end{tikzcd}
	\end{equation*}
	where the rows are cofiber sequences. 
	Since the boundary map $\Th_{d}(n+d)^{\vee} \otimes A \to \Sigma C$ is zero, we have 
	\begin{equation*}
		B' \simeq C \oplus (\Th_{d}(n+d)^{\vee} \otimes A) \period
	\end{equation*}
	Since $B'$ is an extension of $\cofib(\Sup^{0, 0} \to \Th_{d}(n+d))^{\vee} \otimes A$ and $B$, we see that $ B' $ is perfect pure. 
	Hence its direct summand $ C $ is also perfect pure, completing the proof. 
\end{proof}


\subsection{Pure sheaves}\label{subsec:pure_sheaves}

We now give a description of $ \SH(k) $ as \acategory of sheaves of spectra on $ \Pure(k) $.
The following is the key definition of this subsection:

\begin{definition}\label{def:pure_sheaf}
	We say a spectral presheaf 
	\begin{equation*}
		X \colon \perfectpures^{\op} \to \spectra
	\end{equation*}
	is a \emph{pure sheaf} if $ X $ sends cofiber sequences of perfect pure motivic spectra to fiber sequences of spectra. 
	We write
	\begin{equation*}
		\ShSigma(\perfectpures; \spectra) \subseteq \PSh(\Pure(k);\Sp)  
	\end{equation*}
	for the full subcategory spanned by the pure sheaves.
\end{definition}

\begin{remark}
	A pure sheaf $X \colon \perfectpures^{\op} \to \spectra$ is in particular additive.
	Our terminology comes from the fact that, as a consequence of \cite[Theorem 2.8]{pstrkagowski2022synthetic}, among all additive functors pure sheaves are characterized by the sheaf property with respect to the Grothendieck pretopology on $\perfectpures$ where covering families consists of a single pure epimorphism. 

	By \cite[Proposition 2.5]{pstrkagowski2022synthetic}, the left adjoint 
	\begin{equation*}
		L \colon \PSigma(\perfectpures; \spectra) \to \ShSigma(\perfectpures; \spectra)
	\end{equation*}
	to the inclusion can be identified with the sheafication functors with respect to this topology. 
	In particular, it is \texact with respect to the \tstructures inherited from that of spectra. 
\end{remark}

\begin{nul}
	The inclusion 
	\begin{equation*}
		\perfectpures \hookrightarrow \SH(k)
	\end{equation*}
	preserves cofiber sequences. 
	Since the target is stable and cocomplete, it follows formally that its left Kan extension defines a symmetric monoidal left adjoint 
	\begin{equation*}
		F \colon \ShSigma(\perfectpures; \spectra) \to \SH(k) \period 
	\end{equation*}
	Its right adjoint 
	\begin{equation*}
		G \colon \SH(k) \to \ShSigma(\perfectpures; \spectra) 
	\end{equation*}
	is given by the spectral Yoneda embedding; i.e., 
	\begin{equation*}
		G(X)(A) \simeq \map_{\SH(k)}(A, X) \period
	\end{equation*}
\end{nul}

\begin{lemma}
	\label{lemma:homotopy_class_of_maps_between_perfect_pures_locally_zero_in_negative_chow_deg}
	Let $ A, B \in \Pure(k) $ be perfect pure and let $ m < 0 $ be an integer.
	Given a map $ \Sigma^{m} B \to A $, there exists a pure epimorphism $ B' \to B$ such that the composite 
	\begin{equation*}
		\Sigma^{m} B' \to \Sigma^{m} B \to A 
	\end{equation*}
	is zero. 
\end{lemma}

\begin{proof}
	Since $ m < 0 $, by \cref{lemma:vanishing_of_mgl_homology_of_perfect_pures} we have that $\MGL_{m, 0}(B^{\vee} \otimes A) = 0 $.
	Thus the composite map 
	\begin{equation*}
		\Sigma^{m} B \to A \to \MGL \otimes A
	\end{equation*}
	is zero. 
	Since $ B $ is compact, we deduce that there exist integers $ n, d \geq 0 $ such that 
	\begin{equation*}
		\Sigma^{m} B \to A \to \Th_{d}(n) \otimes A 
	\end{equation*}
	is zero. 
	By dualizing, the same follows for the composite 
	\begin{equation*}
		\Sigma^{m} (\Th_{d}(n)^{\vee} \otimes B) \to \Sigma^m B \to A \period
	\end{equation*}
	The map $ \Th_{d}(n)^{\vee} \otimes B \to B $ is the required pure epimorphism. 
\end{proof}

Now for the promised description of $ \SH(k) $:

\begin{theorem}
	\label{theorem:shk_is_additive_sheaves_on_perfect_pure_motives}
	The symmetric monoidal functor
	\begin{equation*}
		F \colon \ShSigma(\perfectpures; \spectra) \to \SH(k)
	\end{equation*}
	is an equivalence. 
\end{theorem}

\begin{proof}
	The \category $\ShSigma(\perfectpures; \spectra)$ is generated under colimits and desuspensions by representable presheaves $y(A)$ for $A \in \perfectpures$. 
	These are defined as a sheafication 
	\begin{equation*}
		y(A)(-) \colonequals L(\tau_{\geq 0}F(-, A))
	\end{equation*}
	of the presheaf given by the connective part of the mapping spectrum. 
	By construction as a left Kan extension, the functor $F$ is uniquely determined by the property of being continuous and the requirement that 
	\begin{equation*}
		F(y(A)) \simeq A \in \SH(k) \period 
	\end{equation*}

	We will analyze the unit map 
	\begin{equation*}
	X \to GF(X) \period 
	\end{equation*}
	for some $X \in \ShSigma(\perfectpures; \spectra)$. If $X \simeq y(A)$ is a representable presheaf, by the above discussion this map takes the form 
	\begin{equation*}
	L(\tau_{\geq 0}F(-, A)) \to G(A)(-) \simeq F(-, A) \period 
	\end{equation*}
	Thus, to verify the result in this case we have to show that the map
	\begin{equation*}
	\tau_{\geq 0}F(-, A) \to F(-, A)
	\end{equation*}
	of presheaves of spectra is a sheafication with respect to the pure epimorphism topology.
	This map is a connective cover before sheafication, and thus will remain so after.
	Thus we only have to check that $G(F(A)) \simeq F(-, A)$ is connective as a sheaf. 

	Suppose that $B$ is perfect pure and we have a class in $g \in \pi_{k} G(F(A)) \simeq F(B, A)$ for $k < 0$, which we can identify with a homotopy class of maps
	\begin{equation*}
	g \colon \Sigma^{k} B \to A \period
	\end{equation*}
	By \cref{lemma:homotopy_class_of_maps_between_perfect_pures_locally_zero_in_negative_chow_deg}, we deduce that there exists a pure epimorphism $B' \to B$ such that $g |_{B'} = 0$. It follows that $F(-, A)$ is connective, as needed. 

	Both functors preserve filtered colimits, $F$ as it is a left adjoint and $ G $ as every perfect pure is compact.
	As both are also exact, we deduce that the subcategory of those $X \in \ShSigma(\perfectpures; \spectra)$ such that the unit map is an equivalence is closed under colimits and desuspensions.
	As $\ShSigma(\perfectpures; \spectra)$ is generated under these by $y(A)$ for $A \in \perfectpures$, we deduce that the unit map is an equivalence for any $ X $, so that $F$ is fully faithful. 

	Since the essential image of $F$ is closed under colimits and desuspensions and contains $A \in \SH(k)$, we deduce that $F$ is an equivalence, as needed. 
\end{proof}

\begin{corollary}\label{cor:functors_out_of_SH_in_terms_of_Pure}
	Let $ \Dcal $ be a stable \category which admits small colimits.
	Restriction along the inclusion $ \Pure(k) \subseteq \SH(k) $ defines an equivalence of \categories
	\begin{equation*}
		\FunL(\SH(k),\Dcal) \to \Fun^{\cofib}(\Pure(k),\Dcal) \period
	\end{equation*}
	Here, the right-hand side is the full subcategory of $ \Fun(\Pure(k),\Dcal) $ spanned by the functors that preserve cofiber sequences.
\end{corollary}

\begin{remark}[$ \MGL $-modules]
	As a consequence of \cref{theorem:shk_is_additive_sheaves_on_perfect_pure_motives}, one can deduce a presheaf description of the \category of $ \MGL $-modules.
	This description was already known and is a consequence of the existence of Bondarko's weight structure on $ \MGL $-modules; see the work of Elmanto--Sosnilo \cite[Theorem 2.2.9]{elmanto2022nilpotent}. 
\end{remark}


\section{The weight filtration on complex orientable homology}  
\label{section:weight_filtration_on_complex_orientable_cohomology}

Let $ A $ be an $ \E_1 $-ring spectrum.
In this section, we show that if $ A $ is \textit{complex orientable}, then 
the $ A $-linearized Betti realization functor $ A \tensor \Be(-) \colon \SH(\CC) \to \Mod_A $ refines to a left adjoint
\begin{equation*}
	\WBe(-;A) \colon \SH(\CC) \to \Modpost{A} 
\end{equation*}
valued in modules in filtered spectra over the \textit{Postnikov filtration} on $ A $.
We refer to $ \WBe(-;A) $ as the \textit{filtered Betti realization} functor.
Note that if $ A $ is an ordinary ring, then $ \Modpost{A} $ is coincides with the filtered derived \category of $ A $ (see \Cref{prop:filtered_modules_ordinary_ring});
hence for a complex variety $ X $, the filtered Betti realization $ \WBe(\Sigma_{+}^{\infty} X;A) $ defines a filtration on the complex $ \Cup^{*}(X(\CC);A) $.
In \cref{sec:descent_and_the_Gillet-Soule_filtration}, we explain how to use filtered Betti realization to recover the Deligne--Gillet--Soulé weight filtration on the compactly supported integral Betti cohomology of a complex variety.

In \cref{subsec:background_on_filtered_objects}, we recall some background on filtered objects.
In \cref{subsec:weight_contexts} we set up an abstract framework for using \Cref{cor:functors_out_of_SH_in_terms_of_Pure} to equip the ($ A $-linear) Betti realization of a motivic spectrum with a filtration.
In \cref{subsec:filtered_Betti_homology}, we construct the filtered Betti realization functor $ \WBe(-;A) $; see \Cref{cor:homological_filtered_Betti_realization_exists_for_complex_orientable_rings,cor:filtration_on_Betti_homology_is_exhaustive}.
In \cref{subsec:filtered_Betti realization_with_coefficients_in_an_ordinary_ring}, we unpack our construction in the case of an ordinary ring.
In \cref{subsec:changing_the_coefficients_of_filtered_Betti_realization}, we explain how filtered Betti realization interacts with changing the coefficient ring $ A $.
In \cref{subsec:filtered_etale_realization}, we use the general setup explained in \cref{subsec:weight_contexts} to construct a filtered refinement of the $ \ell $-adic étale realization functor
\begin{equation*}
	\Re_{\ell} \colon \SH(k) \to \Shethyp(\Et_S;\Sp)\ellcomp \period
\end{equation*}
In \cref{subsec:virtual_Euler_characteristics}, we discuss how one can use filtered Betti realization to construct \textit{virtual Euler characteristics} associated to Morava $ \Kup $-theories.

\begin{notation}
	Let $ k $ be a field of exponential characteristic $ e $.
	Throughout this section, we keep the notational convention $ \SH(k) \colonequals \SH(k)[\einv] $
	introduced in \Cref{ntn:SH_means_SH_with_e_inverted}.
\end{notation}


\subsection{Background on filtered objects}\label{subsec:background_on_filtered_objects} 

We begin by reviewing some background on filtered objects in stable \categories.

\begin{notation}
	Let $ \Ccal $ be a stable \category which admits small colimits.
	We write
	\begin{equation*}
		\Fil(\Ccal) \colonequals \Fun(\ZZ^{\op},\Ccal)
	\end{equation*}
	for the \category of \defn{filtered objects} in $ \Ccal $.
	Here we regard $ \ZZ $ as a poset with the usual partial order, so our filtrations are \textit{decreasing}.
	The colimit functor defines a left adjoint $ \colim \colon \Fil(\Ccal) \to \Ccal $.

	If $ \Ccal $ has a \tstructure, then there is a functor $ \tau_{\geq *} \colon \Ccal \to \Fil(\Ccal) $ given by sending an object $ X \in \Ccal $ to its \defn{Postnikov filtration}
	\begin{equation*}
		\begin{tikzcd}[sep=1.5em]
			\cdots \arrow[r] & \tau_{\geq n} X \arrow[r] & \tau_{\geq n + 1} X \arrow[r] & \cdots \comma
		\end{tikzcd}
	\end{equation*}
	see \cite[Construction 3.3.7]{arXiv:2007.02576}. Moroever:   
	\begin{enumerate}
	\item If the \tstructure is right complete, then $ \colim \tau_{\geq *} \equivalent \id_{\Ccal} $, so that the Postnikov filtration is \emph{exhaustive}. 
	\item If the \tstructure is left complete, then $\lim \tau_{\geq *} \equivalent 0$, so that the Postnikov filtration is \emph{complete}. 
	\end{enumerate}
	Note that the functor $ \tau_{\geq *} \colon \Ccal \to \Fil(\Ccal) $ is additive, but generally \textit{not} exact. 
\end{notation}

\begin{notation}
	Via Day convolution, the addition on $ \ZZ^{\op} $ and the tensor product of spectra assemble into a symmetric monoidal structure 
	\begin{equation*} 
		\tensor \colon \FilSp \cross \FilSp \to \FilSp 
	\end{equation*}
	defined by
	\begin{equation*}
		(X_* \tensor Y_*)_n \colonequals \colim_{a + b \geq n} X_a \tensor Y_b \period
	\end{equation*}
\end{notation}

\begin{nul}
	With respect to the Day convolution symmetric monoidal structure, the functor
	\begin{equation*}
		\tau_{\geq *} \colon \Sp \to \FilSp 
	\end{equation*}
	is lax symmetric monoidal.
	In particular, for any $ \E_n $-ring spectrum $ A $, the filtered spectrum $ \tau_{\geq *}(A) $ acquires a natural $ \E_n $-ring structure.
	Moreover, the functor $ \tau_{\geq *} \colon \Sp \to \FilSp $ refines to a functor
	\begin{equation*}
		\Mod_A = \Mod_A(\Sp) \to \Modpost{A} \comma
	\end{equation*} 
	which we also denote by $ \tau_{\geq *} $.
	We also write
	\begin{equation*}
		\colim \colon \Modpost{A} \to \Mod_A
	\end{equation*}
	for the induced functor.
\end{nul}

\begin{definition}
	\label{definition:diagonal_t_structure_on_filtered_spectra} 
	We say a filtered spectrum $F_{*}X$ is \emph{diagonal connective} if for all $n \in \ZZ$ we have $F_{n}X \in \spectra_{\geq n}$. 
	This determines a unique \tstructure on filtered spectra which we call the \emph{diagonal \tstructure}. 
\end{definition} 

\begin{remark}
	The diagonal \tstructure is compatible with the symmetric monoidal structure on filtered spectra. 
	Since any filtered spectrum of the form $ \tau_{\geq *} A $ is diagonal connective, the \category 
	\begin{equation*}
		\Mod_{\tau_{\geq *} A}(\FilSp)
	\end{equation*}
	of modules in filtered spectra inherits a unique \tstructure for which the forgetful functor is \texact.
	We also refer to this \tstructure as the \emph{diagonal \tstructure}. 
\end{remark}

When $ A = \HR $ is the Eilenberg--MacLane spectrum associated to an ordinary commutative ring, $ \Modpost{\HR} $ recovers the filtered derived \category of $ R $: 

\begin{notation}
	Let $ R $ be an ordinary commutative ring.
	We write $ \Dfil(R) $ for the \categorical enhancement of the filtered derived category of $ R $.
\end{notation}

\begin{proposition}\label{prop:filtered_modules_ordinary_ring}
	Let $ R $ be an ordinary commutative ring.
	There are natural symmetric monoidal equivalences
	\begin{equation*}
		\Modpost{\HR} \equivalent \Fil(\Dcal(R)) \equivalent \Dfil(R) \period
	\end{equation*}
\end{proposition}

\begin{proof}[Proof sketch]
	Note that since $ \HR $ only has a nontrivial homotopy group in degree $ 0 $, the filtered spectrum $ \tau_{\geq *}(\HR) $ is given by
	\begin{equation*}
		\begin{tikzcd}[sep=1.5em]
			\cdots \arrow[r, equals] & \HR \arrow[r, equals] & \HR \arrow[r] & 0 \arrow[r] & 0 \arrow[r] & \cdots \comma
		\end{tikzcd}
	\end{equation*}
	where the nonzero terms are in filtration degrees $ \leq 0 $. 
	With this identification, the left-hand equivalence follows from the natural symmetric monoidal equivalence $ \Mod_{\HR} \equivalent \Dcal(R) $ and a filtered variant of the Schwede--Shipley theorem.
	(See \cite[Proposition A.2.1]{arXiv:2212.09964} for the graded variant of the Schwede--Shipley theorem.)
	The right-hand equivalence is the content of \cite[Theorem 2.6]{MR3806745}.
\end{proof}


\subsection{Weight contexts}\label{subsec:weight_contexts} 

We now describe a general method of equipping a colimit-preserving functor defined on the stable motivic category with additional structure. 

\begin{definition}\label{definition:homological_weight_context}
	Let $ k $ be a field.
	A \emph{weight context} consists of the following data: 
	\begin{enumerate}
	    \item Stable \categories $ \Ccat $ and $ \Dcat $ which admit small colimits. 

	    \item A colimit-preserving functor $ U \colon \Dcat \to \Ccat $. 

	    \item An additive functor $ T \colon \Ccat \to \Dcat$ along with an equivalence $ U \circ T \simeq \id_{\Ccat} $.

	    \item A colimit-preserving functor $ \M \colon \SH(k) \to \Ccat$. 
	\end{enumerate}
	A \emph{solution} to a weight context is a functor $ \WM $ making the following triangle commute
	\begin{equation*}
		\begin{tikzcd}
			& \Dcat \arrow[d, "U"] \\
			\SH(k) \arrow[r, "\M"'] \arrow[ur, dotted, "\WM"] & \Ccat \period
		\end{tikzcd}
	\end{equation*}
\end{definition}

\begin{nul}
	In the setting of \cref{definition:homological_weight_context}, we can think of $\Dcat$ as the \category of objects of $ \Ccat $ equipped with additional structure and of $U$ as the forgetful functor. 
	One should think of the functor $T$, going the other way, as equipping an object $c \in \Ccat$ with a ``trivial structure''. 
	A solution to a weight context should be thought of as a way of functorially equipping objects of the form $\M(X)$ with additional structure. 
	Note that we do not assume that $T$ is exact, and indeed in most examples it is not. 
\end{nul}

The following result is a trivial application of our new description of $ \SH(k) $ explained in \Cref{theorem:shk_is_additive_sheaves_on_perfect_pure_motives}.
However, it turns out that this result has many useful applications.

\begin{theorem}\label{theorem:homological_weight_context_has_unique_solution_if_t_is_mgl_exact}
	Suppose that we are given a weight context as in \cref{definition:homological_weight_context} with the following property:
	\begin{enumerate}[label={$(\ast)$}]
		\item If $ X \to Y \to Z $ is a cofiber sequence in $ \Pure(k) $, then 
		\begin{equation*}
			T(\M(X)) \to T(\M(Y)) \to T(\M(Z))
		\end{equation*}
		is a cofiber sequence in $ \Dcal $. 
	\end{enumerate}
	Then, there exists a unique solution $ \WM \colon \SH(k) \to \Dcat$ satifying the following properties: 
	\begin{enumerate}
	    \item The functor $ \WM $ preserves colimits. 

	    \item The restriction of $ \WM $ to $ \Pure(k) $ is given by $ T \circ \M \colon \Pure(k) \to \Dcal $.
	\end{enumerate}
\end{theorem}

\begin{proof}
	By \Cref{cor:functors_out_of_SH_in_terms_of_Pure}, the assumptions guarentee that $ T \circ \M \colon \Pure(k) \to \Dcal $ uniquely extends to a colimit-preserving functor $ \WM \colon \SH(k) \to \Dcal $.
\end{proof}


\subsection{Filtered Betti realization}\label{subsec:filtered_Betti_homology}

In this subsection, we study \textit{Betti realization} with coefficients in a ring spectrum $ A $.
The main result of this subsection is that when $ A $ is complex orientable, Betti realization comes equipped with a natural exhaustive filtration (\Cref{cor:homological_filtered_Betti_realization_exists_for_complex_orientable_rings,cor:filtration_on_Betti_homology_is_exhaustive}). 

\begin{definition}
	Let $ A $ be an $ \E_1 $-ring spectrum.
	The \defn{$ A $-linear Betti realization} functor is the composite
	\begin{equation*}
		\begin{tikzcd}[sep=4.5em]
			\Be(-;A) \colon \SH(\CC) \arrow[r, "\Be"] & \Sp \arrow[r, "A \tensor (-)"] & \Mod_A \period
		\end{tikzcd}
	\end{equation*}
\end{definition}

\begin{nul}\label{nul:A-linear_homological_Betti_weight_context}
	Let $ A $ be an $ \E_1 $-ring spectrum.
	Then we have a weight context
	\begin{equation*}
		\begin{tikzcd}[sep=3em]
			& \Modpost{A} \arrow[d, shift left, "\colim"] \\ 
			\SH(\CC) \arrow[r, "\Be(-;A)"'] \arrow[ur, dotted] & \Mod_A  \period \arrow[u, "\tau_{\geq *}"{near start}, shift left]
		\end{tikzcd}
	\end{equation*}
	In the notation of \Cref{definition:homological_weight_context}, $ U = \colim $ and $ T = \tau_{\geq *} $.
\end{nul}

\begin{recollection}[complex orientations]
	\label{rec:complex_orientations}
	Let $ A $ be an $ \E_{1} $-ring spectrum.
	A \defn{complex orientation} of $ A $ is a morphism $ \MU \to A $ of associative algebras in the homotopy category $ \hSp $ of spectra. 
	We say that $ A $ is \defn{complex orientatable} if there exists a complex orientation of $ A $. 
	We refer the reader to \cites{Lurie:Chromatic-4}{Lurie:Chromatic-6}[\S4.1]{MR860042} for more background on complex orientations.
\end{recollection}

\begin{example}
	\begin{enumerate}
		\item If $ R $ is an ordinary ring, then there is a natural map of $ \E_{\infty} $-rings $ \MU \to \HR $.
		In particular, $ \HR $ is complex orientable.

		\item The complex $ \Kup $-theory spectrum $ \KU $ has a canonical complex orientation.	

		\item For each prime $ p $ and integer $ n \geq 0 $, the height $ n $ Morava $ \Kup $-theory $ \Kup(n) $ has a canonical complex orientation.
	\end{enumerate}
\end{example}

In order to check the hypotheses of \Cref{theorem:homological_weight_context_has_unique_solution_if_t_is_mgl_exact} for $ \Be(-;A) $, we need the following lemma.
 
\begin{lemma}
	\label{lemma:MU_zero_maps_are_zero_on_any_complex_orientable_ring} 
	Let $ A $ be a complex orientable $ \E_{1} $-ring and let $ f \colon X \to Y $ be a map of spectra such that $ \MU \otimes f $ is zero. 
	Then $ A \otimes f $ is zero as a map of $ A $-modules. 
\end{lemma}

\begin{proof}
	By the extension of scalars adjunction, it is enough to show that the map of spectra 
	\begin{equation}\label{eq:unit_tensor_f}
	  	\begin{tikzcd}[sep=3.5em]
	  		X \equivalent \Sup^0 \tensor X \arrow[r] & A \otimes X \arrow[r, "A \tensor f"] & A \otimes Y
	  	\end{tikzcd} 
	\end{equation}
	induced by the unit of $ A $ is zero. 
	Choose a complex orientation $ \phi \colon \MU \to A $.
	The map \eqref{eq:unit_tensor_f} factors in the homotopy category $ \hSp $ as 
	\begin{equation*}
		\begin{tikzcd}[sep=3.5em]
	  		X \arrow[r] & \MU \otimes X \arrow[r, "\MU \tensor f"] & \MU \otimes Y \arrow[r, "\phi \tensor Y"] & A \otimes Y
	  	\end{tikzcd}
	\end{equation*}
	and the middle is zero by assumption. 
\end{proof}

\begin{corollary}\label{cor:MU-split_cofiber_sequences_are_A-split}
	Let $ A $ be a complex orientable $ \E_{1} $-ring and let
	\begin{equation}\label{eq:generic_cofiber_sequence}
		X \to Y \to Z
	\end{equation}
	be a cofiber sequence of spectra.
	If \eqref{eq:generic_cofiber_sequence} is $ \MU $-split, then \eqref{eq:generic_cofiber_sequence} is $ A $-split.
\end{corollary}

\begin{proof}
	We need to show that if the boundary map $ \MU \tensor Z \to \MU \tensor \Sigma X $ is zero, then the boundary map $ A \tensor Z \to A \tensor \Sigma X $ is also zero.
	This is immediate from \Cref{lemma:MU_zero_maps_are_zero_on_any_complex_orientable_ring}. 
\end{proof}

\begin{lemma}\label{lem:MU-linear_homological_Betti_realization_splits_pure_cofiber_sequences}
	Let $ A $ be a complex orientable $ \E_1 $-ring spectrum.
	Let $ X \to Y \to Z $ be an $ \MGL $-split cofiber sequence in $ \SH(\CC) $.
	Then the null sequence
	\begin{equation*}
		\begin{tikzcd}
			\Be(X;A) \arrow[r] & \Be(Y;A) \arrow[r] & \Be(Z;A) 
		\end{tikzcd}
	\end{equation*}
	is a split cofiber sequence in $ \Mod_A $.
\end{lemma}

\begin{proof}
	Since $ \Be(-;A) = A \tensor \Be(-) $, by \Cref{cor:MU-split_cofiber_sequences_are_A-split} it suffices to show that the cofiber sequence of $ \MU $-modules
	\begin{equation}\label{eq:MU_tensor_cofiber_sequence}
		\begin{tikzcd}
			\MU \tensor \Be(X) \arrow[r] & \MU \tensor \Be(Y) \arrow[r] & \MU \tensor \Be(Z)
		\end{tikzcd}
	\end{equation}
	is split.
	Since Betti realization is symmetric monoidal and $ \Be(\MGL) \equivalent \MU $, the cofiber sequence of \eqref{eq:MU_tensor_cofiber_sequence} is obtained by applying Betti realization to the cofiber sequence of $ \MGL $-modules 
	\begin{equation*}
		\begin{tikzcd}
			\MGL \tensor X \arrow[r] & \MGL \tensor Y \arrow[r] & \MGL \tensor Z \comma
		\end{tikzcd}
	\end{equation*}
	which is split by assumption.
\end{proof}

\begin{example}\label{ex:A-linear_Betti_preserves_cofiber_sequences_of_pures}
	Let $ A $ be a complex orientable $ \E_1 $-ring spectrum, and let $ X \to Y \to Z $ be a cofiber sequence in $ \Pure(\CC) $.
	Combining \Cref{proposition:pure_epimorphisms_detected_by_mgl,lem:MU-linear_homological_Betti_realization_splits_pure_cofiber_sequences}, we see that
	\begin{equation*}
		\begin{tikzcd}
			\Be(X;A) \arrow[r] & \Be(Y;A) \arrow[r] & \Be(Z;A) 
		\end{tikzcd}
	\end{equation*}
	is a split cofiber sequence in $ \Mod_A $.  
\end{example}

As a consequence, for a complex orientable connective $ \E_1 $-ring $ A $, the weight context of \Cref{nul:A-linear_homological_Betti_weight_context} has a solution.
More generally, any weight context based on $ A $-linear Betti realization has a solution.

\begin{proposition}\label{prop:solutions_to_weight_contexts_based_on_Betti_realization}
	Let $ A $ be a complex orientable $ \E_1 $-ring.
	Then any weight context of the form
	\begin{equation*}
		\begin{tikzcd}[sep=3.5em]
			& \Dcal \arrow[d, shift left, "U"] \\ 
			\SH(\CC) \arrow[r, "\Be(-;A)"'] \arrow[ur, dotted] & \Mod_A \period \arrow[u, "T", shift left]
		\end{tikzcd}
	\end{equation*}
	has a unique solution $ \WM \colon \SH(\CC) \to \Dcal $ satifying the following properties: 
	\begin{enumerate}
	    \item The functor $ \WM $ preserves colimits. 

	    \item If $ X \in \SH(\CC) $ is perfect pure, then $ \WM(X) \simeq T(\Be(X;A)) $. 
	\end{enumerate}
\end{proposition}

\begin{proof}
	By \Cref{theorem:homological_weight_context_has_unique_solution_if_t_is_mgl_exact}, it suffices to show that if $ X \to Y \to Z $ is a cofiber sequence in $ \Pure(\CC) $, then 
	\begin{equation*}
		\begin{tikzcd}
			T(\Be(X;A)) \arrow[r] & T(\Be(Y;A)) \arrow[r] & T(\Be(Z;A)) 
		\end{tikzcd}
	\end{equation*}
	is a cofiber sequence in $ \Dcal $.
	By \Cref{ex:A-linear_Betti_preserves_cofiber_sequences_of_pures}, the cofiber sequence 
	\begin{equation*}
		\begin{tikzcd}
			\Be(X;A) \arrow[r] & \Be(Y;A) \arrow[r] & \Be(Z;A) 
		\end{tikzcd}
	\end{equation*}
	in $ \Mod_A $ is split.
	Since $ T \colon \Mod_A \to \Dcal $ is additive, $ T $ preserves this split cofiber sequence.
\end{proof}

\begin{corollary}\label{cor:homological_filtered_Betti_realization_exists_for_complex_orientable_rings}
	Let $ A $ be an $ \E_1 $-ring spectrum.
	If $ A $ is complex orientable, then there exists a unique left adjoint
	\begin{equation*}
		\begin{tikzcd}
			 \WBe(-;A) \colon \SH(\CC) \arrow[r] &  \Modpost{A}  
		\end{tikzcd}
	\end{equation*}
    such that for $ X \in \Pure(\CC) $, we have
    \begin{equation*}
        \WBe(X;A) \simeq \tau_{\geq *} \Be(X;A) \period
    \end{equation*}
\end{corollary}

\begin{proof}
	Apply \Cref{prop:solutions_to_weight_contexts_based_on_Betti_realization} to the weight context \cref{nul:A-linear_homological_Betti_weight_context}.
\end{proof}
	

\begin{definition}[filtered Betti realization]\label{def:filtered_Betti_realization}
	Let $ A $ be a complex orientable $ \E_1 $-ring spectrum.
	We call the functor
	\begin{equation*}
		\WBe(-;A) \colon \SH(\CC) \to \Modpost{A} 
	\end{equation*}
	of \Cref{cor:homological_filtered_Betti_realization_exists_for_complex_orientable_rings} the \defn{$ A $-linear filtered Betti realization} functor. 
\end{definition}

Pleasantly, this filtration is exhaustive:

\begin{corollary}
\label{cor:filtration_on_Betti_homology_is_exhaustive}
	Let $ A $ be a complex orientable $ \E_1 $-ring spectrum.
	Then the triangle of \categories and left adjoints
	\begin{equation*}
		\begin{tikzcd}
			\SH(\CC) \arrow[rr, "\WBe(-;A)"] \arrow[dr, "\Be(-;A)"'] & & \Modpost{A} \arrow[dl, "\colim"] \\ 
			& \Mod_A &
		\end{tikzcd}
	\end{equation*}
	canonically commutes. 
\end{corollary}

\begin{proof}
    Both of the functors $ \SH(\CC) \to \Mod_A $ in the diagram preserve colimits.
    Moreover, by \Cref{cor:homological_filtered_Betti_realization_exists_for_complex_orientable_rings} they agree on $ \Pure(\CC) \subseteq \SH(\CC) $. 
    Thus the conclusion follows from \Cref{cor:functors_out_of_SH_in_terms_of_Pure}.
\end{proof}

We conclude by recording that the filtered Betti realization is compatible with \tstructures.
The relevant \tstructure on the motivic side is the Chow--Novikov \tstructure of \cite{bachmann2022chow}, and on the filtered module side is the diagonal \tstructure: 

\begin{lemma}
	\label{lemma:the_filtered_betti_realization_is_t_exact} 
	Let $ A $ be a complex orientable $ \E_1 $-ring spectrum.
	The filtered Betti realization 
    \begin{equation*}
		\WBe(-;A) \colon \SH(\CC) \to \Modpost{A} 
	\end{equation*}
	is right \texact with respect to the Chow--Novikov \tstructure on motivic spectra and the diagonal \tstructure on filtered spectra; that is, filtered Betti realization preserves connectivity. 
\end{lemma}

\begin{proof}
	By definition, the connective part of the Chow--Novikov \tstructure is generated under colimits and extensions by perfect pure motivic spectra. 
	Thus, it is enough to show that for $ X $ perfect pure 
	\begin{equation*}
		\WBe(X; A) \simeq \tau_{\geq *}(\Be(X))
	\end{equation*}
	is connective, which is clear. 
\end{proof}

For the next result, recall that an object $ X $ of a stable \category with \tstructure $ \Ccal $ is \emph{$ \infty $-connective} if $ X \in \bigcap_{n \in \ZZ} \Ccal_{\geq n} $.
Also recall that the \tstructure on $ \Ccal $ is \textit{left separated} if $ \bigcap_{n \in \ZZ} \Ccal_{\geq n} = 0 $.

\begin{corollary}
	\label{corollary:filtered_betti_realization_inverts_chow_novikov_infty_connective_maps}
	Let $ A $ be a complex orientable $ \E_1 $-ring spectrum.
	The filtered Betti realization 
    \begin{equation*}
		\WBe(-;A) \colon \SH(\CC) \to \Modpost{A} 
	\end{equation*}
	inverts maps of motivic spectra which are $ \infty $-connective with respect to the Chow--Novikov \tstructure. 
\end{corollary}

\begin{proof}
	Since $\WBe(-; A)$ is exact, it is enough to show that if $ X $ is $ \infty $-connective with respect to the Chow--Novikov \tstructure, then $\WBe(X; A) = 0$. 
	By \cref{lemma:the_filtered_betti_realization_is_t_exact}, we deduce that $\WBe(X; A)$ is $ \infty $-connective with respect to the diagonal \tstructure, so that $\WBe(X; A)$ is levelwise $ \infty $-connective. 
	Since the standard \tstructure on spectra is left separated, it follows that $ \WBe(X; A) = 0 $. 
\end{proof}


\subsection{The case of an ordinary ring}
\label{subsec:filtered_Betti realization_with_coefficients_in_an_ordinary_ring}

We now unpack the filtered Betti realization in the case of an ordinary ring.

\begin{notation}
	If $ R $ is an ordinary commutative ring, we simply write
	\begin{equation*}
		\Be(-;R) \colon \SH(\CC) \to \Dcal(R) 
	\end{equation*}
	for $ \Be(-;\HR) $.
	Note that the functor $ \Be(-;R) $ is the unique symmetric monoidal left adjoint with the property that for any smooth $ \CC $-scheme $ X $, we have
	\begin{equation*}
		\Be(\Sigma_{+}^{\infty} X;R) \equivalent \Cup_{*}(X(\CC);R) \period
	\end{equation*} 
\end{notation}

An important feature is that Betti realization with coefficients in an ordinary ring factors through modules over motivic cohomology:

\begin{observation}[{$ \Be(-;R) $ factors through $ \MR $-modules}]\label{obs:ZZ-linear_Betti_realization_factors_through_HZZ-modules}
	Let $ R $ be an ordinary commutative ring.
    Since Betti realization $ \Be \colon \SH(\CC) \to \Sp $ is symmetric monoidal and $ \Be(\MR) \equivalent \HR $, the $ R $-linear Betti realization functor factors through $ \MR $-modules in $ \SH(\CC) $.
    That is, $ R $-linear Betti realization refines to a unique symmetric monoidal left adjoint
    \begin{equation*}
        \Mod_{\MR}(\SH(\CC)) \to \Mod_{\HR}(\Sp) \equivalent \Dcal(R)
    \end{equation*}
    fitting into a commutative square
    \begin{equation*}
        \begin{tikzcd}[sep=3em]
            \SH(\CC) \arrow[r, "\Be"] \arrow[d, "\MR \tensor (-)"'] & \Sp \arrow[d, "\HR \tensor (-)"] \\ 
            \Mod_{\MR}(\SH(\CC)) \arrow[r, dotted] & \Mod_{\HR}(\Sp) \period
        \end{tikzcd}
    \end{equation*}
    We also denote this refinement by $ \Be(-;R) \colon \Mod_{\MR}(\SH(\CC)) \to \Dcal(R) $.
\end{observation}

In this case, \Cref{def:filtered_Betti_realization} specializes to the following:

\begin{example}
	Let $ R $ be an ordinary commutative ring.
	Since the Eilenberg--Mac\-Lane spectrum $ \HR $ admits a canonical complex orientation, there is a filtered Betti realization functor
	\begin{equation*}
		\begin{tikzcd}[sep=3em]
			\SH(\CC) \arrow[rr, "\WBe(-;R)"] & & \Modpost{\HR} \arrow[r, "\sim"{yshift=-0.25ex}] & \Fil(\Dcal(R))
		\end{tikzcd}
	\end{equation*}
	Here the right-hand equivalence is provided by \Cref{prop:filtered_modules_ordinary_ring}.
\end{example}


Again, filtered Betti realization with coefficients in an ordinary ring factors through modules over motivic cohomology:

\begin{observation}[{$ \WBe(-;R) $ factors through $ \MR $-modules}]\label{obs:ZZ-linear_filtered_Betti_realization_factors_through_HZZ-modules}
    Let $ R $ be an ordinary commutative ring.
    In light of \Cref{obs:ZZ-linear_Betti_realization_factors_through_HZZ-modules}, the filtered $ R $-linear Betti realization functor $ \WBe(-;R) $ refines to a unique left adjoint
    \begin{equation*}
    	\Mod_{\MR}(\SH(\CC)) \to \Fil(\Dcal(R))
    \end{equation*}
    fitting into a commutative triangle
    \begin{equation*}
        \begin{tikzcd}[sep=3em]
            & \SH(\CC) \arrow[dr, "\WBe(-;R)"] \arrow[dl, "\MR \tensor (-)"'] & & \\ 
            \Mod_{\MR}(\SH(\CC)) \arrow[rr, dotted] & & \Fil(\Dcal(R)) \period
        \end{tikzcd}
    \end{equation*}
    We also denote this refinement by $ \WBe(-;R) \colon \Mod_{\MR}(\SH(\CC)) \to \Fil(\Dcal(R)) $.
\end{observation}


\subsection{Changing the coefficients of filtered Betti realization}\label{subsec:changing_the_coefficients_of_filtered_Betti_realization}

Let $ \phi \colon A \to B $ be a morphism of complex orientable $ \E_1 $-rings.
In this subsection, we produce a comparison natural transformation
\begin{equation*}
	\post{B} \tensorlimits_{\post{A}} \WBe(-;A) \to \WBe(-;B) 
\end{equation*} 
and show that this natural tranformation is an equivalence if $ \phi $ is flat (\Cref{cor:Be_fil_of_flat_maps}).
To start, we need to analyze the interaction between Postnikov filtrations and tensor products.

\begin{observation}
	Let $ \phi \colon A \to B $ be a morphism of $ \E_1 $-rings.
	Then the square
	\begin{equation*}
		\begin{tikzcd}[column sep=4.5em]
			\Mod_B \arrow[r] \arrow[d, "\tau_{\geq *}"'] & \Mod_A  \arrow[d, "\tau_{\geq *}"] \\ 
			\Modpost{B} \arrow[r] & \Modpost{A}
		\end{tikzcd}
	\end{equation*}
	commutes.
	Here the horizontal functors are the forgetful functors.
	Passing to horizontal left adjoints, there is an exchange transformation filling the square
	\begin{equation*}
		\begin{tikzcd}[row sep=4.5em, column sep=8em]
			\Mod_A \arrow[r, "B \tensor_A (-)"] \arrow[d, "\tau_{\geq *}"'] & \Mod_B  \arrow[d, "\tau_{\geq *}"] \\ 
			\Modpost{A} \arrow[r, "\post{B} \tensorlimits\limits_{\post{A}} (-)"'] \arrow[ur, phantom, "\scriptstyle \Ex_{\phi}" above left, "\Longrightarrow" sloped] & \Modpost{B} \period
		\end{tikzcd}
	\end{equation*}
\end{observation}

\begin{construction}[comparison morphism]\label{construction:comparison_morphism_between_filtrations_for_different_cpx_orientable_spectra}
	Let $ \phi \colon A \to B $ be a morphism of complex orientable connective $ \E_1 $-rings.
	Define a natural transformation
	\begin{equation*}
		c_{\phi} \colon \post{B} \tensorlimits_{\post{A}} \WBe(-;A) \longrightarrow \WBe(-;B)
	\end{equation*}
	of functors $ \SH(\CC) \to \Modpost{B} $ as follows.
	Note that since $ \post{B} \tensorlimits_{\post{A}} \WBe(-;A) $ and $ \WBe(-;B) $ are both left adjoints, by the equivalence
	\begin{equation*}
		\begin{tikzcd}
			\FunL(\SH(\CC),\Modpost{B}) \arrow[r, "\sim"{yshift=-0.25ex}] & \Fun^{\cofib}(\Pure(\CC),\Modpost{B})
		\end{tikzcd}
	\end{equation*}
	of \Cref{cor:functors_out_of_SH_in_terms_of_Pure}, it suffices to construct the restriction $ \restrict{c_{\phi}}{\Pure(\CC)} $ to perfect pure motivic spectra.
	For this, we take the natural transformation
	\begin{equation*}
		\begin{tikzcd}[sep=6em]
			\phantom{\restrict{c_{\phi}}{\Pure(\CC)}} \displaystyle \post{B} \tensorlimits_{\post{A}} \tau_{\geq *}(\Be(-;A)) \arrow[r, "\Ex_{\phi}\Be(-;A)"] & \tau_{\geq *}(B \tensor_A \Be(-;A)) 
		\end{tikzcd}
	\end{equation*}
	induced by the exchange transformation.
\end{construction}

For flat ring maps, the exchange transformation is an equivalence:

\begin{lemma}\label{lem:tensoring_with_a_flat_morphism_preserves_Postnikov_filtrations}
	Let $ \phi \colon A \to B $ be a morphism of $ \E_1 $-rings.
	If $ \phi $ is flat, then the exchange transformation
	\begin{equation*}
		\Ex_{\phi} \colon \post{B} \tensorlimits_{\post{A}} \tau_{\geq *}(-) \longrightarrow  \tau_{\geq *}(B \tensor_A (-))
	\end{equation*}
	is an equivalence of functors $ \Mod_A \to \Modpost{B} $.
\end{lemma}

\begin{proof}
	Since $ \phi $ is flat, the left adjoint $ B \tensor_A (-) \colon \Mod_A \to \Mod_B $ is \texact \cite[Theorem 7.2.2.15]{lurieha}.
	Hence for each $ M \in \Mod_A $ and $ n \in \ZZ $, the natural map
	\begin{equation*}
		B \tensor_A \tau_{\geq n}(M) \longrightarrow \tau_{\geq n}(B \tensor_A M)
	\end{equation*}
	is an equivalence.
\end{proof}

\begin{corollary}\label{cor:Be_fil_of_flat_maps}
	Let $ \phi \colon A \to B $ be a morphism of complex orientable $ \E_1 $-rings.
	If $ \phi $ is flat, then the comparison natural transformation
	\begin{equation*}
		c_{\phi} \colon \post{B} \tensorlimits_{\post{A}} \WBe(-;A) \longrightarrow \WBe(-;B)
	\end{equation*}
	is an equivalence of functors $ \SH(\CC) \to \Modpost{B} $.
\end{corollary}

\begin{proof}
	Since both $ \post{B} \tensorlimits_{\post{A}} \WBe(-;A) $ and $ \WBe(-;B) $ are left adjoints, by \Cref{cor:functors_out_of_SH_in_terms_of_Pure} it suffices to show that $ c_{\phi} $ is an equivalence when restricted to $ \Pure(\CC) $.
	The claim now follows from the definitions of $ \WBe(-;A) $ and $ \WBe(-;B) $ combined with \Cref{lem:tensoring_with_a_flat_morphism_preserves_Postnikov_filtrations}.
\end{proof}

\begin{example}\label{ex:QQ-linearized_BeZZfil_is_BeQQfil}
	The comparison natural transformation
	\begin{equation*}
		\QQ \tensor_\ZZ \WBe(-;\ZZ) \to \WBe(-;\QQ)
	\end{equation*}
	is an equivalence of functors $ \SH(\CC) \to \Fil(\Dcal(\QQ)) $.
\end{example}


\subsection{Filtered étale realization}\label{subsec:filtered_etale_realization}

Let $ k $ be a field and $ \ell \neq \characteristic(k) $ a prime.
In \cref{definition:etale_realization}, we recalled Bachmann's construction an $ \ell $-adic étale realization functor 
\begin{equation*}
	\Re_{\ell} \colon \SH(k) \to \Shethyp(\Et_S;\Sp)\ellcomp \period
\end{equation*}
In this subsection, we show that the complex orientable variants of this functor have a canonical lift to filtered sheaves. 

\begin{definition}
	\label{definition:algebra_in_etale_sheaves_complex_orientable} 
	Let $ k $ be a field and $ \ell \neq \characteristic(k) $ a prime.
	We say that $A \in \Alg(\Shethyp(\Et_S;\Sp)\ellcomp)$ is \defn{complex orientable} if there exists a map of associative algebras
	\begin{equation*}
		\Re_{\ell}(\MGL) \to A
	\end{equation*}
	in the homotopy category of $ \Shethyp(\Et_S;\Sp)\ellcomp $. 
\end{definition}

\begin{remark}
	Write
	\begin{equation*}
		R \colon \Shethyp(\Et_S;\Sp)\ellcomp \to \SH(k)
	\end{equation*}
	for the right adjoint to $ \Re_{\ell} $.
	The condition that $ A $ is complex orientable in the sense of \cref{definition:algebra_in_etale_sheaves_complex_orientable} is equivalent to the condition that the motivic spectrum $ R(A) \in \Alg(\SH(k)) $ representing $ A $-linear étale cohomology is orientable as a motivic spectrum. 
\end{remark}

\begin{nul} 
	Recall that one says that $X \in \Shethyp(\Et_k;\Sp)$ is \emph{coconnective} if for every $E \in \Et_{k}$, the spectrum $ X(E) $ is coconnective.
	This is a coconnective part of a unique \tstructure which we call the \emph{standard \tstructure}; see \cite[\S1.3.2]{SAG}.
	The heart can be described as the category
	\begin{equation*}
		\Shethyp(\Et_k;\Sp)^{\heartsuit} \simeq \Shethyp(\Et_k; \Ab) \simeq \Sh_{\et}(\Et_k; \Ab), 
	\end{equation*}
	of étale sheaves of abelian groups on $ k $.
	A map $X \to Y$ of hypercomplete sheaves is an equivalence if and only if for each $i \in \ZZ$, the induced map $\pi_{i}^{\heartsuit} X \to \pi_{i}^{\heartsuit} Y$ is an isomorphism. 
\end{nul} 

\begin{definition}
	Let $ k $ be a field of exponential characteristic $ e $, let $ \ell \neq e $ be a prime, and let $ A \in \Alg(\Shethyp(\Et_k;\Sp)\ellcomp) $.
	The \defn{$ A $-linear étale realization} functor is the composite
	\begin{equation*}
		\begin{tikzcd}[sep=4.5em]
			\Re_{\ell}(-;A) \colon \SH(k) \arrow[r, "\Re_{\ell}"] & \Shethyp(\Et_k;\Sp)\ellcomp \arrow[r, "A \tensor (-)"] & \Mod_A(\Shethyp(\Et_k;\Sp)\ellcomp) \period
		\end{tikzcd}
	\end{equation*}
\end{definition}

\begin{proposition}
	\label{proposition:existence_of_the_a_linear_etale_realization} 
	Let $ k $ be a field of exponential characteristic $ e $ and $ \ell \neq e $ a prime.
	Let $ A \in \Alg(\Shethyp(\Et_k;\Sp)\ellcomp) $ be complex orientable.
	Then there exists a unique left adjoint
    \begin{equation*}
        \WRe_{\ell}(-;A) \colon \SH(k) \longrightarrow \Fil( \Shethyp(\Et_k; \Sp)\ellcomp)
    \end{equation*}
    such that for $ X \in \Pure(k) $ and any $ n \in \ZZ $ we have
    \begin{equation*}
    	\WRe_{\ell} (X;A) \simeq (\tau_{\geq *} (\Re_{\ell}(X;A))\ellcomp
    \end{equation*}
	the $ \ell $-completion of the Whitehead cover of $\Re_{\ell}(X;A)$ with respect to the standard \tstructure. 
\end{proposition}

\begin{proof}
	By \cref{theorem:homological_weight_context_has_unique_solution_if_t_is_mgl_exact}, it suffices to show that if $ X \to Y \to Z $ is a cofiber sequence in $ \Pure(k) $, then 
	\begin{equation*}
		\begin{tikzcd}
			\tau_{\geq *}(\Re_{\ell}(X;A)) \to \tau_{\geq *}(\Re_{\ell}(Y;A)) \to \tau_{\geq *}(\Re_{\ell}(Z;A))
		\end{tikzcd}
	\end{equation*}
	is a cofiber sequence in $ \Fil( \Shethyp(\Et_k; \Sp)\ellcomp) $.
	Since there exists a map $\Re_{\ell}(\MGL) \to A$ of the algebras in the homotopy category of $ \Shethyp(\Et_k; \Sp)\ellcomp $, the same argument as in \cref{lemma:MU_zero_maps_are_zero_on_any_complex_orientable_ring} shows that 
	\begin{equation*}
		\Re_{\ell}(X;A) \to \Re_{\ell}(Y;A) \to \Re_{\ell}(Z)
	\end{equation*}
	is a split cofiber sequence, hence preserved by all additive functors, such as $\tau_{\geq *}$. 
\end{proof}

\begin{definition}
	\label{definition:filtered_etale_realization}
	We call the left adjoint functor
	\begin{equation*}
	    \WRe_{\ell}(-;A) \colon \SH(k) \to \Fil( \Shethyp(\Et_k; \Sp)\ellcomp)
	\end{equation*}
	of \cref{proposition:existence_of_the_a_linear_etale_realization} the \defn{filtered étale realization}. 
\end{definition}

\begin{remark}
	\label{remark:filtered_etale_realization_also_inverts_chow_infty_connective_maps}
	Since the standard \tstructure on \emph{hypercomplete} sheaves is left separated, same argument as in the Betti case covered in \cref{corollary:filtered_betti_realization_inverts_chow_novikov_infty_connective_maps} shows that filtered étale realization inverts Chow--Novikov $ \infty $-connective maps. 
\end{remark}


\subsection{Virtual Euler characteristics}\label{subsec:virtual_Euler_characteristics}

An old conjecture of Serre, first solved by Deligne using the weight filtration, is the existence of \textit{virtual Euler characteristics}. 
These are invariants 
\begin{equation*}
	a_{i}(X;\QQ) \in \ZZ
\end{equation*}
of a complex variety $ X $ uniquely determined by the following properties:
\begin{enumerate}
    \item If $ X $ is smooth and proper, then
    \begin{equation*}
   		a_{i}(X;\QQ) = \dim_{\QQ} \Hrm^{i}(X(\CC); \QQ) \period
    \end{equation*}

    \item If $ X $ is a variety with an open subvariety $U \subseteq X$ with closed complement $Z \subseteq X$, then 
    \begin{equation*}
    	a_{i}(X;\QQ) = a_{i}(U;\QQ) + a_{i}(Z;\QQ) \period
    \end{equation*}
\end{enumerate}
Over a field of characteristic zero, these virtual Euler characteristics can be defined using Bittner's presentation of the Grothendieck ring of varieties \cite{bittner2004universal}. 

In terms of the weight filtration on compactly supported cochains $\Cc^{*}(X(\CC); \QQ)$, the virtual Euler characteristic is given by the explicit formula
\begin{equation*}
\label{equation:explicit_formula_for_rational_euler_characteristic}
	a_{i}(X;\QQ) = (-1)^{i} \chi_{\QQ}( \gr_{-i} \Cc^{*}(X(\CC); \QQ)) \period
\end{equation*}
Here, 
\begin{equation*}
	\gr_{i} \Cc^{*}(X(\CC); \QQ) \colonequals \cofib\paren{\Wup_{i+1}\Cc^{*}(X(\CC); \QQ) \to \Wup_{i}\Cc^{*}(X(\CC); \QQ)}
\end{equation*}
is the $ i $-th graded piece of the weight filtration, and $\chi_{\QQ}$ denotes the Euler characteristic of a perfect $\QQ$-module in spectra defined by the difference between the dimension in even odd degrees:
\begin{equation*}
	\chi_{\QQ}(P) \colonequals \dim_{\QQ} \pi_{2*}(P) - \dim_{\QQ} \pi_{2*+1}(P) \period
\end{equation*}
Thus, analogous to the way that Khovanov homology categorifies the Jones polynomial \cite{khovanov2000categorification}, the weight filtration can be thought of as the ``geometry'' behind the virtual Euler characteristics. 

Besides ordinary cohomology, there are other complex oriented cohomology theories which behave like fields, known as the \emph{Morava $\Kup$-theories}.
For each prime $ p $ and integer $ n \geq 1 $, we write $ \Kup(n) $ for the \textit{height $ n $ Morava $ \Kup $-theory} at the (implicit) prime $ p $. 
In many ways, despite the fact that their ring of coefficients
\begin{equation*}
	\Kup(n)_{*} \simeq \FF_{p}[v_{n}^{\pm 1}] \quad \text{with} \quad \deg(v_n) = 2p^n - 2 
\end{equation*}
is of positive characteristic, these cohomology theories behave like objects of characteristic zero; see \cites{MR4419631}{HopkinsLurie:ambidexterity}.
This makes Morava $ \Kup $-theories useful, for example, in problems involving orientations of orbifolds, as in Abouzaid and Blumberg's breakthrough work on the Arnold conjecture in symplectic geometry \cites{abouzaid2021arnold}. 

Since the ring of coefficients $\Kup(n)_{*}$ forms a graded field and is concentrated in even degrees, analogously to the case of rational cohomology one can define an Euler characteristic of a perfect $ \Kup(n) $-module $P$ by a formula 
\begin{equation*}
	\chi_{\Kup(n)}(P) \colonequals \dim_{\Kup(n)_{*}}(\pi_{2*}(P)) - \dim_{\Kup(n)_{*}}(\pi_{2*+1}(P)) \period
\end{equation*}
When applied to $\Kup$-cohomology of spaces, these Morava--Euler characteristics satisfy a host of useful properties, and at odd primes can be used to recover an interesting invariant of spaces called \textit{homotopy cardinality}; see the work of Yanovski \cite{yanovski2023homotopy}. 

Since the $ \E_1 $-ring spectra $ \Kup(n) $ are complex orientable, one can show that the Euler characteristics defined by 
\begin{equation*}
	a_{i}(X;\Kup(n)) \colonequals \dim_{\FF_{p}} \Kup(n)^{i}(X(\CC))
\end{equation*}
when $ X $ is smooth and proper satisfy Bittner's relation.
It follows that they extend to a \emph{virtual Morava--Euler characteristic} defined on all complex varieties. 
We now show that the weight filtration on $ \Kup(n) $-cohomology provided by \cref{cor:homological_filtered_Betti_realization_exists_for_complex_orientable_rings} can be thought of as the ``geometry'' behind these virtual Morava--Euler characteristics. This also had the advantage of applying to étale cohomology, including in positive characteristic, where Bittner's theorem is not known to hold, see \cref{remark:etale_morava_euler_characteristics}. 

\begin{notation}
	To keep the notation similar to the rational case, we write 
	\begin{equation*}
		\Wup_{*} \Cc^{*}(X(\CC); \Kup(n)) \colonequals \WBe(\Mc(X); \Kup(n))
	\end{equation*}
	for the weight filtration on compactly supported $ \Kup(n) $-linear cochains, by which we mean the filtered $ \Kup(n) $-linear Betti realization of the compactly supported motive $ \Mc(X) $ introduced in \cref{definition:motive_and_motivec_of_a_variety}. 
	This is a $\tau_{\geq *} \Kup(n)$-module in filtered spectra. 
\end{notation}

\begin{definition}
	\label{definition:virtual_morava_euler_characteristic}
	Let $ X $ be a complex variety.
	The \emph{virtual Morava--Euler characteristic} of $ X $ is defined by 
	\begin{equation*}
		a_{i}(X;\Kup(n)) =  (-1)^{i} (\chi_{\fieldp} \gr_{-i} \Cc^{*}(X(\CC); \Kup(n))) \comma
	\end{equation*}
	the $\fieldp$-Euler characteristic of the $i$-th graded piece of the weight filtration on compactly supported $ \Kup(n) $-linear cochains.
\end{definition}

Note that since the associated graded of $\tau_{\geq *} \Kup(n) $ is given by the homotopy groups $\pi_{*} \Kup(n) \simeq \fieldp[v_{n}^{\pm 1}]$, each graded piece of the weight filtration on $ \Kup(n) $-linear cohomology is in particular a module over $\pi_{0} \Kup(n) \simeq \fieldp$, so that \cref{definition:virtual_morava_euler_characteristic} is well-defined. 

\begin{theorem}
The virtual Morava--Euler characteristic of \cref{definition:virtual_morava_euler_characteristic} has the following properties: 
\begin{enumerate}
    \item If $ X $ is smooth and proper, then
    \begin{equation*}
    	a_{i}(X;\Kup(n)) \colonequals \dim_{\FF_{p}} \Kup(n)^{i}(X(\CC)) \period
    \end{equation*} 

    \item If $ X $ is a variety with an open subvariety $U \subseteq X$ with closed complement $Z \subseteq X$, then 
    \begin{equation*}
    	a_{i}(X;\Kup(n)) = a_{i}(U;\Kup(n)) + a_{i}(Z;\Kup(n)) \period
    \end{equation*}
\end{enumerate}
\end{theorem}

\begin{proof}
	If $ X $ is smooth and proper, then as observed in \cref{example:motive_of_smooth_variety}, the motive of $ X $ can be identified with the Thom spectrum $\Th_{X}(-\Tup_{X})$ of the negative tangent bundle. 
	It follows from the definition of the filtered Betti realization on pure motives as an associated graded of the Postnikov filtration that 
	\begin{equation*}
		\gr_{-i} \Cc^{*}(X(\CC); \Kup(n)) \simeq \Sigma^{-i} \pi_{-i} (\Be(\Th_{X}(-\Tup_{X})) \otimes K). 
	\end{equation*}
	It follows that if $ X $ is smooth and proper, then 
	\begin{align*}
		a_{i}(X;\Kup(n)) &= \dim_{\fieldp} \Kup(n)_{-i}(\Be(\Th_{X}(-\Tup_{X}))) \\ 
		&= \dim_{\fieldp} \Kup(n)_{-i}(\Th_{X(\CC)}(-\Tup_{X(\CC)})) \\ 
		&= \dim_{\fieldp} \Kup(n)^{i}(X(\CC)) \comma
	\end{align*}
	where the second equality is the fact that the Betti realization takes Thom spectra to Thom spectra, and the last one is Atiyah duality. This gives the first claimed property.

	The second property is an immediate consequence of the localization cofiber sequence 
	\begin{equation*}
		\motive(U) \to \motive(X) \to \motive(Z) 
	\end{equation*}
	of \cref{lemma:cofiber_sequence_of_motives_associated_to_open_closed_decomposition}, exactness of filtered Betti realization, and the fact that the Euler characteristic is additive in cofiber sequences. 
\end{proof}

\begin{remark}[étale Morava--Euler characteristics]
	\label{remark:etale_morava_euler_characteristics}
	One can also define analogues of Morava $ \Kup $-theories in the context of étale realization; for example, as étale realizations of Voevodsky' algebraic Morava $ \Kup $-theories%
	\footnote{Since we only consider $ \ell $-adic cohomology in the étale context, and the Hopkins--Morel theorem holds away from the characteristic, the $ \ell $-local Morava $ \Kup $-theories can be constructed over any field as an appropriate localization of a quotient of $ \MGL $, which is analogous to how they are constructed in topology}. 
	These will also be complex orientable, and a variation on \cref{definition:virtual_morava_euler_characteristic} will also yield a Morava--Euler characteristic in the context of étale cohomology. Since étale Morava $ \Kup $-theories have received comparatively little attention in the literature compared to their topological cousins, we decided against writing this section at this level of generality. 
\end{remark}


\section{Descent and the Gillet--Soulé filtration}
\label{sec:descent_and_the_Gillet-Soule_filtration}

In this section, we show that filtration on compactly supported cohomology given by the filtered Betti realization functor can be calculated through an appropriate hypercover.
As a consequence, we deduce that our filtration on integral cohomology of a complex variety agrees with the one constructed by Gillet--Soulé in \cite{gillet1996descent}. 
The key geometric input needed to establish the hypercover formula are Kelly's \ldhtopology on schemes \cite{kelly2017voevodsky}, and a result of Geisser on \ldhhypercovers \cite[Theorem 1.2]{geisser2014homological}.

In \cref{subsec:background_on_the_cdh-topology_and_ldh_topology}, we review background on the \ldhtopology.
In \cref{subsec:hyperdescent_for_orientable_Borel--Moore_homology}, we prove that Borel--Moore homology with coefficients in an orientable motivic spectrum satisfies \ldhhyperdescent; see \Cref{theorem:hypercovers_of_varieties_are_colimit_diagrams_in_shk}.
In \cref{subsec:the_weight_filtration_on_Borel--Moore_homology_through_hypercovers}, we use \ldhhypercovers to calculate the weight filtration on Borel--Moore homology; see \Cref{theorem:weights_on_bm_homology_of_proper_variety_calculated_through_an_ldh_hypercover}.
In \cref{subsec:filtration_on_cohomology_and_the_comparison_with_the_Gillet--Soule_filtration}, we use our perspective on filtrations to recover the Gillet--Soulé weight filtration on the compactly supported integral cochains on a complex variety; wee \Cref{theorem:gillet_soule_filtration_comparison}.


\subsection{Background on the \cdhtopology and \texorpdfstring{$\ell$dh}{ℓdh}-topology}\label{subsec:background_on_the_cdh-topology_and_ldh_topology}

We briefly review the necessary background on the \cdh- and \ldh-topologies.
For more background, see \cites[\S2]{MR4238259}{kellycdh} and \cite{kelly2017voevodsky}, respectively. 

\begin{recollection}[{\cdp- and \cdhtopologies}]
	\hfill
	\begin{enumerate} 
		\item A family of morphisms of schemes $ \{p_i \colon X'_i \to X\}_{i \in I} $ is \defn{completely decomposed} if for each $ x \in X $ there exists an $ i \in I $ and point $ x' \in p_i^{-1}(x) $ such that the induced map of residue fields $\kappa(x) \to \kappa(x')$ is an isomorphism.

		\item The \defn{\cdptopology} on the category of qcqs schemes is defined as follows: a sieve on a qcqs scheme $ X $ is a \defn{\cdpcovering sieve} if and only if it contains a completely decomposed family $ \{p_i \colon X'_i \to X\}_{i \in I} $ where each $ p_i $ is proper and of finite presentation.

		\item The \defn{\cdhtopology} is the topology generated by the \cdptopology and the Nisnevich topology.
	\end{enumerate}
	Also recall that every motivic spectrum satisfies \cdhdescent \cite[Corollary 6.25]{MR3570135}.
	Moreover, for a field $ k $, every \cdhsheaf over $ k $ is automatically a \cdhhypersheaf \cite[Corollary 2.4.16]{MR4238259}.
\end{recollection}

\begin{recollection}[{\ldhtopology}]
	Let $ \ell $ be a prime number.
	\begin{enumerate} 
		\item A morphism of schemes $ p \colon X' \to X $ is an \defn{\fpslcover} if $ p $ is finite flat and surjective, and $ p_{*} \Ocal_{X'} $ is a free $ \Ocal_X $-module of rank prime to $ \ell $.

		\item The \defn{\ldhtopology} is the topology generated by the \cdhtopology and \fpslcovers.
	\end{enumerate}	
\end{recollection}

\begin{definition}
	Let $ X $ be a scheme and let $p \colon \Deltaop \to \Sch_{/X}$ be a simplicial $ X $-scheme.
	We say that $ p $ is a \emph{\cdhhypercover} (respectively, \emph{\ldhhypercover}) if for each $i \geq 0$, the induced map 
	\begin{equation*}
		X_{i} \to (\cosk^{X}_{i-1}X_{\bullet})
	\end{equation*}
	is a \cdhcover (respectively, \ldhcover).
\end{definition}

\begin{remark}
	Unwrapping the definition of the coskeleton, we see that $ p $ is a hypercover if and only if for each $ i \geq 0 $, the matching maps 
	\begin{equation*}
		X_{0} \to X \comma \qquad
		X_{1} \to X_{0} \times_{X} X_{0} \comma \qquad
		X_{2} \to \cdots
	\end{equation*}
	are coverings.
\end{remark}


\subsection{Hyperdescent for orientable Borel--Moore homology}\label{subsec:hyperdescent_for_orientable_Borel--Moore_homology}

In this subsection, we show that Borel--Moore homology with respect to an orientable motivic spectrum satisfies \ldhhyperdescent. 

\begin{notation}
	Throughout this subsection, we fix a base field $ k $ of exponential characteristic $ e $ and a prime $\ell \neq e $.
\end{notation}

\begin{notation}
	In \cref{definition:motive_and_motivec_of_a_variety} we attached to a variety $p \colon X \to \Spec(k)$ a motivic spectrum
	\begin{equation*}
		\motive(X) \colonequals p_{!}(\Unit_{X}) \period 
	\end{equation*}
	By \cref{corollary:motives_of_varieties_are_dualizable_away_from_the_characteristic}, this motivic spectrum is dualizable away from the characteristic, and we write
	\begin{equation*}
		\motive(X)^{\vee}_{(\ell)} \in \SH(k)_{(\ell)}
	\end{equation*}
	for the $ \ell $-local monoidal dual.
\end{notation}

\noindent The rest of this subsection is be devoted to the proof of the following result. 

\begin{theorem}
	\label{theorem:hypercovers_of_varieties_are_colimit_diagrams_in_shk}
	If $X_{\bullet} \to X$ is an \ldhhypercover of $ k $-schemes, then the natural map 
	\begin{equation}\label{eq:ldh-hyperdescent_map}
		\varinjlim_{\Deltaop} \motive(X_{\bullet} )^{\vee}_{(\ell)} \to \motive(X)^{\vee}_{(\ell)} \period
	\end{equation}
	is an $ \MGL $-local equivalence; that is, the map \eqref{eq:ldh-hyperdescent_map} becomes an equivalence after tensoring with $ \MGL $. 
	In particular, the map \eqref{eq:ldh-hyperdescent_map} is $ \infty $-connective with respect to the Chow--Novikov \tstructure. 
\end{theorem}

\begin{remark}
	\label{remark:for_cdh_hypercovers_the_descent_holds_after_inverting_the_characteristic}
	Note that a \cdhhypercover is an \ldhhypercover for all $ \ell $.
	Hence, if $X_{\bullet} \to X$ is a \cdhhypercover, then the $ \ell $-localization  in \cref{theorem:hypercovers_of_varieties_are_colimit_diagrams_in_shk} can be replaced by localization away from the exponential characteristic $ e $.
	That is, the map 
	\begin{equation*}
		\varinjlim_{\Deltaop} \motive(X_{\bullet} )[\einv]^{\vee} \to \motive(X)[\einv]^{\vee} 
	\end{equation*}
	is also an $ \MGL $-local equivalence.  
\end{remark}

\noindent The proof of \cref{theorem:hypercovers_of_varieties_are_colimit_diagrams_in_shk} is somewhat involved and occupies the remainder of this subsection.
Our argument can be informally divided into three parts: 
\begin{enumerate}
    \item First, we show that \cref{theorem:hypercovers_of_varieties_are_colimit_diagrams_in_shk} follows from an \ldhhyperdescent statement in Borel--Moore $ \MGL $-homology.
    This is \cref{lemma:reduction_step_in_proof_of_ldh_hyperdescent}. 

    \item We then use the homotopy \tstructure to prove connectivity estimates on Borel--Moore homology of varieties with respect to a connective, orientable homology theory.
    This is \cref{lemma:vanishing_of_borel_moore_homology_of_arbitrary_variety_for_connective_motivic_spectra}. 
    \item Finally, we use Spitzweck's calculation of the slices of $ \MGL $ and our connectivity estimates to show that \ldhhyperdescent for motivic cohomology implies \ldhhyperdescent for $ \MGL $.
    For motivic cohomology the needed hyperdescent statement was proven by Geisser \cite[Theorem 1.2]{geisser2014homological}, and later generalized by Kelly \cite[Theorem 4.0.13]{MR3673293}.
\end{enumerate}

\begin{convention}
	For the rest of this section, we work $ \ell $-locally, and all motivic spectra are implicitly localized at $ \ell $.
	We begin with part (1), where it is convenient to employ the following notation. 
\end{convention}

\begin{notation}
	\label{notation:dual_motive_tensored_with_motivic_spectrum}
	If $E$ is an $ \ell $-local motivic spectrum, we write 
	\begin{equation*}
		E^{\BM}_{X} \colonequals E \otimes \motive(X)^{\vee}_{(\ell)}. 
	\end{equation*}
	This is justified by 
	\cref{observation:bm_homology_and_c_cohomology_using_the_motive}, since we have equivalences
	\begin{align*}
		\pi_{p, q}(E^{\BM}_{X}) &\simeq [\Sup^{p, q}, E \otimes \motive(X)^{\vee}_{(\ell)}] \\
		&\simeq [\Sigma^{p, q} \motive(X), E] \\
		&\simeq E^{\BM}_{p, q}(X) \period
	\end{align*}
	Note that if $ X $ is smooth and projective, then \Cref{observation:motive_of_smooth_projective_variety_is_dual_to_the_suspension_spectrum} shows that
	\begin{equation*}
		E^{\BM}_{X} \simeq E \otimes \Sigma_{+}^{\infty} X \period
	\end{equation*}
\end{notation} 

\begin{lemma}
	\label{lemma:reduction_step_in_proof_of_ldh_hyperdescent} 
	Assume that the following condition is satisfied: 
	\begin{enumerate}[label={$(\ast)$}]
		\item For any $ k $-scheme $ X $, any \ldhhypercover $X_{\bullet} \to X$, and any $s \in \ZZ$, the canonical comparison map of spectra 
	    \begin{equation*}
	   		\varinjlim \map_{\SH(k)}(\Sup^{2s, s}, \MGLBM_{X_{\bullet}}) \to \map_{\SH(k)}(\Sup^{2s, s}, \MGLBM_{X})
	    \end{equation*}
	    is an equivalence.
	\end{enumerate}
	Then \cref{theorem:hypercovers_of_varieties_are_colimit_diagrams_in_shk} holds. 
\end{lemma}

\begin{proof}
	In terms of \cref{notation:dual_motive_tensored_with_motivic_spectrum}, \cref{theorem:hypercovers_of_varieties_are_colimit_diagrams_in_shk} is equivalent to showing that the natural map of $ \MGL $-modules
	\begin{equation}
	\label{equation:mgl_bm_comparison_map_in_reduction_lemma}
		\varinjlim_{\Deltaop} \MGL_{X_{\bullet}} \to \MGL_{X}
	\end{equation}
	is an equivalence.
	Since all $ \MGL $-local equivalences are Chow--Novikov $ \infty $-connective \cite[Corollary 3.17]{bachmann2022chow}, the second part of \cref{theorem:hypercovers_of_varieties_are_colimit_diagrams_in_shk} follows from the first.

	By \cite[Theorem 2.2.9]{elmanto2022nilpotent}, the spectral Yoneda embedding induces an equivalence between the \category of $ \MGL $-modules and spectral presheaves on the thick subcategory generated by modules of the form $\MGL \otimes S$, where $S \in \Pure(k)$. 
	Thus, \eqref{equation:mgl_bm_comparison_map_in_reduction_lemma} is an equivalence if and only if for any $ S \in \Pure(k) $, the map
	\begin{equation}
		\label{equation:colimit_equivalence_in_proof_of_mgl_ldh_hypercover_statement}
		\varinjlim \map_{\MGL}(\MGL \otimes S, \MGL_{X_{\bullet}}) \to \map_{\MGL}(\MGL \otimes S, \MGL_{X})
	\end{equation}
	is an equivalence.
	Since $ \MGL $ is orientable, $ \MGL $-linear perfect pure motives are generated as a thick subcategory by modules of the form 
	\begin{equation*}
		\Sigma^{2(d+s), d+s} \MGL_{Y} \comma
	\end{equation*}
	where $ Y $ is a smooth projective variety of dimension $d$ and $s \in \ZZ$. For any variety $Z$, we then have 
	\begin{align*}
		\map_{\MGL}(\Sigma^{2(d+s), d+s} \MGL_{Y} , \MGL_{Z}) &\simeq \map_{\MGL}(\Sigma^{2s, s} \MGL, \MGL_{Y \times Z}) \\ 
		&\simeq \map_{\SH(k)}(\Sup^{2s, s}, \MGL_{Y \times Z})
	\end{align*}
	Thus, to show that 
	\eqref{equation:colimit_equivalence_in_proof_of_mgl_ldh_hypercover_statement} is an equivalence it is enough to show that for each smooth projective variety $ Y $ and integer $ s $, the map
	\begin{equation}
		\label{equation:rewritten_colimit_equivalence_in_proof_of_mgl_ldh_hypercover_statement}
		\varinjlim \map_{\SH(k)}(\Sup^{2s, s}, \MGL_{Y \times X_{\bullet}}) \to \map_{\SH(k)}(\Sup^{2s, s}, \MGL_{Y \times X})
	\end{equation}
	is an equivalence. 
	Since $Y \times X_{\bullet} \to Y \times X$ is again an \ldhhypercover, the conclusion follows from assumption $ (*) $. 
\end{proof}

We now proceed with the second step of the proof, which is a vanishing result for Borel--Moore homology of varieties. 
The vanishing holds for motivic spectra that are connective with respect to the \textit{homotopy \tstructure}, which we now recall.

\begin{recollection}[homotopy \tstructure]
	\label{recollection:homotopy_t_structure_and_vanishing_of_connective_cohomology}
	Write 
	\begin{equation*}
		\SH(k)_{\geq 0} \subseteq \SH(k)
	\end{equation*}	
	for the full subcategory generated under colimits and extensions by $\Sigma^{p, q} \Sigma_{+}^{\infty} X$ for $X \in \Sm_{k}$ and $p > q$.
	The subcategory $ \SH(k)_{\geq 0} $ defines the connective part of a unique \tstructure on $ \SH(k) $ called the \emph{homotopy \tstructure}.
	This \tstructure has the following two properties, both proven in \cite[Corollary 2.4]{hoyois2015algebraic}:
	\begin{enumerate}
	    \item The homotopy \tstructure is left complete.
	    That is, the natural functor
	    \begin{equation*}
	    	\SH(k) \to \lim \left( 
	    	\begin{tikzcd}
	    		\cdots \arrow[r] & \SH(k)_{\leq 2} \arrow[r, "\tau_{\leq 1}"] & \SH(k)_{\leq 2} \arrow[r, "\tau_{\leq 0}"] & \SH(k)_{\leq 0}
	    	\end{tikzcd}
	    	\right)
	    \end{equation*}
	    is an equivalence.
	    Hence the homotopy \tstructure is left separated, i.e., $\bigcap _{d \in \ZZ} \SH(k)_{\geq d} = 0$.
	
	    \item If $E$ is connective, then for any smooth variety $ X $, for $ p > q + \dim(X) $ we have
		\begin{equation}
			\label{equation:vanishing_of_cohomology_of_smooths_for_connective_motivic_spectra}
			E^{p,q}(X) \simeq [\Sigma^{-p, -q} \Sigma_{+}^{\infty} X, E] = 0 \period
		\end{equation}
	\end{enumerate}
\end{recollection} 

\begin{lemma}
	\label{lemma:vanishing_of_borel_moore_homology_of_arbitrary_variety_for_connective_motivic_spectra}
	Let $E \in \SH(k)[\einv]$ be connective motivic spectrum that admits a structure of an $ \MGL $-module. 
	Then for any variety $ X $ and integers $ p < q $, we have
	\begin{equation*}
		E^{\BM}_{p, q}(X) = 0 \period 
	\end{equation*}
\end{lemma}

\begin{proof}
	Let us first assume that $ k $ is perfect. 
	Recall that $ E^{\BM}_{p, q}(X) \simeq [\Sigma^{p, q} \motive(X), E] $ and write $\Ccat \subseteq \SH(k)[\einv]$ for full subcategory of motivic spectra $ A $ such that for all $ p < q $, we have
	\begin{equation*}
		[\Sigma^{p, q} A, E] = 0 \period
	\end{equation*}
	It suffices to show that $ \Ccal $ satisfies the hypotheses of \cref{lemma:subcategory_of_motives_containing_smooth_projectives_has_all_motives}. 
	
	Since $ \Ccal $ is closed under extensions, fibers, and retracts, it is enough to show that if $ X $ is a smooth projective $ k $-scheme, then $ \motive(X)[\einv] \in \Ccal $. 
	In this case, \cref{example:motive_of_smooth_variety} shows that
	\begin{align*}
		\motive(X) &\simeq \Th_{X}(-\Tup_{X}) \period \\ 
	\intertext{Hence we have a string of isomorphisms}
		E^{\BM}_{p, q}(X) &\simeq [\Sigma^{p, q} (\Th_{X}(-\Tup_{X})), E] \\
		&\simeq [\MGL \otimes \Sigma^{p, q} (\Th_{X}(-\Tup_{X})), E]_{\MGL} \comma
	\intertext{where the final term denotes homotopy classes of maps of $ \MGL $-modules.
	Write $ d \colonequals \dim(X) $; using the Thom isomorphism we can further rewrite the right-hand side as}
		[\MGL \otimes \Sigma^{p-2d, q-d} (\Sigma_{+}^{\infty} X), E]_{\MGL} &\simeq [\Sigma^{p-2d, q-d} (\Sigma_{+}^{\infty} X), E] \\
		&\simeq E^{2d-p, d-q}(X) \period
	\end{align*}
	As observed in \cref{recollection:homotopy_t_structure_and_vanishing_of_connective_cohomology}, the right-hand side vanishes when $2d-p > d-q+d$, which translates to $p < q$, as needed. 

	If $ k $ is not perfect, then write $ k \to k' $ for the perfection of $ k $.
	As in the proof of \cref{corollary:motives_of_varieties_are_dualizable_away_from_the_characteristic}, we reduce to the perfect case by using the equivalence $\SH(k)[\einv] \simeq \SH(k')[\einv]$ of \cite[Corollary 2.1.7]{MR3999675}. 
\end{proof}

We now proceed with the third step of the proof, which reduces from $ \MGL $-homology to motivic cohomology.
We need to make use of the slice tower, which we now recall. 

\begin{recollection}[{effective covers \& slice filtration}]
	\label{recollection:slice_filtration}
	Let $ E $ be a motivic spectrum and $ r \in \ZZ $.
	We write $ \f_r E $ for the \defn{$ r $-th effective cover} of $ E $.
	These effective covers give rise to a functorial filtration
	\begin{equation*}
		\cdots \to \f_{1} E \to \f_{0} E \to \f_{-1} E \to \cdots \to  E \period
	\end{equation*}
	We write 
	\begin{equation*}
		\s_{r} E \colonequals \cofib(\f_{r+1} E \to \f_{r} E)
	\end{equation*}
	for the \emph{$r$-th slice}.
	We also write 
	\begin{equation*}
		\c_{r} E \colonequals \cofib(\f_{r+1} E \to E) \period 
	\end{equation*}
\end{recollection} 

\begin{recollection}[{slices of $ \MGL $}]
	\label{recollection:slices_of_mgl}
	The spectrum $ \MGL $ is $ 0 $-effective, i.e., $ \f_0 \MGL \equivalent \MGL $ \cite[Corollary 3.2]{spitzweck2010relations}.
	Assuming the Hopkins--Morel equivalence, Spitzweck calculated the slices of $ \MGL $ as 
	\begin{equation}
	\label{equation:formula_for_slices_of_mgl}
		s_{r} \MGL \simeq \motiviccoh(\pi_{2r} \MU) \comma
	\end{equation}
	where on the right-hand side we have the motivic cohomology spectrum associated to $ \pi_{2r} \MU $, which is a free abelian group of finite rank \cite[Theorem 4.7]{spitzweck2010relations}. 
	The Hopkins--Morel equivalence was subsequently proven by Hoyois away from the characteristic \cite{hoyois2015algebraic}, showing that \eqref{equation:formula_for_slices_of_mgl} holds $ \ell $-locally. 
\end{recollection} 

\begin{lemma}
	\label{lemma:borel_moore_mgl_homology_approximate_by_cofiber_of_slice}
	Let $ X $ be a $ k $-variety. 
	Then for $ p < q+r $, the canonical map 
	\begin{equation*}
		(\MGL_{(\ell)})^{\BM}_{p, q}(X) \to (\c_{r} \MGL_{(\ell)})^{\BM}_{p, q}(X) 
	\end{equation*}
	is an isomorphism. 
\end{lemma}

\begin{proof}
	By a result of Spitzweck \cite[{Proof of Theorem 4.7}]{spitzweck2010relations}, the $ (r+1) $-st effective cover $ \f_{r+1} \MGL_{(\ell)} $ is a colimit of spectra of the form $\Sigma^{2(r+1), r+1} \MGL_{(\ell)}$.
	In particular, $ \f_{r+1} \MGL_{(\ell)} $ is $(r+1)$-connective in the homotopy \tstructure.
	The desired result now follows from the cofiber sequence
	\begin{equation*}
		\f_{r+1} \MGL_{(\ell)} \to \MGL_{(\ell)} \to \c_{r} \MGL_{(\ell)}
	\end{equation*}
	and \cref{lemma:vanishing_of_borel_moore_homology_of_arbitrary_variety_for_connective_motivic_spectra}.
\end{proof}

We now complete the promised argument. 

\begin{proof}[{Proof of \cref{theorem:hypercovers_of_varieties_are_colimit_diagrams_in_shk}}]
	Throuthout the proof, we implicitly work $ \ell $-locally and drop the $ \ell $-localization from notation.
	By \cref{lemma:reduction_step_in_proof_of_ldh_hyperdescent}, it is enough to show that if $X_{\bullet} \to X$ is an \ldhhypercover and $s \in \ZZ$, then the natural map
	\begin{equation*}
		\varinjlim \map_{\SH(k)}(\Sup^{2s, s}, \MGLBM_{X_{\bullet}}) \to \map_{\SH(k)}(\Sup^{2s, s}, \MGLBM_{X})
	\end{equation*}
	is an equivalence.
	As the standard \tstructure on spectra is right complete, a diagram $F \colon \Ccat^{\triangleright} \to \spectra$ is a colimit if and only if for each $ m \in \ZZ $, the diagram 
	\begin{equation*}
		(\tau_{\geq m} \circ F) \colon \Ccat^{\triangleright} \to \spectra_{\geq m}
	\end{equation*}
	of $m$-coconnective spectra is a colimit.
	Thus, the map 
	\begin{equation*}
		\varinjlim \map_{\SH(k)}(\Sup^{2s, s}, \MGLBM_{X_{\bullet}}) \to  \map_{\SH(k)}(\Sup^{2s, s}, \MGLBM_{X}) 
	\end{equation*}
	is an equivalence if and only if for each $m \in \ZZ$, the induced map of spectra
	\begin{equation}
		\label{equation:r_truncated_mgl_hypercover_colimit_comparison_map}
		\varinjlim (\tau_{\geq m} \map_{\SH(k)}(\Sup^{2s, s}, \MGLBM_{X_{\bullet}})) \to \tau_{\geq m} \map_{\SH(k)}(\Sup^{2s, s}, \MGLBM_{X}) 
	\end{equation}
	has an $(m+1)$-connective cofiber.
	By \cref{lemma:borel_moore_mgl_homology_approximate_by_cofiber_of_slice}, for all $ k $-varieties $ Z $ and integers $ k < r - s $, the map 
	\begin{equation*}
		\pi_{k} \map_{\SH(k)}(\Sup^{2s, s}, \MGLBM_{Z}) \to \pi_{k} \map_{\SH(k)}(\Sup^{2s, s}, \c_{r} \MGLBM_Z)
	\end{equation*}
	is an isomorphism.
	Thus, if $r > m+s$, then the map \eqref{equation:r_truncated_mgl_hypercover_colimit_comparison_map} is equivalent to the map
	\begin{equation*}
		\varinjlim (\tau_{\geq m} \map_{\SH(k)}(\Sup^{2s, s}, (\c_{r} \MGL)_{X_{\bullet}}^{\BM} )) \to \tau_{\geq m} \map_{\SH(k)}(\Sup^{2s, s}, (\c_{r} \MGL)_{X}^{\BM} ) \period
	\end{equation*}
	Thus it suffices to show that for each $r \in \ZZ$ the map 
	\begin{equation*}
		\varinjlim \map_{\SH(k)}(\Sup^{2s, s}, (\c_{r} \MGL)_{X_{\bullet}}^{\BM}) \to  \map_{\SH(k)}(\Sup^{2s, s}, (\c_{r} \MGL)_{X}^{\BM}) 
	\end{equation*}
	is an equivalence. 
	In other words, we have to show \ldhhyperdescent for $\c_{r} \MGL$-Borel--Moore homology of varieties. 

	As we observed in \cref{recollection:slices_of_mgl}, by a result of Spitzweck the slices of algebraic cobordism are given by suspensions of motivic cohomology associated to finitely generated abelian groups. It follows that $\c_{r} \MGL $ belongs to the smallest thick subcategory containing the motivic cohomology spectrum $\motiviccoh \ZZ$ and closed under bigraded suspensions.
	Thus suffices to show that 
	\begin{equation*}
		\varinjlim \map_{\SH(k)}(\Sup^{2s, s}, (\motiviccoh \ZZ)_{X_{\bullet}}) \to  \map_{\SH(k)}(\Sup^{2s, s}, (\motiviccoh \ZZ)_{X}) 
	\end{equation*}
	is an equivalence; in other words, that $ \ell $-localized motivic cohomology of varieties satifies \ldhhyperdescent.
	Since $ \ell $-localized motivic cohomology has transfers along finite flat morphisms, this follows from a theorem of Geisser \cite[Theorem 1.2]{MR3226920}; see also a generalization due to Kelly \cite[Theorem 4.0.13]{MR3673293}.
\end{proof}


\subsection{The weight filtration on Borel--Moore homology via \texorpdfstring{\ldh}{ℓdh}-hyperdescent}\label{subsec:the_weight_filtration_on_Borel--Moore_homology_through_hypercovers}

In this subsection, we explain how \cref{theorem:hypercovers_of_varieties_are_colimit_diagrams_in_shk} allows one can calculate the weight filtration on Borel--Moore homology using \ldhhypercovers. 

\begin{recollection}
	\label{recollection:etale_and_betti_realization_in_hypercover_section}
	If $ X $ is a complex variety and $A \in \Alg(\spectra)$ is an algebra in spectra, then the Borel--Moore homology of the topological space $ X(\CC) $ with coefficients in $ A $ can be identified with the homotopy of the Betti realization of the monoidal dual of the compactly supported motive $ \motive(X) $ of \cref{definition:motive_and_motivec_of_a_variety}:
	\begin{equation*}
		\Hrm^{\BM}_{*}(X(\CC); A) \simeq \pi_{*} \Be(\motive(X)^{\vee}; A) \period
	\end{equation*}
	If $ A $ is complex orientable, then \Cref{cor:homological_filtered_Betti_realization_exists_for_complex_orientable_rings} gives a canonical lift of $ \Be(\motive(X)^{\vee}; A) $ to a filtered spectrum 
	\begin{equation*}
		\WBe(\motive(X); A) \in \Mod_{\tau_{\geq *} A}(\Fil \spectra) \period
	\end{equation*}
	Hence this filtration induces a \textit{weight filtration} on the Borel--Moore homology groups $ \Hrm^{\BM}_{*}(X(\CC); A) $. 

	Analogously, if $ k $ is an arbitrary field and if $A \in \Alg(\Shethyp(\Et_S;\Sp)\ellcomp)$ is complex orientable, then to any $ k $-variety $ X $ we can associate a hypercomplete étale sheaf of spectra 
	\begin{equation*}
		\Re_{\ell}(\motive(X)_{(\ell)}^{\vee}; A)) \period
	\end{equation*}
	This hypersheaf inherits a filtration from the filtered étale realization of \cref{definition:filtered_etale_realization}.
\end{recollection}

\begin{notation}
	To treat both the Betti and étale cases uniformly, for a variety $ X $ and $ A $ as in \cref{recollection:etale_and_betti_realization_in_hypercover_section} we write 
	\begin{equation*}
			\Crm_{*}^{\BM}(X; A) \colonequals  
			\begin{cases}
				\Be(\motive(X)^{\vee}; A) & \text{(Betti)} \\
			  	\Re_{\ell}(\motive(X)_{(\ell)}^{\vee}; A)) & \text{(Étale)}
			\end{cases}
	\end{equation*}
	Informally, these are the $ A $-linear Borel--Moore ``chains'', although note that in the étale case it is a hypersheaf of spectra on the étale site of $ k $ rather than a spectrum itself. 
	If $ k $ is separably closed, both types of ``chains'' are given by a spectrum. 
\end{notation}

\begin{recollection}[\ldhhypercovers by smooth schemes]
	\label{recollection:any_variety_admits_a_smooth_ldh_hypercover}
	If $ X $ is proper, then by using a theory of alterations, we can construct an \ldhhypercover $X_{\bullet} \to X$ such that $X_{i}$ is smooth and projective for each $i \geq 0$, see \cite[{Lemma 2 in \S 1.4}]{MR1409056}, where in op. cit. every time one invokes a resolution of singularities, we instead use the theory of alterations to obtain an \ldhcover. 
	If $ k $ is of characteristic zero, then by resolution of singularities any variety admits a \cdhhypercover which is levelwise smooth and projective.
\end{recollection}

\begin{theorem}
	\label{theorem:weights_on_bm_homology_of_proper_variety_calculated_through_an_ldh_hypercover} 
	Let $ X $ be a proper variety.
	Assume one of the following hypotheses:
	\begin{enumerate}
		\item Let $X_{\bullet} \to X$ be an \ldhhypercover such that $X_{i}$ is smooth and projective for each $i \geq 0$ and let $ A $ be $ \ell $-local.

		\item Let $ X_{\bullet} \to X $ be a \cdhhypercover and assume that the exponential characteristic is invertible in $ A $.
	\end{enumerate}
	Then we have 
	\begin{equation}
		\label{equation:calculating_weights_of_borel_moore_cochains}
		\Wup_{*} \Crm_{*}^{\BM}(X; A) \simeq \varinjlim_{[i] \in \Deltaop} \Wup_{*} \Crm_{*}^{\BM}(X_{i}; A) \simeq \varinjlim_{[i] \in \Deltaop} \tau_{\geq *} \Crm_{*}^{\BM}(X_{i}; A)
	\end{equation}
	where the colimit is calculated in $ \Modpost{A} $. 
\end{theorem}

\begin{proof}
	By a combination of \cref{theorem:hypercovers_of_varieties_are_colimit_diagrams_in_shk} of \cref{remark:for_cdh_hypercovers_the_descent_holds_after_inverting_the_characteristic}, we see that the canonical map
	\begin{equation*}
		\varinjlim \motive(X_{\bullet})^{\vee} \to \motive(X)^{\vee}
	\end{equation*}
	is Chow--Novikov $ \infty $-connective.
	By \Cref{corollary:filtered_betti_realization_inverts_chow_novikov_infty_connective_maps}, the filtered Betti realization inverts Chow--Novikov $ \infty $-connective maps; similarly, \cref{remark:filtered_etale_realization_also_inverts_chow_infty_connective_maps} shows that étale realization inverts Chow--Novikov $ \infty $-connective maps.
	Hence we deduce the left-hand equivalence.
	Since each $ X_i $ is smooth and proper, by construction we have 
	\begin{equation*}
		\WBe(\motive(X_{i})^{\vee}; A) \simeq \tau_{\geq *} \Be(\motive(X_{i})^{\vee}; A) \period
	\end{equation*}
	Hence the right-hand equivalence follows.
\end{proof}

\begin{corollary}
	Let $ X $ be a proper complex variety and let $A \in \Alg(\spectra)$ be complex orientable.
	Then the filtration on 
	\begin{equation*}
		\Hrm^{\BM}_{*}(X(\CC); A) \simeq \pi_{*} \Crm_{*}^{\BM}(X; A) 
	\end{equation*}
	induced from the weight filtration on the left-hand side coincides with the filtration induced by the hypercover spectral sequence. 
	\begin{equation*}
		\Eup^{1}_{s, t} \colonequals \Hrm^{\BM}_{t}(X_{s}(\CC); A) \Rightarrow \Hrm^{\BM}_{s+t}(X(\CC); A) \period 
	\end{equation*}
\end{corollary}

\begin{proof}
	The filtered spectrum $ \varinjlim_{[i] \in \Deltaop} \tau_{\geq *} \Crm_{*}^{\BM}(X_{i}; A)$ appearing in \cref{theorem:weights_on_bm_homology_of_proper_variety_calculated_through_an_ldh_hypercover} can be identified with Deligne's décalage of the simplicial spectrum $\Crm_{*}^{\BM}(X_{\bullet}; A)$.
	See \cites[\S9]{arXiv:2109.01017}[\S1.2.4]{HA}.
	By a result of Levine \cite[Proposition 6.3]{levine2015adams}, the resulting filtration on the homotopy groups of the colimit coincides with the one induced by the spectral sequence of geometric realization. 
\end{proof}

As observed in \cref{recollection:any_variety_admits_a_smooth_ldh_hypercover}, any proper variety admits an \ldhhypercover by smooth varieties.
Hence \cref{theorem:weights_on_bm_homology_of_proper_variety_calculated_through_an_ldh_hypercover} provides a way to explicitly calculate the weight filtration on Borel--Moore homology. 
If $U$ is not necessarily proper, then the weight filtration can be calculated as follows. 

\begin{proposition}
	Let $ X $ be a proper variety and $ Z \subseteq X $ a closed subvariety with open complement $ U $.
	Then the induced maps on Borel--Moore homology form a canonical cofiber sequence
	\begin{equation*}
		\WCstar^{\BM}(Z; A) \to \WCstar(X; A) \to \WCstar^{\BM}(U; A).
	\end{equation*}
	In particular, the weight filtration on $ \Cup_{*}^{\BM}(U;A) $ is canonically determined by the weight filtrations on $ \Cup_{*}(X;A) $ and $ \Cup_{*}^{\BM}(Z;A) $. 
\end{proposition}

\begin{proof}
	Immediate from the localization sequence of \cref{lemma:cofiber_sequence_of_motives_associated_to_open_closed_decomposition} and the fact that the filtered realization is exact. 
\end{proof}


\subsection{Filtration on cohomology and the comparison with the Gillet--Soulé filtration}\label{subsec:filtration_on_cohomology_and_the_comparison_with_the_Gillet--Soule_filtration}

In this subsection, we apply \cref{theorem:hypercovers_of_varieties_are_colimit_diagrams_in_shk} to compare the filtration on compactly supported integral cohomology of a complex variety with the \textit{Gillet--Soulé filtration} introduced in \cite{gillet1996descent}.
Recall that the Gillet--Soulé filtration refines Deligne's weight filtration on rational cohomology \cite{deligne1971theorie}. 

\begin{warning}[there are two different filtrations]
	\label{warning:two_filtrations_on_cochains} 
	There are \emph{two} filtrations one can construct on cohomology using the filtered realization functors introduced in this paper.
	To avoid complicating notation, let us focus on the Betti case; the discussion applies equally well to the filtered étale realization. 

	If $A \in \CAlg(\spectra)$ is complex orientable and $ X $ is a complex variety, then we have an identification 
	\begin{equation*}
		\Hc^{*}(X(\CC); A) \simeq \pi_{-*} \Be(\motive(X); A) \period
	\end{equation*}
	A natural way to lift the right-hand side to a filtered object is to consider 
	\begin{equation}
		\label{equation:natural_filtration_on_compactly_supported_cohomology} 
		\WBe(\motive(X); A) \period 
	\end{equation}
	However, an alternative is to observe that by \cref{corollary:motives_of_varieties_are_dualizable_away_from_the_characteristic}, the motivic spectrum $\motive(X)$ is dualizable; hence we can also consider the dual
	\begin{equation}
		\label{equation:wrong_filtration_on_compactly_supported_cohomology} 
		\map_{\tau_{\geq *} A}(\WBe(\motive(X)^{\vee}; A), \tau_{\geq *} A) \comma
	\end{equation}
	of $\Be(\motive(X)^{\vee};A) $ inside $ \Modpost{A} $.
	Recall that filtered Betti realization is not symmetric monoidal, but only \textit{lax} symmetric monoidal.
	Hence due to the failure of the universal coefficient theorem, \eqref{equation:natural_filtration_on_compactly_supported_cohomology} and \eqref{equation:wrong_filtration_on_compactly_supported_cohomology} need not coincide.
	This failure can already be observed when $ X $ is smooth and proper in which case: 
	\begin{enumerate}
	    \item The filtered object \eqref{equation:natural_filtration_on_compactly_supported_cohomology} can be identified with the Whitehead filtration on cochains.

	    \item The filtered object \eqref{equation:wrong_filtration_on_compactly_supported_cohomology} can be identified with the dual of the Whitehead filtration on chains. 
	\end{enumerate}
	When $ A $ is ordinary cohomology with coefficients in a field, these two coincide.
	However, in general they do not coincide. 
\end{warning}

\begin{nul}
	Note that out of the two ways of filtering cochains described in \cref{warning:two_filtrations_on_cochains}, it is the \emph{first} one which is preferable.  Indeed, if $ X $ is a proper variety, then the diagonal map $X \to X \times X$ equips $\motive(X)$ with a canonical structure of a commutative algebra in $\SH(\CC)$.
	Since $\WBe(-;A) $ is lax symmetric monoidal, it follows that 
	\begin{equation*}
		\WBe(\motive(X); A) 
	\end{equation*}
	canonically inherits the structure of a commutative algebra in filtered $\tau_{\geq *} A$-modules%
	\footnote{Dually, $\motive(X)^{\vee}$ is a cocommutative coalgebra; however, lax monoidal functors need not preserve coalgebras.
	This is why \eqref{equation:natural_filtration_on_compactly_supported_cohomology} is preferable to \eqref{equation:wrong_filtration_on_compactly_supported_cohomology}.
	This is the same reason why cohomology groups of a topological space form a commutative algebra, but homology groups need not form a coalgebra unless we have some further flatness assumption.}.
\end{nul}

\begin{notation}
	Let $ X $ be a complex variety.
	We write 
	\begin{equation*}
		\Cc^{*}(X(\CC); \ZZ) \in \Dcat(\ZZ)
	\end{equation*}
	for the complex of compactly supported integral cochains on $ X(\CC) $, considered as an object of the derived \category. 
\end{notation}

We recall the definition of the Gillet--Soulé filtration. 

\begin{recollection}[the Gillet--Soulé filtration]
	\label{recollection:gillet_soule_filtration}
	If $ X $ is a proper complex variety, then using resolution of singularities we can construct a \cdhhypercover $X_{\bullet} \to X$ by smooth proper varieties.
	The \defn{Gillet--Soulé filtration} on the Betti cohomology of $ X $ is the filtration associated to the spectral sequence
	\begin{equation*}
		\Hrm^{s}(X_{t}(\CC); \ZZ) \Rightarrow \Hrm^{s-t}(X(\CC); \ZZ) \period
	\end{equation*}
	Turning this into a filtered spectrum using décalage yields a definition 
	\begin{equation*}
		\WGS \Cup^{*}(X(\CC); \ZZ) \colonequals \lim_{[n] \in \DDelta} \tau_{\geq *} \Cup^{*}(X_{n}(\CC); \ZZ) \in \Fil(\Dcat(\ZZ))
	\end{equation*}
	If $ X $ is not necessarily proper, we embed $ X $ as an open subvariety $X \subseteq \Xbar $ of a proper variety $ \Xbar $ with closed complement $Z$ and define 
	\begin{equation*}
		\WGS \Cc^{*}(X(\CC); \ZZ) \colonequals \fib\paren{ \WGS \Cc^{*}(\Xbar(\CC); \ZZ) \to \WGS \Cc^{*}(Z(\CC); \ZZ) } \period
	\end{equation*}
	The results of \cite{gillet1996descent} show that, as objects of the filtered derived \category, these filtrations neither depend on the choice of the hypercover  $ X_{\bullet} $ nor on the choice of the compactification $ \Xbar $. 
	We refer to the filtered object $ \WGS \Cc^{*}(X(\CC); \ZZ) $ as the \defn{Gillet--Soulé filtration} on $ \Cc^{*}(X(\CC); \ZZ) $.
\end{recollection}

\begin{notation}[filtered Betti realization \& compactly supported cochains]
	If $ X $ is a complex variety, we have an identification 
	\begin{equation*}
		\Cc^{*}(X(\CC); \ZZ) \simeq \Be(\motive(X); \ZZ)
	\end{equation*}
	of objects in the derived \category $\Dcat(\ZZ)$.
	We write 
	\begin{equation*}
		\Wup_{*} \Cc^{*}(X(\CC); \ZZ) \colonequals \WBe(\motive(X); \ZZ)
	\end{equation*}
	for the filtration induced by the filtered Betti realization of \cref{def:filtered_Betti_realization}. 
\end{notation}

The filtration on compactly supported integral cochains inherited from the filtered Betti realization coincides with the Gillet--Soulé filtration: 

\begin{theorem}
	\label{theorem:gillet_soule_filtration_comparison}
	Let $ X $ be a complex variety.
	Then there exists an equivalence
	\begin{equation}
		\label{equation:equivalence_of_filtrations_in_comparison_with_gs_filtration}
		\Wup_{*} \Cc^{*}(X(\CC); \ZZ) \simeq \Wup^{\GS}_{*} \Cc^{*}(X(\CC); \ZZ)
	\end{equation}
	of objects of the filtered derived \category of $ \ZZ $.
\end{theorem}

Before proceeding with the proof, let us remark that the main difficulty lies in the fact that the Gillet--Soulé filtration is defined as a \emph{limit}, whereas filtered Betti realization is a left adjoint, hence preserves \emph{colimits}.
Since we are in the stable context, finite limits and be expressed as finite colimits and vice versa, but the limit defining the Gillet--Soulé filtration is a totalization of a cosimplicial object and hence is not finite.
To prove \cref{theorem:gillet_soule_filtration_comparison}, we will show that after passing to the associated graded, the \cdhhypercover can be replaced by a suitable chain complex in effective Chow motives.
Gillet and Soulé's work \cite{MR1409056} then shows that this complex of effective Chow motives can be chosen to be bounded. 

The key step in the proof of \cref{theorem:gillet_soule_filtration_comparison} is to argue that the associated graded of the filtered Betti realization is defined on $\MZZ_{c=0}$-modules. 
This takes some preparation. 

\begin{recollection}[associated graded]
	We write
	\begin{equation*}
		\Gr(\Sp) \colonequals \Fun(\ZZ^{\disc}, \Sp)
	\end{equation*}
	for the \category of \defn{graded spectra}.
	Given a filtered spectrum $F_{*} S$, the \defn{associated graded} of $F_{*} S$ is the graded spectrum defined by
	\begin{equation*}
		\gr_{k} (F_{*} S) \colonequals \cofib(F_{k+1}S \to F_{k} S) \period
	\end{equation*} 
	A filtered spectrum $F_{*} S$ is \defn{complete} if $\varprojlim_{n \in \ZZ} F_{n} S = 0 $.
	We write
	\begin{equation*}
		\Fil^{\wedge}(\Sp) \subseteq \FilSp
	\end{equation*}
	for the full subcategory spanned by the complete filtered spectra.
	On this subcategory, passing the associated graded functor
	\begin{equation*}
		\gr_{*} \colon \Fil^{\wedge}(\Sp) \to \Gr(\Sp)
	\end{equation*}
	is conservative.
\end{recollection}

\begin{notation}
	Let $\MZZ \in \SH(\CC)$ denote the motivic cohomology spectrum and $\MZZ_{c=0} \simeq \MZZ_{c \leq 0}$ its connective cover in the Chow--Novikov \tstructure. 
\end{notation}

The first observation is that the associated graded of the $\ZZ$-linear filtered Betti realization factors through $\MZZ_{c=0}$-modules:

\begin{lemma}
	\label{lemma:associated_graded_of_z_linear_betti_is_functorial_in_hzczero_modules}
	There exists a left adjoint functor 
	\begin{equation*}
		\gr_{*} \Be_{c=0}(-; \ZZ) \colon \Mod_{\MZZ_{c=0}}(\SH(\CC)) \to \Gr(\Dcat(\ZZ)) 
	\end{equation*}
	such that there is an equivalence
	\begin{equation*}
		\gr_{*} \Be_{c=0}(\MZZ_{c=0} \otimes S; \ZZ) \simeq \gr_{*}(\WBe(S; \ZZ))
	\end{equation*}
	natural in $S \in \SH(\CC)$.  
\end{lemma}

\begin{proof}
	Write $ \Chow(\CC) $ for the additive \emph{$ 1 $-category} of pure Chow motives over $ \CC $.
	By \cite[\S 4.2]{bachmann2022chow}, there is an equivalence of \categories 
	\begin{equation*}
		\Mod_{\MZZ_{c=0}}(\SH(\CC)) \equivalent \PSigma(\Chow(\CC); \spectra)
	\end{equation*}
	between $\MZZ_{c=0}$-modules and spectral presheaves on $ \Chow(\CC) $. 
	It follows that any additive functor on $\Chow(\CC)$ valued in a cocomplete stable \category extends uniquely to a colimit-preserving functor on all $\MZZ_{c=0}$-modules. 
	The needed functor $\gr_{*} \Be_{c=0}(-; \ZZ)$ is defined as the unique colimit-preserving functor such that 
	\begin{equation*}
		\gr_{n} \Be_{c=0}(M; \ZZ) \colonequals \Sigma^{-n} \Hrm_{\Be}^{n}(M; \ZZ)
	\end{equation*}
	for any $M \in \Chow(\CC)$, where the right-hand side is the homological Betti realization of a Chow motive.  
	If $S \in \SH(\CC)$ is perfect pure, then wehave  $\MZZ_{c=0} \otimes S \in \Chow(\CC) $ so that 
	\begin{align*}
		\gr_{*} \Be_{c=0}(\MZZ_{c=0} \otimes S; \ZZ) &\simeq \Sigma^{-n} \Hrm_{\Be}^{n}(\MZZ_{c=0} \otimes S; \ZZ) \\ 
		&\simeq \gr_{*}( \WBe (S; \ZZ)) \period
	\end{align*}
	Since both sides preserve colimits, \cref{cor:functors_out_of_SH_in_terms_of_Pure} implies that this natural equivalence defined on perfect pures extends to an equivalence on all of $\SH(\CC)$. 
\end{proof}

\begin{proof}[{Proof of \cref{theorem:gillet_soule_filtration_comparison}}]
	By \cref{lemma:cofiber_sequence_of_motives_associated_to_open_closed_decomposition}, the left-hand filtration takes open-closed decompositions to fiber sequences, and by definition the Gillet--Soulé filtration takes open-closed decompositions to fiber sequences.
	Hence we can assume that $ X $ is proper. 
	Using resolution of singularities, we can choose a \cdhhypercover $X_{\bullet} \to X$ such that $X_{i}$ is smooth and proper for each $i \geq 0$. 
	Functoriality of the filtered Betti realization applied to $X_{\bullet} \to X$ gives a canonical comparison map 
	\begin{equation}
		\label{equation:comparison_map_between_ours_and_hs_filtrations_in_proof}
		\Wup_{*} \Crm^{*}(X(\CC); \ZZ) \to \varprojlim_{[i] \in \DDelta} \Wup_{*} \Crm^{*}(X_{i}(\CC); \ZZ) \simeq \varprojlim \tau_{\geq *} \Crm^{*}(X_{i}(\CC); \ZZ) \simeq \Wup^{\GS}_{*} \Crm^{*}(X(\CC); \ZZ) \period
	\end{equation}
	We will prove that \eqref{equation:comparison_map_between_ours_and_hs_filtrations_in_proof} is an equivalence.

	First, we claim that both the source and target of \eqref{equation:comparison_map_between_ours_and_hs_filtrations_in_proof} are complete. 
	Indeed, the target is a limit of Whitehead filtrations, which are complete, and hence is complete itself.
	On the other hand, since the subcategory of those motivic spectra $ S $ such that $\WBe(S; \ZZ)$ is complete is thick and contains motives of all smooth and proper varieties, \cref{lemma:subcategory_of_motives_containing_smooth_projectives_has_all_motives} implies that the source is also complete. 

	We deduce that it is enough to show that \eqref{equation:comparison_map_between_ours_and_hs_filtrations_in_proof} is an equivalence after passing to associated graded objects.
	By \cref{lemma:associated_graded_of_z_linear_betti_is_functorial_in_hzczero_modules}, the map between associated graded objects can be identified with the comparison map 
	\begin{equation*}
		\gr_{*} \Be_{c=0}(\MZZ_{c=0} \otimes \motive(X); \ZZ) \to \varprojlim_{[i] \in \DDelta} \gr_{*} \Be_{c=0}(
		(\MZZ_{c=0} \otimes \motive(X_{i}); \ZZ) \colon
	\end{equation*}
	Since $X_{\bullet} \to X$ is a \cdhcover and $\MZZ_{c=0}$ is an $ \MGL $-module, by  \cref{theorem:hypercovers_of_varieties_are_colimit_diagrams_in_shk} and \cref{remark:for_cdh_hypercovers_the_descent_holds_after_inverting_the_characteristic}, we have 
	\begin{equation*}
		\MZZ_{c=0} \otimes \motive(X)^{\vee} \simeq \varinjlim_{[i] \in \Deltaop} \MZZ_{c=0} \otimes \motive(X_{i})^{\vee} \period
	\end{equation*}
	Passing to monoidal duals, this shows that 
	\begin{equation}
		\label{equation:limit_expression_of_hzczero_motive_of_variety_in_terms_of_cdh_hypercover}
		\MZZ_{c=0} \otimes \motive(X) \to \varprojlim_{[i] \in \DDelta} \MZZ_{c=0} \otimes \motive(X_{i}) \period
	\end{equation}
	We have to show that this limit is preserved by the functor $\gr_{*} \Be_{c=0}(-; \ZZ)$ of \cref{lemma:associated_graded_of_z_linear_betti_is_functorial_in_hzczero_modules}. 

	Through the Dold--Kan correspondence, the cosimplicial object $\MZZ_{c=0} \otimes \motive(X_{\bullet}) \colon \Delta \to \Chow(\CC)$ determines a chain complex of pure Chow motives; this complex can be identified with the weight complex of \cite[p. 137-138]{gillet1996descent}.
	By \cite[p.137, Theorem 2]{gillet1996descent}, this chain complex is homotopy equivalent to a bounded one. 
	Using the Dold--Kan correspondence, this homotopy equivalence of chain complexes determines a map $\MZZ_{c=0} \otimes X_{\bullet} \to C_{\bullet}$ of cosimplicial Chow motives which is a cosimplicial homotopy equivalence.
	The assumption that the chain complex associated to $C_{\bullet}$ is bounded implies that $C_{\bullet}$ is $n$-coskeletal for some $n$. 

	We have a commutative diagram of $\MZZ_{c=0}$-modules of the form 
	\begin{equation*}
		\begin{tikzcd}
			& \MZZ_{c=0} \otimes \motive(X) \arrow[dr] \arrow[dl] & \\
			\displaystyle \varprojlim_{[i] \in \DDelta} \MZZ_{c=0} \otimes \motive(X_{i}) \arrow[rr] & & \displaystyle \varprojlim_{[i] \in \DDelta} C_{i} \period
		\end{tikzcd}
	\end{equation*}
	Since the horizontal map is induced by a cosimplicial homotopy equivalence, it is an equivalence, and similarly $\gr_{*} \Be_{c=0}(\MZZ_{c=0} \otimes \motive(X_{\bullet}); \ZZ) \simeq \varprojlim \gr_{*} \Be_{c=0}(C_{\bullet}; \ZZ)$.
	Thus, it is enough to show that 
	\begin{equation*}
		\gr_{*} \Be_{c=0}(\motive(X); \ZZ) \to \varprojlim \gr_{*} \Be_{c=0}(C_{\bullet}; \ZZ)
	\end{equation*}
	is an equivalence. However, as the right-hand side is a totalization of an $n$-coskeletal cosimplicial object, it can be identified with a finite limit.
	As $\gr_{*} \Be_{c=0}(-; \ZZ)$ is exact, it preserves finite limits, ending the argument. 
\end{proof}

\begin{remark}
	A key step in the proof of \cref{theorem:gillet_soule_filtration_comparison} is the boundedness result for the Gillet--Soulé weight complex, which implies that the infinite \cdhhypercover $X_{\bullet} \to X$ can be replaced by an object of finitary nature.
	On the other hand, as a consequence of \cref{lemma:subcategory_of_motives_containing_smooth_projectives_has_all_motives}, the filtered spectrum
	\begin{equation*}
		\WBe(\motive(X); A) 
	\end{equation*}
	can \emph{always} be obtained from the Whitehead filtration on $ A $-linear cochains of smooth, proper varieties using only finite limits and colimits. 
	Unlike in the case of Borel--Moore homology covered by \cref{theorem:weights_on_bm_homology_of_proper_variety_calculated_through_an_ldh_hypercover}, for general $ A $ we do not know if the filtration on cochains satisfies \cdhdescent; although  \cref{theorem:gillet_soule_filtration_comparison} shows that it does when $A = \ZZ$.  
\end{remark}

\begin{remark}[Kuijper's work]
	\label{remark:work_of_kuijper}
	In the case of a field of characteristic zero, a weight filtration on compactly supported cohomology with coefficients in a complex orientable ring spectrum $ A $ can also be constructed using the recent work of Kuijper \cite{kuijper2022general}.
	We claim that this filtration agrees with the filtered realization introduced in this work applied to $\motive(X)$. 

	For simplicity, let us consider the complex Betti case.
	We have the association
	\begin{equation*}
		X \mapsto \tau_{\geq *} \Cc^{*}(X(\CC); A) \in \FilSp \comma
	\end{equation*}
	which we think of as a presheaf defined on smooth and proper varieties.
	As observed in \cite[8.3]{kuijper2022general}, if $ A $-cohomology admits Gysin maps, then this presheaf satisfies descent for blow-ups squares.
	Moreover, if $ A $ is complex orientable then $ A $-cohomology admits Gysin maps. 
	Thus, by \cite[Theorem 1.1]{kuijper2022general}, this presheaf uniquely extends to one defined on all varieties, giving the sought after weight filtration on compactly supported cohomology.
	Note that the two filtrations
	\begin{equation*}
		X \mapsto \tau_{\geq *} \Cc^{*}(X(\CC); A) \andeq X \mapsto \Wup_{*} \Be(\motive(X); A)
	\end{equation*}
	agree on smooth and proper varieties, have the localization property, and satisfy descent for blow-up squares. 
	Hence the uniqueness part of Kujiper's result, shows that these filtrations necessarily agree on all complex varieties.
\end{remark}


\section{Synthetic Betti realization}\label{sec:synthetic_realizations}

Write $ \SynMU $ for the \category of \textit{$ \MU $-based synthetic spectra} introduced by the second-named author in \cite{pstrkagowski2022synthetic}. 
The goal of the section is to show that the Betti realization functor $ \Be \colon \SH(\CC) \to \Sp $ refines to a lax symmetric monoidal left adjoint
\begin{equation*}
	\Besyn \colon \SH(\CC) \to \SynMU 
\end{equation*}
as well as explore its basic properties.
We refer to this refinement as \textit{synthetic Betti realization}.

In \cref{subsec:recollection_on_synthetic_spectra}, we recall the background on synthetic spectra necessary to understand the construction of the synthetic Betti realization functor.
In \cref{subsection:synthetic_spectra_as_filtered_spectra}, we explain give an alternative description of synthetic spectra as modules in filtered spectra over the filtration on the sphere given by descent along the faithfully flat map $ \Sup^0 \to \MU $.
This description is later used to compare synthetic Betti realization with filtered Betti realization.
In \cref{subsection:synthetic_complex_betti_realization}, we construct the functor $ \Besyn $; see \Cref{theorem:existence_of_complex_synthetic_betti_homology}.
In \cref{subsec:comparing_synthetic_and_filtered_Betti_realization}, we explain the relationship between synthetic Betti realization to filtered Betti realization.
In particular, if $ A $ is a Landweber exact complex oriented $ \E_1 $-ring, then the filtered Betti realization $ \WBe(-;A) $ can be recovered from synthetic Betti realization; see \Cref{thm:synthetic_realization_refines_filtered_realization_for_Landwebder_exact_rings}.
In \cref{subsection:ideas_on_synthetic_real_and_etale_realizations}, we give conjectural description of a synthetic lift of a general motivic realization functor, such as étale realization.


\subsection{Recollection on synthetic spectra}\label{subsec:recollection_on_synthetic_spectra}

Initiated by Quillen, \textit{chromatic homotopy theory} studies the relationship between stable homotopy theory and the arithmetic formal groups. 
An important aspect of this relationship is the \textit{Adams--Novikov spectral sequence}
\begin{equation*}
	\Hrm^{s}(\Mfg; \omega^{t/2}) \Rightarrow \pi_{s-t} \Sup^{0}
\end{equation*}
relating cohomology of the moduli stack of formal groups to stable homotopy groups. 
\textit{Synthetic spectra} can be informally thought of as categorification of this spectral sequence.
The purpose of this subsection is to briefly review what we need about synthetic spectra for this paper; we refer the reader to \cite{MR4574661} for more details.

We first recall the construction of $ \MU $-based synthetic spectra from \cite[\S4]{pstrkagowski2022synthetic}. 
We say that a spectrum $ A $ is \emph{finite $ \MU $-projective} if $ A $ is a compact object of $ \Sp $ and $ \MU \otimes A $ is free as an $ \MU $-module; that is, there exists integers $ d_{1}, \ldots, d_n $ and an equivalence of $ \MU $-modules 
\begin{equation*}
	\MU \otimes A \simeq \Sigma^{d_1} \MU \directsum \cdots \directsum \Sigma^{d_n} \MU \period
\end{equation*}
Equivalently, $ A $ is compact and $ \MU_{*}(A) $ is free as an $ \MU_{*} $-module. 
We write
\begin{equation*}
	\spectra_{\MU}^{\fp} \subseteq \Sp
\end{equation*}
for the full subcategory spanned by the finite $ \MU $-projective spectra. 
We say that map $ f \colon A \to B$ of finite $ \MU $-projectives is an \emph{$ \MU $-epimorphism} if $ f $ becomes a split epimorphism after tensoring with $ \MU $; equivalently, if $\MU_{*}A \to \MU_{*}B$ is surjective. 
This notion of a covering equips the site $\spectra_{\MU}^{\fp}$ with a Grothendieck topology. 

\begin{definition}
	The \category of \defn{$ \MU $-based synthetic spectra} is given by 
	\begin{equation*}
		\SynMU \colonequals \ShSigma(\spectra_{\MU}^{\fp}; \spectra)
	\end{equation*}
	the \category of additive sheaves on spectra on the site $\spectra_{\MU}^{\fp}$ of finite $ \MU $-projective spectra. 
\end{definition}

\begin{nul}[$ \SynMU $ as a deformation of $ \Sp $]
	The \category $ \SynMU $ is stable and presentable.
	Moreover, through left Kan extension it inherits a symmetric monoidal tensor product from that of finite spectra. 
	As we briefly explained in the introduction, the \category $ \SynMU $ is best understood as \acategorical deformation of spectra in the following sense. 
	Its monoidal unit has a canonical endomorphism
	\begin{equation*}
		\tau \colon \Unit_{\Syn} \to \Unit_{\Syn}
	\end{equation*}
	which should be thought of as a formal parameter.
	Moreover, there is an equivalence 
	\begin{equation*}
		\SynMU^{\tau = 1} \simeq \spectra
	\end{equation*}
	between the generic fiber and spectra.
	The special fiber is related to arithmetic.
	Write $ \Mfg^{\delta} $ for the \textit{Dirac} moduli stack of formal groups (that is, a sheaf on the category of graded-commutative rings) as defined in \cite[\S 5.2]{diracgeometry2}.  
	Then there is an equivalence
	\begin{equation*}
		\SynMU^{\tau = 0} \simeq \IndCoh(\Mfg^{\delta})
	\end{equation*}
	between the special fiber and $\Ind$-coherent sheaves on $ \Mfg^{\delta} $.
\end{nul}

One can describe this \category of $\Ind$-coherent sheaves on $ \Mfg^{\delta} $ in terms of $ \Ind $-coherent sheaves on the usual moduli stack of formal groups as follows:

\begin{remark}[Dirac moduli of formal groups and its classical counterpart]
	\label{remark:dirac_moduli_of_formal_groups_and_its_classical_counterpart}
	The \category of $\Ind$-coherent sheaves on $\Mfg^{\delta}$ admits a fully faithful embedding 
	\begin{equation*}
		i \colon \IndCoh(\Mfg) \hookrightarrow \IndCoh(\Mfg^{\delta})
	\end{equation*}
	from $\Ind$-coherent sheaves on the moduli stack $ \Mfg $ of formal groups in classical algebraic geometry. 
	The target is obtained from the source by attaching an anti-symmetric square root $\omega^{\nicefrac{1}{2}}$ of the Lie algebra line bundle $\omega \in \IndCoh(\Mfg)$ in the sense that any $\Fcal \in \IndCoh(\Mfg^{\delta})$ can be uniquely written in the form 
	\begin{equation*}
		\Fcal \simeq (i(\Fcal_{0})) \oplus (\omega^{\nicefrac{1}{2}} \otimes i(\Fcal_{1}))
	\end{equation*}
	for $\Fcal_{0}, \Fcal_{1} \in \IndCoh(\Mfg)$. 
	Informally, the additional root arises from the fact that in spectra, the Betti realization $\Be(\PP_{\CC}^{1}) \simeq \Sup^{2}$ of the Tate motive has a tensor square root, given by the $1$-sphere $\Sup^{1}$. 
	This situation is quite special to complex Betti realization. 
\end{remark}

\begin{remark}
	The embedding $i \colon \IndCoh(\Mfg) \hookrightarrow \IndCoh(\Mfg^{\delta})$ mentioned in \cref{remark:dirac_moduli_of_formal_groups_and_its_classical_counterpart} can be identified with the embedding of special fibers
	\begin{equation*}
		(\SynMUev)^{\tau = 0} \hookrightarrow \SynMU^{\tau = 0}
	\end{equation*}
	induced by the inclusion of \textit{even synthetic spectra} of \cite[\S 5.2]{pstrkagowski2022synthetic} into all synthetic spectra. 
\end{remark}

\begin{remark}[Ind-coherent sheaves and Hovey's stable \category]
	In terms of Hopf algebroids, we have a canonical equivalence 
	\begin{equation*}
		\IndCoh(\Mfg^{\delta}) \simeq \Stable_{\MU_{*}\MU}
	\end{equation*}
	between sheaves on the Dirac moduli of formal groups and Hovey's stable \category of $\MU_{*}\MU$-comodules as in \cite{hovey2004homotopy}. Under this equivalence, the subcategory 
	\begin{equation*}
		\IndCoh(\Mfg) \subseteq \IndCoh(\Mfg^{\delta})
	\end{equation*}
	of sheaves on the classical moduli stack corresponds to the stable \category of $\MU_{*}\MU$-comodules concentrated in even degrees.
\end{remark}

\begin{nul}[synthetic analogues]
	The \category of synthetic spectra is equipped with a fully faithful embedding $\nu \colon \spectra \hookrightarrow \SynMU$, called the \emph{synthetic analogue}, which fits into a commutative diagram 
	\begin{equation}
		\label{equation:diagram_explaining_properties_of_nu}
		\begin{tikzcd}[row sep=3em, column sep=5em]
			& \Sp \arrow[dl, equals] \arrow[d, hook, "\nu"'] \arrow[dr, "\MU_{*}(-)"] & \\
			\Sp & \SynMU \arrow[l, "(-)^{\tau = 1}"] \arrow[r, "(-)^{\tau = 0}"'] & \IndCoh(\Mfg)
		\end{tikzcd}
	\end{equation}
	The functor $ \nu $ is additive, but it is \emph{not exact}.
	However, one can show that a cofiber sequence 
	\begin{equation*}
		A \to B \to C
	\end{equation*}
	of spectra is preserved by $ \nu $ if and only if 
	\begin{equation*}
		0 \to \MU_{*}A \to \MU_{*}B \to \MU_{*}C \to 0
	\end{equation*}
	is short exact \cite[Lemma 4.23]{pstrkagowski2022synthetic}. 
	In particular, $ \nu \colon \Sp \hookrightarrow \SynMU $ preserves $ \MU $-split cofiber sequences; this is the crucial property we need to construct the synthetic lift of the Betti realization functor.
\end{nul} 


\subsection{Synthetic spectra as filtered spectra} 
\label{subsection:synthetic_spectra_as_filtered_spectra}

We now explain an alternative presentation of synthetic spectra in terms of filtered spectra.
There are two relevant filtrations on the sphere that come from descent along the faithfully flat map $ \Sup^0 \to \MU $.

\begin{notation}
	\hfill
	\begin{enumerate}
		\item Write $ \filev^{*}(\Sup^0) $ the commutative algebra in filtered spectra defined by the limit
		\begin{equation*}
			\filev^{*}(\Sup^0) \colonequals \lim_{[n] \in \DDelta} \tau_{\geq 2*}(\MU^{\tensor [n]}) \period
		\end{equation*}
		Here, the limit is taken over the diagram given by applying the \textit{double-speed} Postnikov filtration to the cobar construction of the unit $ \Sup^0 \to \MU $.
		This filtration on the sphere is the Adams--Novikov filtration; it can also be identified with the \textit{even filtration} of \cites{hahn2022motivic}{perfect_even_modules}.

		\item Write $ \fil^{*}(\Sup^0) $ for the commutative algebra in filtered spectra defined by the limit
		\begin{equation*}
			\fil^{*}(\Sup^0) \colonequals \lim_{[n] \in \DDelta} \tau_{\geq *}(\MU^{\tensor [n]}) \period
		\end{equation*}
		We refer to $ \fil^{*}(\Sup^0) $ as the \textit{$ \MU $-descent filtration} on $ \Sup^0 $.
		The filtration $ \fil^{*}(\Sup^0) $ agrees with the \textit{half-weight even filtration} of \cite[Remark 2.26]{perfect_even_modules}.
	\end{enumerate}	
\end{notation}

The following description of synthetic spectra in terms of filtered spectra is due to Gheorghe--Krause--Isaksen--Ricka \cite{MR4432907}.
See also \cites[\S1.3]{arXiv:2111.15212}[Corollary 6.1]{arXiv:2402.03257}[\S3.2]{arXiv:2304.04685}.

\begin{proposition}[synthetic spectra as filtered spectra]\label{prop:synthetic_spectra_as_filtered_spectra}
	\hfill
	\begin{enumerate}
		\item There is an equivalence of symmetric monoidal \categories
		\begin{equation*}
			\begin{tikzcd}
				\Gamma^{*} \colon \SynMU \arrow[r, "\sim"{yshift=-0.25ex}] & \Mod_{\fil^{*}(\Sup^{0})}(\FilSp) \period
			\end{tikzcd}
		\end{equation*}

		\item The equivalence $ \Gamma^{*} $ restricts to an equivalence of symmetric monoidal \categories
		\begin{equation*}
			\begin{tikzcd}
				\SynMUev \arrow[r, "\sim"{yshift=-0.25ex}] & \Mod_{\filev^{*}(\Sup^{0})}(\FilSp) \period
			\end{tikzcd}
		\end{equation*}

		\item The triangle 
		\begin{equation*}
			\begin{tikzcd}
				\SynMU \arrow[rr, "\Gamma^{*}", "\sim"'{yshift=0.25ex}] \arrow[dr, "(-)^{\tau = 1}"'] & & \Mod_{\fil^{*}(\Sup^{0})}(\FilSp) \arrow[dl, "\colim"] \\ 
				& \Sp & 
			\end{tikzcd}
		\end{equation*}
		canonically commutes.
	\end{enumerate}
\end{proposition}

Now let us further relate filtered objects in modules over a complex orientable ring to synthetic spectra.

\begin{lemma}\label{lem:equivalent_descriptions_of_filtered_modules_over_a_complex orientable_ring}
	Let $ A $ be a complex orientable $ \E_1 $-ring spectrum.
	Then:
	\begin{enumerate}
		\item There is an equivalence $ \Gamma^{*}(\nu(A)) \equivalent \tau_{\geq *} A $ of $ \E_1 $-algebras in $ \FilSp $.

		\item The equivalence $ \Gamma^{*} \colon \SynMU \equivalence \Mod_{\fil^{*}(\Sup^{0})}(\FilSp) $ induces is an equivalence of \categories
		\begin{equation*}
			\Mod_{\nu(A)}(\SynMU) \equivalence \Modpost{A} \period
		\end{equation*}
	\end{enumerate}
\end{lemma}

\begin{proof}
	For (1), note that under the equivalence $ \Gamma^{*} \colon \SynMU \equivalence \Mod_{\fil^{*}(\Sup^{0})}(\FilSp) $, $\nu(A)$ corresponds to $ A $ equipped with the Adams--Novikov filtration.
	Since $ A $ is complex orientable and hence a retract of an $ \MU $-module, the Adams--Novikov filtration on $ A $ is identified with the Postnikov filtration $ \tau_{\geq *}A $.
	See \cite[Proposition 4.60]{pstrkagowski2022synthetic}. 

	For (2), note that by (1) and \Cref{prop:synthetic_spectra_as_filtered_spectra}, we have equivalences
	\begin{align*}
		\Mod_{\nu(A)}(\SynMU) &\equivalent \Mod_{\Gamma^{*}(\nu(A))}(\Mod_{\fil^{*}(\Sup^0)}(\FilSp)) \\ 
		&\equivalent \Modpost{A} \period \qedhere
	\end{align*}
\end{proof}


\subsection{Synthetic complex Betti realization} 
\label{subsection:synthetic_complex_betti_realization}

In this subsection, we refine the Betti realization functor 
\begin{equation*}
	\Be \colon \SH(\CC) \to \spectra
\end{equation*}
of \Cref{construction:complex_betti_realization} to a colimit-preserving functor valued in synthetic spectra.
This refinement is analogous to our construction of filtered Betti realization (\Cref{def:filtered_Betti_realization}).
Specifically, the idea is to send a perfect pure motivic spectrum $ X $ to the `trivial' synthetic spectrum $ \nu(\Be(X)) $ associated to the Betti realization of $ X $. 
To check that this extends to a colimit-preserving functor $ \SH(\CC) \to \SynMU $, we need to check that it the functor $ \nu(\Be(-)) $ preserves cofiber sequences in $ \Pure(\CC) $:

\begin{lemma}\label{lem:MU_tensor_Betti_realization_splits_cofiber_sequence_of_pures}
	Let $ X \to Y \to Z $ be an $ \MGL $-split cofiber sequence in $ \SH(\CC) $.
	Then the null sequence
	\begin{equation*}
		\begin{tikzcd}
			\nu(\Be(X)) \arrow[r] & \nu(\Be(Y)) \arrow[r] & \nu(\Be(Z))
		\end{tikzcd}
	\end{equation*}
	is a cofiber sequence of synthetic spectra.
\end{lemma}

\begin{proof}
	By \Cref{lem:MU-linear_homological_Betti_realization_splits_pure_cofiber_sequences}, the induced cofiber sequence of $ \MU $-modules 
	\begin{equation*}
		\begin{tikzcd}
			\MU \tensor \Be(X) \arrow[r] & \MU \tensor \Be(Y) \arrow[r] & \MU \tensor \Be(Z) 
		\end{tikzcd}
	\end{equation*}
	is split.
	The claim now follows from the fact that the functor $ \nu \colon \Sp \hookrightarrow \SynMU $ preserves $\MU_{*}$-exact cofiber sequences \cite[Lemma 4.23]{pstrkagowski2022synthetic}.
\end{proof}

\begin{example}\label{ex:nu_Betti_preserves_cofiber_sequences_of_pures}
	Let $ X \to Y \to Z $ is a cofiber sequence in $ \Pure(\CC) $.
	Combining \Cref{proposition:pure_epimorphisms_detected_by_mgl,lem:MU_tensor_Betti_realization_splits_cofiber_sequence_of_pures} shows that
	\begin{equation*}
		\begin{tikzcd}
			\nu(\Be(X)) \arrow[r] & \nu(\Be(Y)) \arrow[r] & \nu(\Be(Z))
		\end{tikzcd}
	\end{equation*}
	is a cofiber sequence of synthetic spectra.
\end{example}

\begin{theorem}\label{theorem:existence_of_complex_synthetic_betti_homology}
	There exists a unique lax symmetric monoidal left adjoint 
	\begin{equation*}
		\Besyn \colon \SH(\CC) \to \SynMU 
	\end{equation*}
	such that for $ X \in \Pure(\CC) $, we have
	\begin{equation*}
		\Besyn(X) \simeq \nu(\Be(X)) \period
	\end{equation*}
\end{theorem}

\begin{proof}
	By \cref{theorem:homological_weight_context_has_unique_solution_if_t_is_mgl_exact}, it suffices to show that if $X \to Y \to Z$ is a cofiber sequence in $ \Pure(\CC) $, then the sequence 
	\begin{equation*}
		\begin{tikzcd}
			\nu(\Be(X)) \arrow[r] & \nu(\Be(Y)) \arrow[r] & \nu(\Be(Z))
		\end{tikzcd}
	\end{equation*}
	is a cofiber sequence in $ \SynMU $; this is the content of \Cref{ex:nu_Betti_preserves_cofiber_sequences_of_pures}.
\end{proof}

\begin{definition}[synthetic Betti realization]\label{def:synthetic_Betti_homology}
	We refer to the functor $\Besyn \colon \SH(\CC) \to \SynMU$ of \cref{theorem:existence_of_complex_synthetic_betti_homology} as complex \emph{synthetic Betti realization}. 
\end{definition}

Using the fact that synthetic Betti realization is a left adjoint, it is not hard to see that synthetic Betti realization refines the usual Betti realization:  

\begin{lemma}
	The triangle of \categories and left adjoints 
	\begin{equation*}
		\begin{tikzcd}
			\SH(\CC) \arrow[rr, "\Besyn"] \arrow[dr, "\Be"'] & & \SynMU \arrow[dl, "(-)^{\tau=1}"] \\
			& \Sp &
		\end{tikzcd}
	\end{equation*}
	canonically commutes. 
\end{lemma}

\begin{proof}
	By \cref{cor:functors_out_of_SH_in_terms_of_Pure}, it suffices to show that this diagram commutes when restricted to perfect pures. 
	If $X \in \Pure(\CC)$, this follows from the equivalences
	\begin{equation*}
		\Besyn(X)^{\tau = 1} \simeq (\nu(\Be(X))^{\tau=1} \simeq \Be(X) \period
	\end{equation*}
	Here, the second equivalence uses that \eqref{equation:diagram_explaining_properties_of_nu} commutes. 
\end{proof}


\subsection{Comparing synthetic Betti realization and filtered Betti realization}\label{subsec:comparing_synthetic_and_filtered_Betti_realization}

As a consequence of \Cref{prop:synthetic_spectra_as_filtered_spectra}, we see that synthetic Betti realization can be thought of as equipping $\Be(X)$ with an additional filtration compatible with the $ \MU $-descent filtration of the sphere. 
In this subsection, we use the filtered perspective on synethetic spectra explained in \cref{subsection:synthetic_spectra_as_filtered_spectra} to compare synthetic Betti realization with filtered Betti realization (recall \Cref{def:filtered_Betti_realization}).

\begin{construction}[$ A $-linear realization of a synthetic spectrum]
	Let $ A $ be a complex orientable $ \E_{1} $-ring.
	Write $ \Re_{A} $ for the composite
	\begin{equation*}
		\begin{tikzcd}[sep=4.5em]
			\SynMU \arrow[r, "\nu(A) \otimes (-)"] & \Mod_{\nu(A)}(\SynMU) \arrow[r, "\sim"{yshift=-0.25ex}] & \Modpost{A} \period
		\end{tikzcd}
	\end{equation*}
	Here, the second functor is the equivalence of \Cref{lem:equivalent_descriptions_of_filtered_modules_over_a_complex orientable_ring}.
	We call $ \Re_A $ the \defn{$ A $-linear realization} functor. 
\end{construction}

\begin{observation}
	Since $ \nu $ is lax symmetric monoidal, for any spectrum $ X $ we have a canonical comparison map 
	\begin{equation*}
		\nu(A) \otimes \nu(X) \to \nu(A \otimes X) 
	\end{equation*}
	which we can identify with a map 
	\begin{equation}
		\label{equation:canonical_map_from_a_linear_realization_of_nu}
		\Re_{A}(\nu(X)) \to \tau_{\geq *}(A \otimes X) \comma
	\end{equation}
	where we use that $ A \otimes X $ is a retract of an $ \MU $-module to identify $ \Gamma^{*} \nu(A \otimes X) $ with $ \tau_{\geq *}(A \otimes X) $.
\end{observation}

\begin{remark}    
	If $f \colon A \to B$ is a map of complex orientable $ \E_{1} $-ring spectra, then the induced map $\nu(A) \to \nu(B)$ of $ \E_{1} $-algebras in synthetic spectra gives rise to a natural transformation 
	\begin{equation*}
		\Re_{A}(-) \to \Re_{B}(-) \period
	\end{equation*}
	This natural transformation is adjoint to a comparison morphism, which we denote by 
	\begin{equation*}
		\Re(f) \colon \tau_{\geq *} B \tensorlimits_{\tau_{\geq *} A} \Re_{A}(-) \longrightarrow \Re_{B}(-) \period
	\end{equation*}
	In fact, $\Re(f)$ is an equivalence, as for $X \in \SynMU$ it can be identified with the canonical map 
	\begin{equation*}
		\nu(B) \tensorlimits_{\nu(A)} \nu(A) \otimes \nu(X) \to \nu(B) \otimes \nu(X) \period
	\end{equation*}
\end{remark}

As a consequence of \cref{cor:functors_out_of_SH_in_terms_of_Pure}, we can make the following definition. 

\begin{definition}
	\label{definition:comparison_map_between_synthetic_and_a_linear_realizations}
	We write 
	\begin{equation*}
		\phi_{A} \colon \Re_{A}(\Besyn(-)) \to \WBe(-;A)
	\end{equation*}
	for the unique natural transformation of colimit-preserving functors 
	\begin{equation*}
		\SH(\CC) \to \Modpost{A}
	\end{equation*}
	such that for every perfect pure $S \in \Pure(\CC)$ it can be identified with the map 
	\begin{equation*}
		\Re_{A}(\Besyn(S)) \simeq \Re_{A}(\nu(\Be(S)) \to \tau_{\geq *}(A \otimes \Be(S)) \simeq \WBe(S;A)
	\end{equation*}
	of \eqref{equation:canonical_map_from_a_linear_realization_of_nu}.
\end{definition}

For the following result, recall that a complex oriented ring spectrum $ A $ is said to be \textit{Landweber exact} if the map $\Spec(A_{*}) \to \Mfg$ classiying the Quillen formal group is flat. 
For example, this is true if $\pi_{*}(A)$ is a rational vector space. 

\begin{theorem}\label{thm:synthetic_realization_refines_filtered_realization_for_Landwebder_exact_rings}
	Let $ X \in \SH(\CC) $ and let $ A $ be a complex oriented $ \E_1 $-ring.
	Assume that one of the following conditions holds:
	\begin{enumerate}
	    \item The motivic spectrum $ X $ is cellular.

	    \item The complex oriented $ \E_1 $-ring $ A $ is Landweber exact.
	\end{enumerate}
	Then the map
	\begin{equation*}
		\phi_{A} \colon \Re_{A}(\Besyn(X)) \to \WBe(X;A)
	\end{equation*}
	is an equivalence.
\end{theorem}

\begin{proof}
	Suppose first that $ X $ is cellular.
	Since both functors preserve colimits, it suffices to show that $ \phi_A $ is an equivalence for motivic spectra of the form 
	\begin{equation*}
		X \simeq \Sup^{2n, n} \simeq (\PP^{1})^{\otimes n} \period
	\end{equation*}
	Since $ \Sup^{2n, n} $ is perfect pure, we see that $\phi_{A}$ can be identified with the canonical map 
	\begin{equation*}
		\nu(A) \otimes \nu(\Be(\Sup^{2n, n})) \simeq \nu(A) \otimes \nu(\Sup^{n}) \to \nu(A \otimes \Sup^{n}) \simeq \nu(A \otimes \Be(\Sup^{2n, n})) \period
	\end{equation*}
	Since $\Sup^{n}$ is $ \MU $-finite projective, \cite[Lemma 4.24]{pstrkagowski2022synthetic} implies that this map is an equivalence. 

	In the Landweber exact case, \cite[Propositions 2.12 \& 2.13]{hovey1999morava} shows that $ A $ is a filtered colimit of finite $ \MU $-projectives.
	The proof is now the same as the proof in the cellular case.
\end{proof}

\begin{remark}
	If $f \colon A \to B$ is a morphism of complex oriented $ \E_{1} $-rings, then the comparison map 
	\begin{equation*}
		c_{f} \colon \tau_{\geq *} B \tensorlimits_{\tau_{\geq *} A} \WBe(-;A) \to \WBe(-;B) 
	\end{equation*}
	of \cref{construction:comparison_morphism_between_filtrations_for_different_cpx_orientable_spectra} is compatible with the those of \cref{definition:comparison_map_between_synthetic_and_a_linear_realizations} in the sense that we have a commutative diagram 
	\begin{equation*}
		\begin{tikzcd}[row sep=3em, column sep=10em]
			\displaystyle \tau_{\geq *}B \tensorlimits_{\tau_{\geq *}A}\Re_{A}(\Besyn(-)) \arrow[r, "\tau_{\geq *}B \tensorlimits\limits_{\tau_{\geq *}A} \phi_{A}"] \arrow[d, "\Re(f)"'] %
			& \displaystyle \tau_{\geq *} B \tensorlimits_{\tau_{\geq *} A} \WBe(-;A) \arrow[d, "c_f"] \\
			\Re_{B}(\Besyn(-)) \arrow[r, "\phi_B"'] & \WBe(-;B)
		\end{tikzcd}
	\end{equation*}
	of functors $\SynMU \to \Mod_{\tau_{\geq *} B}(\FilSp)$ and natural transformations.
	To see this, note that all these functors preserve colimits, and so to give such a square it is enough to define it on perfect pures. If $S \in \Pure(k)$, then the above square reduces to 
	\begin{equation*}
		\begin{tikzcd}[column sep=4em, row sep=3em]
			\displaystyle \nu(B) \tensorlimits_{\nu(A)} \nu(A) \otimes \nu(\Be(S)) \arrow[r] \arrow[d] & \displaystyle \nu(B) \tensorlimits_{\nu(A)} \nu(B \otimes \Be(S)) \arrow[d] \\
			\displaystyle \nu(B) \tensorlimits_{A} \nu(\Be(S)) \arrow[r] & \nu(B \otimes \Be(S)) \period
		\end{tikzcd} 
	\end{equation*}
\end{remark}


\subsection{Synthetic real Betti realization and synthetic étale realization} 
\label{subsection:ideas_on_synthetic_real_and_etale_realizations}

In this section, highly inspired by the work of Burklund--Hahn--Senger on the \category of Artin--Tate real motivic spectra \cite{burklund2020galois}, we give conjectural description of a synthetic lift of a general motivic realization functor, such as étale realization. 

We first describe the main difference which makes the general case more interesting than the complex one. Notice that the complex Betti realization is valued in the \category of spectra, and the synthetic lift of \cref{theorem:existence_of_complex_synthetic_betti_homology} shows that it can be naturally lifted to the \category of synthetic spectra, which was constructed previously in \cite{pstrkagowski2022synthetic}. However, in both the case of the real Betti realization 
\begin{equation*}
	\Be_{\Cup_{2}} \colon \SH(\RR) \to \spectra^{\Cup_{2}}
\end{equation*}
and the étale realization 
\begin{equation*}
	\Re_{\ell} \colon \SH(k) \to \Shethyp(\Et_k;\Sp)\ellcomp \comma
\end{equation*}
the target is spectra equipped with additional structure (either that of a genuine $\Cup_{2}$-spectrum or, informally, a continuous action of the absolute Galois group $\Gal(\overline{k}/k)$).
Hence it is natural to expect that the synthetic lift of these realizations would not be valued in ordinary synthetic spectra, but rather in a synthetic deformation which takes this additional structure into account.
Thus, before one can discuss the existence of a lift, one first has to construct an appropriately structured deformation.
We propose a candidate for such a deformation in \cref{definition:even_synthetic_deformation_in_general_case}.
As an invitation towards further research in this direction, we also make conjectures on its structure.  

We can treat both Betti and étale realizations uniformly by introducing the following notion: 

\begin{definition}
	\label{definition:abstract_realization_functor}
	An \emph{abstract realization functor} over a field $ k $ is a symmetric monoidal, colimit-preserving functor 
	\begin{equation*}
		\Re \colon \SH(k) \to \Ccat
	\end{equation*}
	valued in presentably symmetric monoidal, stable \category. 
\end{definition}

\begin{warning}
	In the generality of \cref{definition:abstract_realization_functor}, one should probably only expect a functor valued in an \emph{even} synthetic category, so that the formalism of this section applied to $\Be \colon \SH(\CC) \to \spectra$ does not recover exactly the construction of \cref{subsection:synthetic_complex_betti_realization}, but rather only its even variant. 
	We believe that the existence of a non-even extension of the synthetic deformation is special to the case of the complex Betti realization.
	The reason is that it relies on the existence of a tensor square root of the Tate motive $\Be(\PP^{1}) \simeq \Sup^{2} \in \spectra$; this is not true in either the real or étale contexts. 
\end{warning} 

We now fix an abstract realization functor $\Re \colon \SH(k) \to \Ccat$.
To motivate the following definition, we recall \cite[Proposition 3.6]{pstrkagowski2023perfect}.
Write $\Perf(\Sp)_{\ev} \subseteq \Sp $ for the \category finite spectra with an even cell decomposition.
An ($ \MU $-based) \emph{even synthetic spectrum} can be identified with an additive sheaf
\begin{equation*}
	X \colon (\Perf(\Sp)_{\ev})^{\op} \to \spectra
\end{equation*}
on $\Perf(\Sp)_{\ev}$ with respect to the topology of $\MU_{*}$-epimorphisms (equivalently, with respect to the topology where coverings are maps whose fiber is again even). 
Since the even cells can be identified as $\Sup^{2k} \simeq \Be((\PP^{1})^{\otimes k})$, this suggests the following notions.

\begin{definition}
	Let $ \Re \colon \SH(k) \to \Ccat $ be an abstract realization functor.
	The \defn{$ \Ccat $-Tate motive} is
	\begin{equation*}
		L_{\Ccat} \colonequals \Re(\PP^{1}) \period
	\end{equation*} 
\end{definition}

\begin{definition}
	\label{definition:perfect_even_object_of_deformation}
	Let $ \Re \colon \SH(k) \to \Ccat $ be an abstract realization functor.
	We say that an object of $ \Ccat $ is \emph{perfect even} if it belongs to the smallest subcategory 
	\begin{equation*}
	\Perf(\Ccat)_{\ev} \subseteq \Perf(\Ccat)
	\end{equation*}
	containing $L_{\Ccat} ^{\otimes n}$ for all $n \in \ZZ$ and closed under retracts and extensions. 
\end{definition}

\begin{nul}
	Since $\Re(\MGL)$ is a filtered colimit of perfect evens, arguing as in \cref{proposition:pure_epimorphisms_detected_by_mgl}, one shows that the following two conditions are equivalent for a map $f \colon c \to d$ between perfect evens of \cref{definition:perfect_even_object_of_deformation}:
	\begin{enumerate}
	    \item $\fib(f) \in \Ccat$ is perfect even.

	    \item $\Re(\MGL) \otimes c \to \Re(\MGL) \otimes d$ admits a section. 
	\end{enumerate}
	We say that a map $ f $ of perfect evens is an \emph{even epimorphism} if $ f $ satisfies these two equivalent conditions.
\end{nul}

\begin{definition}
	\label{definition:even_synthetic_deformation_in_general_case}
	Let $ \Re \colon \SH(k) \to \Ccat $ be an abstract realization functor.
	The \emph{even synthetic deformation} of $ \Ccat $ is the \category
	\begin{equation*}
		\Syn^{\ev}(\Ccat) \colonequals \ShSigma(\Perf(\Ccat)_{\ev}; \Ccat)
	\end{equation*}
	of $ \Ccat $-valued additive sheaves with respect to the even epimorphism topology. 
\end{definition}

\begin{remark}
	\label{remark:generic_fiber_of_the_general_synthetic_category}
	Since the inclusion $\Perf(\Ccat)_{\ev} \hookrightarrow \Ccat$ preserves cofiber sequences, its left Kan extension gives a localization functor
	\begin{equation*}
		\Syn^{\ev}(\Ccat) \to \Ccat \period
	\end{equation*}
	This localization should be informally thought of as expressing the target as the generic fiber of the source. 
\end{remark}

\begin{remark}
	As \acategory, the even synthetic deformation depends only on $ \Ccat $ and on the invertible object $L_{\Ccat}$. 
	However, to define the synthetic analogue functor $\nu \colon \Ccat \to \Syn^{\ev}(\Ccat)$ we use more information about the functor $ \Re $. 
\end{remark}

Recall that in the classical case, the synthetic analogue $\nu \colon \spectra \hookrightarrow \SynMUev$ is given by the spectral Yoneda embedding followed by taking connective covers. 
The work of Burklund--Hahn--Senger suggests that in the general case, the right replacement for connectivity of spectra is that of effectivity. 

\begin{definition}
	Let $ \Re \colon \SH(k) \to \Ccat $ be an abstract realization functor.
	We say that an object $ c \in \Ccat $ is \emph{effective} if $ c $ belongs to the smallest subcategory 
	\begin{equation*}
		\Ccat^{\eff} \subseteq \Ccat
	\end{equation*}
	which contains $ \Re(\Sigma^{-n} \Sigma_{+}^{\infty} X)$ for $X \in \Sm_{k}$ and $ n \geq 0 $ and is closed under colimits.
	For an integer $ q \in \ZZ $, we say that an object $ c \in \Ccat $ is \emph{$q$-effective} if $ c $ belongs to the smallest subcategory 
	\begin{equation*}
		\Ccat^{\eff}(q) \subseteq \Ccat
	\end{equation*}
	which contains $L_{\Ccat}^{\otimes q} \otimes E $ for $ E \in \Ccal^{\eff} $ effective.
\end{definition}

\begin{nul}
	By construction $\Ccat^{\eff}(q)$ is presentable and the inclusion $\Ccat^{\eff}(q) \subseteq \Ccal $ admits a right adjoint 
	\begin{equation*}
		\f_{q} \colon \Ccat \to \Ccat^{\eff}(q) 
	\end{equation*}
	which we call the $q$-th \emph{effective cover}. 
	As
	\begin{equation*}
		\Ccat^{\eff}(q+1) \subseteq \Ccat^{\eff}(q) \comma
	\end{equation*}
	we have canonical natural transformations $\f_{q+1}(-) \to \f_{q}(-)$ which assemble into the \emph{slice tower} 
	\begin{equation}
		\label{equation:slice_tower}
		\cdots \to \f_{q+1}(-) \to \f_{q}(-) \to \f_{q-1}(-) \to \cdots \period
	\end{equation}
\end{nul}

\begin{remark}
	\label{remark:commuting_slices_past_tensoring_with_tate}
	Since $L_{\Ccat} \otimes \Ccat^{\eff}(q) = \Ccat^{\eff}(q+1)$, we have that for any $c \in \Ccat$, we have 
	\begin{equation*}
		L_{\Ccat} \otimes \f_{q}(c) \simeq \f_{q+1}(L_{\Ccat} \otimes c) \period
	\end{equation*}
	Informally, if we think of the slice tower as the variant of the Postnikov tower, tensoring with the Tate motive plays the role of the suspension. 
\end{remark}

\begin{definition}
	\label{definition:general_synthetic_analogue}
	If $c \in \Ccat$, the \emph{synthetic analogue} $\nu(c) \in \Syn^{\ev}(\Ccat)$ is given by the sheafication of the presheaf 
	\begin{equation*}
		\f_{0} \Map_{\Ccat}(-, c) \colon \Perf(\Ccat)_{\ev}^{\op} \to \Ccat \comma
	\end{equation*}
	where $\Map_{\Ccat}$ is the internal mapping object of $ \Ccat $. 
\end{definition}

\noindent The existence of the synthetic lift of $ \Re $ relies on the following conjecture. 

\begin{conjecture}
	\label{conjecture:general_synthetic_analogue_preserves_mgl_split_cofiber_sequences}
	The functor $\nu \colon \Ccat \to \Syn^{\ev}(\Ccat)$ preserves $\Re(\MGL)$-split cofiber sequences.
\end{conjecture}

\begin{remark}
	Notice that if \cref{conjecture:general_synthetic_analogue_preserves_mgl_split_cofiber_sequences} holds, then using \cref{cor:functors_out_of_SH_in_terms_of_Pure} we can define the synthetic lift 
	\begin{equation*}
		\Re^{\syn} \colon \SH(k) \to \Syn^{\ev}(\Ccat)
	\end{equation*}
	as the unique colimit-preserving functor such that 
	\begin{equation*}
		\Re^{\syn}(S) \simeq \nu(\Re(S))
	\end{equation*}
	for any perfect pure $ S $. 
\end{remark}

As we mentioned at the beginning of this section, our approach is inspired by the work of Burklund--Hanh--Senger, who instead of additive sheaves work with filtered objects. 
We now explain how the synthetic deformation presented here should conjecturally be related to the filtered object perspective of \cite{burklund2020galois}. 

Using \cref{remark:commuting_slices_past_tensoring_with_tate}, the slice tower of $\Map(-, L^{\otimes 0})$ induces a filtered object 
\begin{equation*}
	\nu(L^{\otimes *}) \in \Fil(\Syn^{\ev}(\Ccat))
\end{equation*}
of the form 
\begin{equation*}
	\cdots \to \nu(L^{\otimes 1}) \to \nu(L^{\otimes 0}) \to \nu(L^{\otimes -1}) \to \cdots \period
\end{equation*}
This object the slice analogue of the Postnikov tower of \cite[\S 5.2]{patchkoria2021adams}, which is shown in op. cit. to encode the Adams filtration. 
We conjecture that $\nu(L^{\otimes *})$ has an analogous property in the context of motivic realizations, at least for Artin--Tate objects. 

\begin{conjecture}
	We have that: 
	\begin{enumerate}
	    \item The internal mapping object functor 
	    \begin{equation*}
			\Map_{\Ccat}(\nu(L^{\otimes *}), -) \colon \Syn^{\ev}(\Ccat) \to \Fil(\Ccat)
	    \end{equation*}
	    can be promoted to an equivalence 
	    \begin{equation*}
	     	\Syn^{\ev}(\Ccat) \simeq \Mod_{\fil^{*}(L^{\otimes 0})}(\Fil(\Ccat))
	    \end{equation*}
	    between the synthetic deformation and modules over
	    \begin{equation*}
	    	\fil^{*}(L^{\otimes 0}) \simeq \End_{\Ccat}(\nu(L^{\otimes *}), \nu(L^{\otimes *})) \period
	    \end{equation*}

	    \item Through the equivalence of (1), the synthetic realization functor
	    \begin{equation*}
	    	\Re^{\syn} \colon \SH(k) \to \Fil(\Ccat)
	    \end{equation*}
	    can be identified on Artin--Tate objects with the functor sending $X \in \SH^{\AT}(k)$ to 
	    \begin{equation*}
	   		\cdots \to \Re(\f_{1} X) \to \Re(\f_{0}X) \to \Re(\f_{-1} X) \to \cdots \comma
	    \end{equation*}
	    the realization of its tower of effective covers (see \Cref{recollection:slice_filtration}). 
	\end{enumerate}
\end{conjecture}


\printbibliography

\end{document}

%% file: preamble.tex
\usepackage[T1]{fontenc}
\usepackage[utf8]{inputenc}
\usepackage[english]{babel}
\usepackage{inconsolata} 
\usepackage[dvipsnames]{xcolor}

\usepackage{url} 
\usepackage{hyperref}
\hypersetup{
  colorlinks = true, 
  linktoc = all,     
  linktocpage = true,
  linkcolor = RoyalBlue, 
  urlcolor = RoyalBlue, 
  citecolor = OliveGreen, 
  pdfencoding = unicode,
  hypertexnames = true,
  bookmarksdepth = subsubsection,
}

\usepackage{mathtools}

\usepackage{amssymb}
\usepackage{amsmath}
\usepackage{units}

\usepackage{eucal}

\usepackage{todonotes}
\usepackage{faktor}

\swapnumbers 

\usepackage[shortlabels]{enumitem}
\setlist[enumerate]{
     itemsep  = 0.2cm,
       label  = {\upshape (\arabic*)},
         ref  = \arabic*,
   leftmargin = *,
  labelindent = 0.2cm,
}

\makeatletter
\def\@tocline#1#2#3#4#5#6#7{\relax
  \ifnum #1>\c@tocdepth 
  \else
    \par \addpenalty\@secpenalty\addvspace{#2}%
    \begingroup \hyphenpenalty\@M
    \@ifempty{#4}{%
      \@tempdima\csname r@tocindent\number#1\endcsname\relax
    }{%
      \@tempdima#4\relax
    }%
    \parindent\z@ \leftskip#3\relax \advance\leftskip\@tempdima\relax
    \rightskip\@pnumwidth plus4em \parfillskip-\@pnumwidth
    #5\leavevmode\hskip-\@tempdima
      \ifcase #1
       \or\or \hskip 1em \or \hskip 2em \else \hskip 3em \fi%
      #6\nobreak\relax
    \hfill\hbox to\@pnumwidth{\@tocpagenum{#7}}\par
    \nobreak
    \endgroup
  \fi}
\makeatother

\usepackage{tikz}
\usetikzlibrary{arrows,backgrounds,
    decorations.pathreplacing,
    decorations.pathmorphing
}
\usetikzlibrary{matrix,arrows}
\usepackage{tikz-cd}

\usepackage[style=ieee,citestyle=numeric,backend=biber,sorting=nyt,url=false]{biblatex}

\AtEveryBibitem{\clearfield{issn}}  
\AtEveryCitekey{\clearfield{issn}}
\AtBeginBibliography{\small}    

\renewcommand*{\bibnamedash}{\underline{\hspace{3em}}\kern 0.1em} 

\DeclareFieldFormat{postnote}{#1}

\DeclareFieldFormat[article,inbook,incollection,inproceedings,patent,thesis,unpublished]{title}{\textit{#1\isdot}} 

\DeclareFieldFormat{journaltitle}{#1}  

\DeclareFieldFormat[book,inbook,incollection,inproceedings]{series}{#1} 

\defbibheading{references}[\refname]{%
\section*{#1}%
\addcontentsline{toc}{section}{References} %
\markboth{#1}{#1}
}

\DeclareSourcemap{
  \maps[datatype=bibtex]{
    \map{
      \step[fieldsource=note, final]
      \step[fieldset=addendum, origfieldval, final]
      \step[fieldset=note, null]
    }
  }
}

\usepackage[nameinlink,capitalise,noabbrev]{cleveref} 

\theoremstyle{plain}

\numberwithin{equation}{subsection}

\newtheorem{theorem}[equation]{Theorem}
\newtheorem{lemma}[equation]{Lemma}
\newtheorem{proposition}[equation]{Proposition}
\newtheorem{corollary}[equation]{Corollary}
\newtheorem{conjecture}[equation]{Conjecture}

\theoremstyle{definition}

\newtheorem{construction}[equation]{Construction}
\newtheorem{definition}[equation]{Definition}
\newtheorem{example}[equation]{Example}
\newtheorem{notation}[equation]{Notation}
\newtheorem{nul}[equation]{}
\newtheorem{recollection}[equation]{Recollection}
\newtheorem{remark}[equation]{Remark}
\newtheorem{warning}[equation]{Warning}

\newtheorem*{remark*}{Remark}
\newtheorem*{terminology*}{Terminology}
\newtheorem*{interpretation*}{Interpretation}
\newtheorem*{definition*}{Definition}
\newtheorem*{conjecture*}{Conjecture}
\newtheorem*{notation*}{Notation}
\newtheorem*{convention*}{Convention}

\theoremstyle{remark}

\crefname{section}{\S\!}{\S\S\!} 
\Crefname{section}{Section}{Sections}
\crefname{subsection}{\S\!}{\S\S\!} 
\Crefname{subsection}{Subsection}{Subsections}

\crefformat{nul}{(#2#1#3)} 
\crefname{nul}{}{}
\Crefformat{nul}{(#2#1#3)}
\Crefname{nul}{}{}

\crefname{axiom}{Axiom}{Axioms}
\Crefname{axiom}{Axiom}{Axioms}
\crefname{convention}{Convention}{Conventions}
\Crefname{convention}{Convention}{Conventions}
\crefname{conjecture}{Conjecture}{Conjectures}
\Crefname{conjecture}{Conjecture}{Conjectures}
\crefname{construction}{Construction}{Constructions}
\Crefname{construction}{Construction}{Constructions}
\crefname{corollary}{Corollary}{Corollaries}
\Crefname{corollary}{Corollary}{Corollaries}
\crefname{counterexample}{Counterexample}{Counterexamples}
\Crefname{counterexample}{Counterexample}{Counterexamples}
\crefname{definition}{Definition}{Definitions}
\Crefname{definition}{Definition}{Definitions}
\crefname{example}{Example}{Examples}
\Crefname{example}{Example}{Examples}
\crefname{exercise}{Exercise}{Exercises}
\Crefname{exercise}{Exercise}{Exercises}
\crefname{lemma}{Lemma}{Lemmas} 
\Crefname{lemma}{Lemma}{Lemmas} 
\crefname{notation}{Notation}{Notations}
\Crefname{notation}{Notation}{Notations}
\crefname{observation}{Observation}{Observations}
\Crefname{observation}{Observation}{Observations}
\crefname{proposition}{Proposition}{Propositions}
\Crefname{proposition}{Proposition}{Propositions}
\crefname{question}{Question}{Questions}
\Crefname{question}{Question}{Questions}
\crefname{recollection}{Recollection}{Recollections}
\Crefname{recollection}{Recollection}{Recollections}
\crefname{remark}{Remark}{Remarks}
\Crefname{remark}{Remark}{Remarks}
\crefname{subexample}{Subexample}{Subexamples}
\Crefname{subexample}{Subexample}{Subexamples}
\crefname{theorem}{Theorem}{Theorems}
\Crefname{theorem}{Theorem}{Theorems}
\crefname{variant}{Variant}{Variants}
\Crefname{variant}{Variant}{Variants}
\crefname{warning}{Warning}{Warnings}
\Crefname{warning}{Warning}{Warnings}

\input{piotr_macros}

\input{peter_macros}

%% file: piotr_macros.tex
\newcommand{\euscr}[1]{\EuScript{#1}} 
\newcommand{\spectra}{\euscr{S}p} 
\newcommand{\Ccat}{\euscr{C}} 
\newcommand{\Dcat}{\euscr{D}} 

\newcommand{\motiviccoh}{\mathrm{M}}

\newcommand{\motive}{\mathrm{M}_{\c}}

\newcommand{\fieldp}{\FF_{p}} 

\DeclareMathOperator{\Spec}{Spec}

\DeclareMathOperator{\Hom}{Hom}
\DeclareMathOperator{\End}{End}

\DeclareMathOperator{\Mod}{Mod}

\DeclareMathOperator{\CAlg}{CAlg}
\DeclareMathOperator{\Alg}{Alg}

\newcommand{\colim}{\operatornamewithlimits{colim}}

\DeclareMathOperator{\Fun}{Fun}

\DeclareMathOperator{\op}{op}
\DeclareMathOperator{\fg}{fg}

\DeclareMathOperator{\Sch}{Sch}

\DeclareMathOperator{\Map}{Map}
\DeclareMathOperator{\map}{map}

\DeclareMathOperator{\id}{id}
\DeclareMathOperator{\Sp}{Sp}
\DeclareMathOperator{\Ab}{Ab}

\DeclareMathOperator{\gr}{gr}
\DeclareMathOperator{\Fil}{Fil}

\DeclareMathOperator{\Ind}{Ind}

\DeclareMathOperator{\et}{\acute{e}t}
\DeclareMathOperator{\Syn}{Syn}
\DeclareMathOperator{\fil}{fil}

\newcommand{\SH}{\mathcal{SH}}
\newcommand{\perfectpures}{\mathrm{Pure}(k)}

\newcommand{\MGL}{\mathrm{MGL}}
\newcommand{\Hrm}{\mathrm{H}}
\newcommand{\Crm}{\mathrm{C}}
\DeclareMathOperator{\Perf}{Perf}
\DeclareMathOperator{\MU}{MU}
\DeclareMathOperator{\KU}{KU}
\DeclareMathOperator{\Th}{Th}

\newcommand{\M}{\mathbb{M}}
\newcommand{\WM}{W\mathbb{M}}
\newcommand{\sphere}{\mathbb{S}}

%% file: peter_macros.tex
\theoremstyle{definition}
\newtheorem{observation}[equation]{Observation}
\newtheorem{convention}[equation]{Convention}

\crefformat{nul}{(#2#1#3)} 
\crefname{nul}{}{}
\Crefformat{nul}{(#2#1#3)}
\Crefname{nul}{}{}

\usepackage{colonequals}
\usepackage{xspace}

\newcommand{\ldh}{$\ell$dh}
\newcommand{\cdh}{cdh}
\newcommand{\cdp}{cdp}
\newcommand{\fpsl}{fps$\ell'$}

\newcommand{\cdpcovering}{\xspace{\cdp-cover\-ing}\xspace}

\newcommand{\cdptopology}{\xspace{\cdp-topol\-o\-gy}\xspace}

\newcommand{\cdhcover}{\xspace{\cdh-cover}\xspace}

\newcommand{\cdhhypercover}{\xspace{\cdh-hy\-per\-cover}\xspace}

\newcommand{\cdhhypersheaf}{\xspace{\cdh-hy\-per\-sheaf}\xspace}
\newcommand{\cdhdescent}{\xspace{\cdh-de\-scent}\xspace}
\newcommand{\cdhsheaf}{\xspace{\cdh-sheaf}\xspace}

\newcommand{\cdhtopology}{\xspace{\cdh-topol\-o\-gy}\xspace}
\newcommand{\cdhtopologies}{\xspace{\cdh-topolo\-gies}\xspace}

\newcommand{\ldhcover}{\xspace{\ldh-cover}\xspace}

\newcommand{\ldhhypercover}{\xspace{\ldh-hy\-per\-cover}\xspace}
\newcommand{\ldhhypercovers}{\xspace{\ldh-hy\-per\-covers}\xspace}

\newcommand{\ldhhyperdescent}{\xspace{\ldh-hy\-per\-de\-scent}\xspace}

\newcommand{\ldhdescent}{\xspace{\ldh-de\-scent}\xspace}

\newcommand{\ldhtopology}{\xspace{\ldh-topol\-o\-gy}\xspace}

\newcommand{\fpslcover}{\xspace{\fpsl-cover}\xspace}
\newcommand{\fpslcovers}{\xspace{\fpsl-covers}\xspace}

\renewcommand{\sphere}{\Sup^{0}}

\renewcommand{\varinjlim}{\colim}
\renewcommand{\varprojlim}{\lim}

\newcommand{\Xbar}{\overline{X}}

\newcommand{\paren}[1]{\left(#1 \right)}

\newcommand{\einv}{\nicefrac{1}{e}}

\newcommand{\Et}{\mathrm{\acute{E}t}}
\newcommand{\Sm}{\mathrm{Sm}}
\newcommand{\Spc}{\mathrm{Spc}}

\renewcommand{\spectra}{\Sp}

\newcommand{\hSp}{\mathrm{hSp}}

\newcommand{\FilSp}{\mathrm{FilSp}}

\newcommand{\Dfil}{\Dcal^{\fil}}

\newcommand{\Pure}{\mathrm{Pure}}

\DeclareMathOperator{\eff}{eff}
\DeclareMathOperator{\AT}{AT}

\newcommand{\post}[1]{\tau_{\geq *}(#1)}
\newcommand{\Modpost}[1]{\Mod_{\post{#1}}(\FilSp)}

\newcommand{\DDelta}{\Delta}
\newcommand{\Deltaop}{\DDelta^{\op}}

\newcommand{\Unit}{\mathbf{1}}

\renewcommand{\Mc}{\motive}

\DeclareMathOperator{\characteristic}{char}
\DeclareMathOperator{\cofib}{cofib}
\DeclareMathOperator{\hyp}{hyp}

\DeclareMathOperator{\syn}{syn}

\newcommand{\disc}{\mathrm{disc}}

\renewcommand{\MU}{\mathrm{MU}}
\newcommand{\Be}{\mathrm{Be}}

\newcommand{\WBe}{\Wup_{\ast}\Be}

\renewcommand{\Re}{\mathrm{Re}}
\newcommand{\WRe}{\Wup_{\ast}\Re}

\newcommand{\Gal}{\mathrm{Gal}}
\DeclareMathOperator{\Sh}{Sh}
\DeclareMathOperator{\PSh}{PSh}

\DeclareMathOperator{\Stable}{Stable}

\newcommand{\SynMU}{\Syn_{\MU}}
\newcommand{\SynMUev}{\SynMU^{\ev}}

\newcommand{\PSigma}{\PSh_{\Sigma}}
\newcommand{\ShSigma}{\Sh_{\Sigma}}

\newcommand{\Shethyp}{\Sh_{\et}^{\hyp}}

\newcommand{\Let}{\mathrm{L}_{\et}}

\newcommand{\ellcomp}{_{\ell}^{\wedge}}

\newcommand{\directsum}{\oplus}
\newcommand{\tensor}{\otimes}
\newcommand{\of}{\circ}
\newcommand{\tensorlimits}{\operatornamewithlimits{\tensor}} 
\newcommand{\cross}{\times}

\newcommand{\fib}{\mathrm{fib}}

\newcommand{\FunL}{\Fun^{\mathrm{L}}}

\newcommand{\Cc}{\Cup_{\mathrm{c}}}

\newcommand{\Hc}{\Hup_{\mathrm{c}}}

\DeclareMathOperator{\fp}{fp}

\DeclareMathOperator{\ev}{ev}

\newcommand{\filev}{\fil_{\ev}}

\newcommand{\Besyn}{\Be_{\syn}}

\newcommand{\equivalent}{\simeq}

\DeclareMathOperator{\Gr}{Gr}

\DeclareMathOperator{\BM}{BM}

\newcommand{\MGLBM}{\MGL^{\BM}}

\newcommand{\HR}{\Hup R}

\newcommand{\MZZ}{\Mup\ZZ}
\renewcommand{\MR}{\Mup R}

\DeclareMathOperator{\Ex}{Ex}
\DeclareMathOperator{\cell}{cell}

\newcommand{\restrict}[2]{{#1}|_{#2}}

\newcommand{\Mfg}{\mathcal{M}_{\fg}}
\newcommand{\IndCoh}{\mathrm{IndCoh}}

\newcommand{\Chow}{\mathrm{Chow}}
\newcommand{\DM}{\mathrm{DM}}

\renewcommand{\SH}{\mathrm{SH}}

\newcommand{\cosk}{\mathrm{cosk}}

\newcommand{\wheart}{\mathrm{w}\heartsuit}

\newcommand{\w}{\mathrm{w}}

\DeclareMathOperator{\GS}{GS}

\newcommand{\WGS}{\Wup_{*}^{\GS}}

\renewcommand{\c}{\mathrm{c}}
\newcommand{\f}{\mathrm{f}}
\newcommand{\s}{\mathrm{s}}

\renewcommand{\Cup}{\mathrm{C}}
\newcommand{\Eup}{\mathrm{E}}
\newcommand{\Hup}{\mathrm{H}}
\newcommand{\Kup}{\mathrm{K}}
\newcommand{\Mup}{\mathrm{M}}

\newcommand{\Sup}{\mathrm{S}}
\newcommand{\Tup}{\mathrm{T}}
\newcommand{\Wup}{\mathrm{W}}

\newcommand{\WCstar}{\Wup_{*}\Cup_{*}}

\newcommand{\E}{\mathbf{E}}

\renewcommand{\AA}{\mathbb{A}}
\newcommand{\CC}{\mathbb{C}}
\newcommand{\FF}{\mathbb{F}}

\newcommand{\PP}{\mathbb{P}}
\newcommand{\QQ}{\mathbb{Q}}
\newcommand{\RR}{\mathbb{R}}

\newcommand{\ZZ}{\mathbb{Z}}

\newcommand{\Ccal}{\mathcal{C}}
\newcommand{\Dcal}{\mathcal{D}}
\newcommand{\Fcal}{\mathcal{F}}
\newcommand{\Ncal}{\mathcal{N}}
\newcommand{\Ocal}{\mathcal{O}}

\newcommand{\reflectedop}[2]{{\mathop{\reflectbox{$#1{#2}$}}}}
\newcommand{\noloc}{{\mathpalette\reflectedop\colon}}
\newcommand{\adjto}[4]{{#1}\colon{#2} \rightleftarrows {#3}\noloc{#4}}

\newcommand{\inclusion}{\hookrightarrow}
\newcommand{\equivalence}{\xrightarrow{\sim}}

\newcommand{\kbar}{\bar{k}}

\newcommand{\pbar}{\bar{p}}
\newcommand{\fbar}{\bar{f}}

\newcommand{\fupperstar}{f^{\ast}}
\newcommand{\flowerstar}{f_{\ast}}
\newcommand{\flowersharp}{f_{\sharp}}
\newcommand{\flowershriek}{f_{!}}
\newcommand{\fuppershriek}{f^{!}}

\newcommand{\iupperstar}{i^{\ast}}
\newcommand{\ilowerstar}{i_{\ast}}

\newcommand{\ilowershriek}{i_{!}}
\newcommand{\iuppershriek}{i^{!}}

\newcommand{\jupperstar}{j^{\ast}}
\newcommand{\jlowerstar}{j_{\ast}}
\newcommand{\jlowersharp}{j_{\sharp}}
\newcommand{\jlowershriek}{j_{!}}
\newcommand{\juppershriek}{j^{!}}

\newcommand{\pupperstar}{p^{\ast}}
\newcommand{\plowerstar}{p_{\ast}}

\newcommand{\plowershriek}{p_{!}}
\newcommand{\puppershriek}{p^{!}}

\newcommand{\andeq}{\text{\qquad and\qquad}}

\newcommand{\comma}{\rlap{\ ,}}
\newcommand{\period}{\rlap{\ .}}

\newcommand{\defn}{\emph}

\newcommand{\category}{\xspace{$\infty$-cat\-e\-go\-ry}\xspace}
\newcommand{\categories}{\xspace{$\infty$-cat\-e\-gories}\xspace}
\newcommand{\categorical}{\xspace{$\infty$-cat\-e\-gor\-i\-cal}\xspace}

\newcommand{\acategory}{{an $\infty$-cat\-e\-go\-ry}\xspace}

\newcommand{\acategorical}{{an $\infty$-cat\-e\-gor\-i\-cal}\xspace}

\DeclareMathOperator{\tup}{t}
\newcommand{\tstructure}{\xspace{$\tup$-struc\-ture}\xspace}
\newcommand{\tstructures}{\xspace{$\tup$-struc\-tures}\xspace}
\newcommand{\texact}{\xspace{$\tup$-exact}\xspace}

\newcommand{\topos}{\xspace{$\infty$-topos}\xspace}